\documentclass[12pt]{amsart}

\usepackage{amssymb}
\usepackage{amsmath, graphicx, rotating}
\usepackage{color}
\usepackage{soul}
\usepackage[dvipsnames]{xcolor}

\usepackage[T1]{fontenc}
\usepackage{lmodern}
\usepackage[english]{babel}

\usepackage{ upgreek }
\usepackage{stmaryrd}
\SetSymbolFont{stmry}{bold}{U}{stmry}{m}{n}
\usepackage{amsthm}
\usepackage{float}

\usepackage{ bbm }
\usepackage{ stmaryrd }
\usepackage{ mathrsfs }
\usepackage{ frcursive }
\usepackage{ comment }

\usepackage{pgf, tikz}
\usetikzlibrary{shapes}
\usepackage{varioref}
\usepackage{enumitem}
\usepackage{todonotes}

\definecolor{rouge}{rgb}{0.7,0.00,0.00}
\definecolor{vert}{rgb}{0.00,0.5,0.00}
\definecolor{bleu}{rgb}{0.00,0.00,0.8}
\usepackage[margin=1in]{geometry}
\newtheorem{theorem}{Theorem}[section]
\newtheorem*{theorem*}{Theorem}
\newtheorem{lemma}[theorem]{Lemma}

\newtheorem{corollary}[theorem]{Corollary}
\newtheorem{proposition}[theorem]{Proposition}
\newtheorem{hypothesis}{Hypothesis M\kern-0.1mm}
\labelformat{hypothesis}{\textbf{M\kern-0.1mm#1}}

\theoremstyle{definition}

\usepackage[colorlinks=true, linkcolor=blue, citecolor=blue]{hyperref}

\numberwithin{equation}{section}

\newcommand*{\abs}[1]{\left\lvert#1\right\rvert}
\newcommand*{\norm}[1]{\left\lVert#1\right\rVert}

\newcommand*{\pent}[1]{\left\lfloor#1\right\rfloor}
\newcommand*{\sachant}[2]{\left.#1 \,\middle|\,#2\right.}

\def\bb#1{\mathbb{#1}}
\def\bs#1{\boldsymbol{#1}}
\def\bf#1{\mathbf{#1}}
\def\scr#1{\mathscr{#1}}
\def\bbm#1{\mathbbm{#1}}
\def\geq{\geqslant}
\def\leq{\leqslant}
\def\hat{\widehat}

\newcommand\ee{\varepsilon}

\DeclareMathOperator{\LL}{L}
\DeclareMathOperator{\dd}{d\!}
\DeclareMathOperator{\e}{e}

\DeclareMathOperator{\supp}{supp}

\begin{document}
\title[Conditioned local limit theorems]{Conditioned local limit theorems for random walks defined on finite Markov chains}

\author{Ion Grama}
\curraddr[Grama, I.]{ Universit\'{e} de Bretagne-Sud, LMBA UMR CNRS 6205,
Vannes, France}
\email{ion.grama@univ-ubs.fr}

\author{Ronan Lauvergnat}
\curraddr[Lauvergnat, R.]{Universit\'{e} de Bretagne-Sud, LMBA UMR CNRS 6205,
Vannes, France}
\email{ronan.lauvergnat@univ-ubs.fr}

\author{\'Emile Le Page}
\curraddr[Le Page, \'E.]{Universit\'{e} de Bretagne-Sud, LMBA UMR CNRS 6205,
Vannes, France}
\email{emile.le-page@univ-ubs.fr}

\date{\today}

\begin{abstract}
Let $(X_n)_{n\geq 0}$ be a Markov chain
with values in a finite state space $\bb X$ starting at $X_0=x \in \bb X$ and let $f$ be a real function defined on $\bb X$.
Set $S_n=\sum_{k=1}^{n} f(X_k)$, $n\geq 1$. For any $y \in \bb R$ denote by $\tau_y$  the first time when $y+S_n$ becomes non-positive.
We study the asymptotic behaviour of the probability 
$\bb P_x \left( y+S_{n} \in [z,z+a] \,,\, \tau_y > n \right)$
as $n\to+\infty.$
We first establish for this probability a conditional version of the local limit theorem of Stone.
Then we find for it an asymptotic equivalent of order $n^{3/2}$ and give a generalization which is useful in applications.
We also describe the asymptotic behaviour of the probability $\bb P_x \left( \tau_y = n \right)$ as  $n\to+\infty$.
\end{abstract}

\maketitle

\section{Introduction}

Assume that on the probability space $(\Omega, \scr{F}, \bb P)$  we are given a sequence of real valued  
random variables  $( X_n )_{n\geq 1}$. Consider the random walk $S_n=\sum_{k=1}^{n}X_k$, $n\geq 1.$ 
Suppose first that $( X_n )_{n\geq 1}$ are independent identically 
distributed of zero mean and finite variance.
For any $y>0$ denote by $\tau_y$  the first time when $y+S_n$ becomes non-positive.
The study of the asymptotic behaviour of the probability $\mathbb P(\tau_y>n)$ and of the law of $y+S_n$ 
conditioned to stay positive (i.e.\ given the event  $\{\tau_y>n\}$) has been initiated by Spitzer  \cite{spitzer_principles_1964} and developed subsequently by 
Iglehart \cite{iglehart_functional_1974},   
Bolthausen \cite{bolthausen_functional_1976},
Doney \cite{doney_asymptotic_1989}, 
Bertoin and Doney \cite{bertoin_conditioning_1994},
Borovkov \cite{borovkov_asymptotic_2004,borovkov_asymptotic_2004-1}, to cite only a few. 
Important progress has been achieved by employing a new approach based on 
the existence of the harmonic function in Varopoulos \cite{varopoulos_potential_1999}, \cite{varopoulos_potential_2000}, 
Eichelbacher and K\"onig \cite{eichelsbacher_ordered_2008} and recently by   
Denisov and Wachtel \cite{denisov_conditional_2010, denisov_exit_2015, denisov_random_2015}.   
In this line Grama, Le Page and Peign\'e \cite{grama_conditioned_2016} and the authors in \cite{GLLP_affine_2016}, \cite{grama_limit_2016-1} have studied sums of functions defined on Markov chains under spectral gap assumptions.
The goal of the present paper is to complete these investigations  by establishing local limit theorems for random walks defined on finite Markov chains and conditioned to stay positive.  

Local limit theorems for the sum of independent random variables without conditioning have attracted much attention,  
since the pioneering work of Gnedenko \cite{gnedenko_local_1948} and Stone \cite{stone_local_1965}.
The first local limit theorem for a random walk conditioned to stay positive has been established in Iglehart \cite{iglehart_random_1974}
in the context of walks with negative drift $\mathbb E X_1 < 0$.
Caravenna \cite{caravenna_local_2005} 
studied conditioned local limit theorems for random variables in the domain of attraction of the normal law and 
Vatutin and Wachtel \cite{vatutin_local_2008}  for random variables $X_k$ in the domain of attraction of the stable law. 
Denisov and Wachtel \cite{denisov_random_2015} obtained a local limit theorem for random walks in $\mathbb Z^d$
conditioned to stay in a cone based on the harmonic function approach.
 
The ordinary and conditioned local limit theorems in the case of Markov chains 
are less studied in the literature.
Le Page \cite{page_theoremes_1982} stated a local limit theorem for products of random matrices and  
Guivarch and Hardy \cite{guivarch_theoremes_1988} have considered a local limit theorem for sums $S_n=\sum_{k=1}^{n} f(X_k)$, where $( X_n )_{n\geq 0}$ is a Markov chain 
under spectral gap assumptions and $f$ a real function defined on the state space of the chain.
In the conditional case we are aware only of the results of Presman \cite{presman_1967} and \cite{presman_1969} who has considered the case of finite Markov chains in a more general setting but which, because of rather stringent assumptions, does not cover the results of this paper. We note also the work of Le Page and Peign\'e \cite{le_page_local_1997} who have proved 
a conditioned local limit theorem for the stochastic recursion.

Let us briefly review  main results of the paper concerning conditioned local limit behaviour of the walk
$S_n=\sum_{k=1}^{n} f(X_k)$ defined on a finite Markov chain $( X_n )_{n\geq 0}$.
From more general statement of Theorem  \ref{GNLLT},
under the conditions that the underlying Markov chain is irreducible and aperiodic and that $( S_n )_{n\geq 0}$ is centred and non-lattice, for fixed $x\in \mathbb X$ and $y\in \mathbb R$, it follows that, uniformly in $z \geq 0,$
\begin{equation}
\label{introrez001}
\lim_{n\to \infty} 
\left(n \bb P_x \left( y+S_{n} \in [z,z+a] \,,\, \tau_y > n \right) - \frac{2a V(x,y)}{\sqrt{2\pi} \sigma^2} \varphi_+ \left( \frac{z}{\sqrt{n}\sigma} \right) \right) =0,
\end{equation}
where 
$\varphi_+(t) = t\e^{-\frac{t^2}{2}} \bbm 1_{\{t\geq 0\}}$ is the Rayleigh  density.
The relation \eqref{introrez001}
is an extension of the classical local limit theorem by Stone \cite{stone_local_1965}
to the case of Markov chains.
We refer to Caravenna \cite{caravenna_local_2005} and Vatutin and Wachtel \cite{vatutin_local_2008}, where the corresponding result
has been obtained for independent random variables in the domain of attraction of the normal law.

We note that while \eqref{introrez001} is consistent for large $z$, it is not informative for $z$ in a compact set.
A meaningful local limit behaviour for fixed values of $z$ 
can be obtained from our Theorem \ref{LLTC}. 
Under the same assumptions, for any fixed $x\in \mathbb X$, $y\in \mathbb R$ and  $z \geq 0,$
\begin{align}
\label{introrez002}
	\lim_{n\to +\infty} n^{3/2} &\bb P_x \left( y+S_{n} \in [z,z+a] \,,\, \tau_y > n \right) 
	 = \frac{2V(x,y)}{\sqrt{2\pi}\sigma^3} \int_z^{z+a} \int_{\bb X}  V^*\left( x', z' \right) \bs \nu (\dd x') \dd z'.
\end{align}
For sums of independent random variables similar limit behaviour was found in Vatutin and Wachtel \cite{vatutin_local_2008}. 
It should be noted that \eqref{introrez001} and \eqref{introrez002} complement each other: the main term in \eqref{introrez001}  is meaningful for large $z$ such that $z \sim n^{1/2} $ as $n\to\infty$,
while \eqref{introrez002} holds for $z$ in compact sets.

We also state extensions of \eqref{introrez001} and \eqref{introrez002} to the joint law of $X_n$ and $y+S_n$. 
These extensions are useful in applications,  in particular,  for determining 
the exact asymptotic behaviour of the survival time for branching processes in a Markovian environment.
They also allow us to infer the local limit behaviour of the exit time $\tau_y$ (see Theorem \ref{CAPEBIS}):
under the assumptions mentioned before, 
for any $x\in \bb X$ and $y \in \bb R$,
\[
\lim_{n\to+\infty} n^{3/2}\bb P_x \left( \tau_y = n \right) = \frac{2V(x,y)}{\sqrt{2\pi} \sigma^3} \int_0^{+\infty} \bb E_{\bs \nu}^* \left( V^*(X_1^*,z) \,;\, S_1^* \geq z \right) \dd z.
\]

The approach employed in this paper is different from that  in \cite{presman_1967}, \cite{presman_1969} and \cite{le_page_local_1997} which all
are based on Wiener-Hopf arguments. Our technique is close to that in Denisov and Wachtel \cite{denisov_random_2015},
however, in order to make it work for a random walk $S_n=\sum_{k=1}^{n} f(X_k)$ defined on a Markov chain $( X_n )_{n\geq 0}$, we have to overcome some essential difficulties. 
One of them is related to the problem of the reversibility of the Markov walk $(S_n )_{n\geq 0}$.
Let us explain this point in more details.
When $( X_n )_{n\geq 1}$ are $\mathbb Z$-valued independent identically distributed random variables,
let $( S_n^* )_{n\geq 1}$ be the reverse walk given by $S_n^* = \sum_{k=1}^n X_k^*$, 
where $( X_n^* )_{n\geq 1}$ is a sequence of independent identically distributed random variables of the same law as $-X_1$. 
Denote by $\tau^*_z$ the first time when $( z+S_k^* )_{k\geq 0}$ becomes non-positive.
Then,  due to exchangeability of the random variables $( X_n )_{n\geq 1}$, we have
\begin{equation}
\label{exceang001}
\mathbb P(y+S_n=z, \tau_y>n) = \mathbb P(z+S^*_n=y, \tau^*_z >n).
\end{equation}
This relation does not hold any more for the walk  $S_n=\sum_{k=1}^{n} f(X_k)$, 
where $( X_n )_{n\geq 0}$ is a Markov chain. 
Even though $( X_n )_{n\geq 0}$ takes values on a finite state space $\mathbb X$ and there exists a dual chain 
$( X^*_n )_{n\geq 0},$ 
the main difficulty is that the function $f:\mathbb X \mapsto \mathbb R$ can be arbitrary and therefore
the Markov walk $(S_n )_{n\geq 0}$ is not necessarily lattice valued.  
In this case the Markov chain formed by the couple
$( X_n, y+S_n )_{n\geq 0}$ cannot be reversed directly as in \eqref{exceang001}. 
We cope with this by altering the arrival interval $[z,z+h]$ in the following two-sided bound
\begin{align}
	\sum_{x^* \in \bb X} \bb E_{x^*}^* &\left( \psi_{x}^*(X_n^*) \bbm 1_{\left\{ z+S_n^* \in [y-h,y],\, \tau_{z}^* > n \right\}} \right) \bs \nu(x^*) \nonumber \\
	\label{dific001}
	&\qquad\qquad \leq \mathbb P_x (y+S_n\in [z,z+h], \tau_y > n) \\
	&\qquad\qquad \qquad\qquad\qquad\qquad\leq \sum_{x^* \in \bb X} \bb E_{x^*}^* \left( \psi_{x}^*(X_n^*) \bbm 1_{\left\{ z+h+S_n^* \in [y,y+h],\, \tau_{z+h}^* > n \right\}} \right) \bs \nu(x^*), \nonumber
\end{align}
where $\bs \nu$ is the invariant probability of the Markov chain $( X_n )_{n\geq 1}$,  $\psi_{x}^*: \mathbb X \mapsto \bb R_+$ 
is a function  such that $\bs \nu \left( \psi_{x}^* \right) = 1$ (see \eqref{hirondelle} for a precise definition) and 
 $S_n^* = -\sum_{k=1}^n f \left( X_k^* \right)$, $\forall n \geq 1.$ Following this idea, for a fixed $a >0$ we split the interval $[z,z+a]$ into $p$ subintervals of length $h=a/p$ and
we determine the exact 
upper and lower bounds for the corresponding expectations in \eqref{dific001}.
We then patch up the obtained bounds
to obtain a precise asymptotic as $n \to +\infty$ for the probabilities $\mathbb P_x (y+S_n\in [z,z+a], \tau_y > n)$ for a fixed $a>0$ 
and let then $p$ go to $+\infty$. 
This resumes very succinctly how we suggest generalizing \eqref{exceang001} to the non-lattice case. 
Together with some further developments in Sections \ref{heron} and \ref{giraffe}, this allows us to establish Theorems  \ref{GNLLT} and \ref{LLTC}.

The outline of the paper is as follows: 
\begin{itemize}
\item Section \ref{sec not res}: We give the necessary notations and formulate the main results.
\item Section \ref{dual chain}: Introduce the dual Markov chain and state some of its properties.
\item Section \ref{OPPET}: Introduce and study the perturbed transition operator.
\item Section \ref{SecLLTWC}: We prove a non-asymptotic local limit theorem for sums defined on Markov chains.
\item Section \ref{auxbounds}: We collect some auxiliary bounds.
\item Sections \ref{heron}, \ref{giraffe} and \ref{2theor} : Proofs of Theorems \ref{GNLLT}, \ref{LLTC} and \ref{CAPE}, \ref{CAPEBIS}, respectively.
\item Section \ref{appendix}. We state auxiliary assertions which are necessary for the proofs.
\end{itemize}

Let us end this section by fixing some notations. The symbol $c$ will denote a positive constant depending on the all previously introduced constants. Sometimes, to stress the dependence of the constants on some parameters 
 $\alpha,\beta,\dots$ we shall use the notations $ c_{\alpha}, c_{\alpha,\beta},\dots$. All these constants are likely to change their values every occurrence. The indicator of an event $A$ is denoted by $\mathbbm 1_A$. For any bounded measurable function $f$ on $\bb X$, random variable $X$ in $\bb X$ and event $A$, the integral $\int_{\bb X}  f(x) \bb P (X \in \dd x, A)$ means the expectation $\bb E\left( f(X); A\right)=\bb E \left(f(X) \mathbbm 1_A\right)$.

\section{Notations and results} \label{sec not res}

Let $( X_n )_{n\geq 0}$ be a homogeneous Markov chain on the probability 
space $(\Omega, \scr{F}, \bb P)$ with values in the finite state space $\bb X$. 
Denote by $\scr C$ the set of complex functions defined on $\bb X$ endowed with the norm $\norm{\cdot}_{\infty}$: $\norm{g}_{\infty} = \sup_{x\in \bb X} \abs{g(x)}$, for any $g\in \scr C$. Let $\bf P$ be the transition kernel of the Markov chain
$( X_n )_{n\geq 0}$ to which we associate the following transition operator: for any $x\in \bb X$ and $g \in \scr C$,
\[
\bf P g(x) = \sum_{x'\in \bb X} g(x') \bf P(x,x').
\]
For any $x\in \bb X$, denote by $\bb P_x$ and $\bb E_x$ the probability, respectively the expectation, generated by the finite dimensional distributions of the Markov chain $( X_n )_{n\geq 0}$ starting at $X_0 = x$.
We assume that the Markov chain is irreducible and aperiodic, which is equivalent to the following hypothesis.

\begin{hypothesis}
\label{primitive}
The matrix $\bf P$ is primitive: there exits $k_0 \geq 1$ such that for any $x \in \bb X$ and any non-negative and non identically zero function $g \in \scr C$,
\[
\bf P^{k_0}g(x) > 0.
\]
\end{hypothesis}

Let $f$ be a real valued function defined on $\bb X$ and let $(S_n)_{n\geq 0}$ be the process defined by
\[
S_0 = 0 \qquad \text{and} \qquad S_n = f \left( X_1 \right) + \dots + f \left( X_n \right), \quad \forall n \geq 1.
\]
For any starting point $y \in \bb R$ we consider the Markov walk $(y+S_n)_{n\geq 0}$ and we denote by $\tau_y$ the first time when the Markov walk becomes non-positive:
\[
\tau_y := \inf \left\{ k \geq 1, \; y+S_k \leq 0 \right\}.
\]

Under \ref{primitive}, by the Perron-Frobenius theorem, there is a unique positive invariant probability $\bs \nu$ on $\bb X$ satisfying the following property: there exist $c_1>0$ and $c_2>0$ such that for any function $g \in \scr C$ and $n \geq 1$,
\begin{equation}
	\label{decexpNLLT}
	\sup_{x\in \bb X} \abs{\bb E_x \left( g\left( X_n \right) \right) - \bs \nu(g)} = \norm{\bf P^ng - \bs \nu(g)}_{\infty} \leq \norm{g}_{\infty} c_1\e^{-c_2n},
\end{equation}
where $\bs \nu(g) = \sum_{x\in \bb X} g(x) \bs \nu(x)$. 

The following two hypotheses ensure that the Markov walk has no-drift and is non-lattice, respectively.
\begin{hypothesis}
\label{nucentre}
The function $f$ is centred:
\[
\bs \nu \left( f \right) = 0.
\]
\end{hypothesis}

\begin{hypothesis}
\label{nondegenere}
For any $(\theta,a) \in \bb R^2$, there exists a sequence $x_0, \dots, x_n$ in $\bb X$ such that
\[
\bf P(x_0,x_1) \cdots \bf P(x_{n-1},x_n) \bf P(x_n,x_0) > 0
\]
and
\[
f(x_0) + \cdots + f(x_n) - (n+1)\theta \notin a\bb Z.
\]
\end{hypothesis}

Under Hypothesis \ref{primitive}, it is shown in Section \ref{OPPET} that Hypothesis \ref{nondegenere} is equivalent to the condition that the perturbed operator $\bf P_t $ has a spectral radius less than $1$ for $ t\neq 0$; for more details we refer to Section \ref{OPPET}.
Furthermore, in the Appendix (see Lemma \ref{couette}, Section \ref{appendix}), we show that Hypotheses \ref{primitive}-\ref{nondegenere} imply that the following number $\sigma^2$, which is the limit of $\bb E_x ( S_n^2 )/n$ as $n \to +\infty$ for any $x \in \bb X$, is not zero:
\begin{equation}
\label{mu-sigma001LLT}
\sigma^2 := \bs \nu(f^2) + 2\sum_{n=1}^{+\infty} \bs \nu \left( f \bf P^n f \right) > 0.
\end{equation}

Under spectral gap assumptions, the asymptotic behaviour of the survival probability $\bb P_x \left( \tau_y > n \right)$ and of the conditional law of the Markov walk $\frac{y+S_n}{\sqrt{n}}$ given the event $\{ \tau_y > n \}$ have been studied in \cite{grama_limit_2016-1}. It is easy to see that under \ref{primitive}, \ref{nucentre} and \eqref{mu-sigma001LLT} the conditions of \cite{grama_limit_2016-1} are satisfied (see Section \ref{appendix}). 
We summarize the main results of \cite{grama_limit_2016-1} in the following propositions.

\begin{proposition}[Preliminary results, part I] Assume Hypotheses \ref{primitive}-\ref{nondegenere}.
\label{RESUI}
There exists a non-degenerate non-negative function $V$ on $\bb X \times \bb R$ such that
\begin{enumerate}[ref=\arabic*, leftmargin=*, label=\arabic*.]
	\item \label{RESUI001} For any $(x,y) \in \bb X \times \bb R$ and $n \geq 1$,
	\[
	\bb E_x \left( V \left( X_n, y+S_n \right) \,;\, \tau_y > n \right) = V(x,y).
	\]
	\item \label{RESUI002} For any $x \in \bb X$, the function $V(x,\cdot)$ is non-decreasing and for any $(x,y) \in \bb X \times \bb R$,
	\[
	V(x,y) \leq c \left( 1+\max(y,0) \right).
	\]
	\item \label{RESUI003} For any $x \in \bb X$, $y \in \bb R$ and $\delta \in (0,1)$,
	\[
	\left( 1- \delta \right)\max(y,0) - c_{\delta} \leq V(x,y) \leq \left(1+\delta \right)\max(y,0) + c_{\delta}.
	\]
\end{enumerate}
\end{proposition}

Since the function $V$ satisfies the point \ref{RESUI001}, it is said to be harmonic.

\begin{proposition}[Preliminary results, part II] Assume Hypotheses \ref{primitive}-\ref{nondegenere}.
\begin{enumerate}[ref=\arabic*, leftmargin=*, label=\arabic*.]
\label{RESUII}
	\item \label{RESUII001} For any $(x,y) \in \bb X \times \bb R$,
	\[
	\lim_{n\to +\infty} \sqrt{n}\bb P_x \left( \tau_y > n \right) = \frac{2V(x,y)}{\sqrt{2\pi} \sigma},
	\]
	where $\sigma$ is defined by \eqref{mu-sigma001LLT}.
	\item \label{RESUII002} For any $(x,y) \in \bb X \times \bb R$ and $n\geq 1$,
	\[
	\bb P_x \left( \tau_y > n \right) \leq c\frac{ 1 + \max(y,0) }{\sqrt{n}}.
	\]
\end{enumerate}
\end{proposition}

Define the support of $V$ by
\begin{equation}
	\label{loup}
	\supp (V) := \{ (x,y) \in \bb X \times \bb R : V(x,y)> 0 \}.
\end{equation}
Note that from property \ref{RESUI003} of Proposition \ref{RESUI}, for any fixed $x\in \bb X$, the function $y \mapsto V(x,y)$ is positive for large $y$. 
For further details on the properties of $\supp (V)$ we refer to \cite{grama_limit_2016-1}.

\begin{proposition}[Preliminary results, part III] Assume Hypotheses \ref{primitive}-\ref{nondegenere}.
\label{RESUIII}
\begin{enumerate}[ref=\arabic*, leftmargin=*, label=\arabic*.]
	\item \label{RESUIII001} For any $(x,y) \in \supp(V)$ and $t\geq 0$,
	\[
	\bb P_x \left( \sachant{\frac{y+S_n}{\sigma \sqrt{n}} \leq t }{\tau_y >n} \right) \underset{n\to+\infty}{\longrightarrow} \mathbf \Phi^+(t),
	\]
	where $\bf \Phi^+(t) = 1-\e^{-\frac{t^2}{2}}$ is the Rayleigh distribution function.
	\item \label{RESUIII002} There exists $\ee_0 >0$ such that, for any $\ee \in (0,\ee_0)$, $n\geq 1$, $t_0 > 0$, $t\in[0,t_0]$ and $(x,y) \in \bb X \times \bb R$,
	\[
	\abs{ \bb P_x \left( y+S_n \leq t \sqrt{n} \sigma \,,\, \tau_y > n \right) - \frac{2V(x,y)}{\sqrt{2\pi n}\sigma} \bf \Phi^+(t) } \leq c_{\ee,t_0} \frac{\left( 1+\max(y,0)^2 \right)}{n^{1/2+\ee}}.
	\]
\end{enumerate}
\end{proposition}

In the point \ref{RESUII001} of Proposition \ref{RESUII}  and the point \ref{RESUIII002} of Proposition \ref{RESUIII}, the function $V$ can be zero, so that for all pairs $(x,y)$ satisfying $V(x,y)=0$ it holds 
\[
\lim_{n\to +\infty} \sqrt{n} \bb P_x \left(\tau_y > n \right) = 0
\]
and
\[
\lim_{n\to +\infty} \sqrt{n} \bb P_x \left( y+S_n \leq t \sqrt{n} \sigma \,,\, \tau_y > n \right) = 0.
\]

Now we proceed to formulate the main results of the paper. 
Our first result is an extension of Gnedenko-Stone local limit theorem originally stated for sums of independent random variables.
The following theorem generalizes it to the case of sums of random variables defined on Markov chains conditioned to stay positive. 

\begin{theorem} \label{GNLLT}
Assume Hypotheses \ref{primitive}-\ref{nondegenere}.
Let $a>0$ be a positive real. Then there exists $\ee_0 \in (0,1/4)$ such that for any $\ee \in (0,\ee_0)$,  non-negative function $\psi \in \scr C$, $y \in \bb R$ and
$ n \geq 3\ee^{-3}$, we have
\begin{align*}
	&\sup_{x\in \bb X,\, z \geq 0} n\abs{ \bb E_x \left( \psi \left( X_{n} \right) \,;\, y+S_{n} \in [z,z+a] \,,\, \tau_y > n \right) - \frac{2a \bs \nu \left( \psi \right) V(x,y)}{\sqrt{2\pi}  \sigma^2 n} \varphi_+ \left( \frac{z}{\sqrt{n}\sigma} \right)} \\
	&\hspace{3cm} \leq c \left(1 + \max(y,0) \right) \norm{\psi}_{\infty} \left( \sqrt{\ee} + \frac{c_{\ee}\left( 1+\max(y,0) \right)}{n^{\ee}} \right),
\end{align*}
where $\varphi_+(t) = t\e^{-\frac{t^2}{2}} \bbm 1_{\{t\geq 0\}}$ is the Rayleigh  density and the constants $c$ and $c_{\ee}$ may depend on $a$.
\end{theorem}
Note that Theorem \ref{GNLLT} is meaningful only for large values of $z$ such that $z \sim n^{1/2} $ as $n\to\infty$.
Indeed, the remainder term is of order $n^{-1-\varepsilon},$ 
with some small $\varepsilon >0,$  
while for a fixed $z$ the leading term is of order $n^{-3/2}$. 
When $z = cn^{1/2}$ the leading term becomes of order $n^{-1}$ while the remainder is still $o(n^{-1})$. To deal with the case of $z$ in compact sets a more refined result will be given below.
We will deduce it from Theorem \ref{GNLLT}, however for the proof we need the concept of duality.

Let us introduce the dual Markov chain and the corresponding associated  Markov walk. 
Since $\bs \nu$ is positive on $\bb X$, the following dual Markov kernel $\bf P^*$ is well defined:
\begin{equation}
\label{statue}
\bf P^* \left( x,x^* \right) = \frac{\bs \nu \left( x^* \right)}{\bs \nu (x)} \bf P \left( x^*,x \right), \quad \forall (x,x^*) \in \bb X^2.
\end{equation}
It is easy to see that $\bs \nu$ is also $\bf P^*$-invariant.
The dual of $( X_n)_{n\geq 0}$ is the Markov chain $\left( X_n^* \right)_{n\geq 0}$ with values in $\bb X$ and transition probability $\bf P^*$. 
Without loss of generality we can consider that the dual Markov chain $\left( X_n^* \right)_{n\geq 0}$ is
defined on an extension of the probability space $(\Omega, \scr F, \bb P)$
and that it is independent of the Markov chain $( X_n)_{n\geq 0}$.
We define the associated dual Markov walk by
\begin{equation}
\label{bataille}
S_0^* = 0 \qquad \text{and} \qquad S_n^* = \sum_{k=1}^n -f \left( X_k^* \right), \quad \forall n \geq 1.
\end{equation}
For any $z\in \bb R$, define also the exit time
\begin{equation}
\label{bataillebis}
\tau_z^* := \inf \left\{ k \geq 1 : z+S_k^* \leq 0 \right\}.
\end{equation}
For any $\in \bb X$, denote by $\bb P_x^*$ and $\bb E_x^*$ the probability, respectively the expectation, generated by the finite dimensional distributions of the Markov chain $( X_n^* )_{n\geq 0}$ starting at $X_0^* = x$. 
It is shown in Section \ref{dual chain} that the dual Markov chain  $\left( X_n^* \right)_{n\geq 0}$
satisfies Hypotheses \ref{primitive}-\ref{nondegenere} 
as do the original chain $\left( X_n \right)_{n\geq 0}$. 
Thus, Propositions \ref{RESUI}-\ref{RESUIII} hold also for $\left( X_n^* \right)_{n\geq 0}$ with $V,$ $\tau,$ $(S_n)_{n\geq 0}$ and $\bb P_x$
replaced by $V^*,$ $\tau^*,$ $(S_n^*)_{n\geq 0}$ and $\bb P_x^*$.
Note also that both chains have the same invariant probability $\bs \nu$. Denote by $\bb E_{\bs \nu}$, $\bb E_{\bs \nu}^*$ the expectations generated by the finite dimensional distributions of the Markov chains $( X_n )_{n\geq 0}$ and $( X_n^* )_{n\geq 0}$
in the stationary regime. 

Our second result is a conditional version of the local limit theorem for fixed $x,y$ and $z$. 
\begin{theorem} \label{LLTC}
Assume Hypotheses \ref{primitive}-\ref{nondegenere}.
\begin{enumerate}[ref=\arabic*, leftmargin=*, label=\arabic*.]
\item \label{LLTC001} For any non-negative function $\psi \in \scr C$, $a>0$, $x\in \bb X$, $y \in \bb R$ and $z \geq 0$,
\begin{align*}
	\lim_{n\to +\infty} n^{3/2} &\bb E_x \left( \psi \left( X_{n} \right) \,;\, y+S_{n} \in [z,z+a] \,,\, \tau_y > n \right) \\
	&\qquad = \frac{2V(x,y)}{\sqrt{2\pi}\sigma^3} \int_z^{z+a} \bb E_{\bs \nu}^* \left( \psi \left( X_1^* \right) V^*\left( X_1^*, z'+S_1^* \right) \,;\, \tau_{z'}^* > 1 \right) \dd z'.
\end{align*}
\item \label{LLTC002} Moreover, there exists $c > 0$ such that for any $a>0$, non-negative function $\psi \in \scr C$, $y \in \bb R$, $z \geq 0$ and $n \geq 1$,
\[
\sup_{x\in \bb X} \bb E_x \left( \psi \left( X_{n} \right) \,;\, y+S_{n} \in [z,z+a] \,,\, \tau_y > n \right) \leq \frac{c \norm{\psi}_{\infty}}{n^{3/2}} \left( 1+a^3 \right)\left( 1+z \right)\left( 1+\max(y,0) \right).
\]
\end{enumerate}
\end{theorem}
In the particular case when $\psi=1$, the previous theorem rewrites as follows:

\begin{corollary} \label{COROL}
Assume Hypotheses \ref{primitive}-\ref{nondegenere}.
\begin{enumerate}[ref=\arabic*, leftmargin=*, label=\arabic*.]
\item \label{COROL001} For any $a>0$, $x\in \bb X$, $y \in \bb R$ and $z \geq 0$,
\begin{align*}
	\lim_{n\to +\infty} n^{3/2} &\bb P_x \left( y+S_{n} \in [z,z+a] \,,\, \tau_y > n \right) \\
	&\qquad = \frac{2V(x,y)}{\sqrt{2\pi}\sigma^3} \int_z^{z+a} \int_{\bb X}  V^*\left( x', z' \right) \bs \nu (\dd x') \dd z'.
\end{align*}
\item \label{COROL002}Moreover, there exists $c > 0$ such that for any $a>0$, $y \in \bb R$, $z \geq 0$ and $n \geq 1$,
\[
\sup_{x\in \bb X} \bb P_x \left( y+S_{n} \in [z,z+a] \,,\, \tau_y > n \right) \leq \frac{c}{n^{3/2}} \left( 1+a^3 \right) \left( 1+z \right)\left( 1+\max(y,0) \right).
\]
\end{enumerate}
\end{corollary}

Note that the assertion \ref{LLTC001} of Theorem \ref{LLTC} and assertion \ref{COROL001} of Corollary \ref{COROL} 
hold for fixed $a>0$, $x\in \bb X$, $y \in \bb R$ and $z \geq 0$ and that these results 
are no longer true when $z$ is not in a compact set, for instance when $z \sim n^{1/2}$.

The following result extends Theorem \ref{LLTC} to some functionals of the trajectories of the chain $( X_n )_{n\geq 0}$.
For any $(x,x^*) \in \bb X^2$, the probability generated by the finite dimensional distributions of the two dimensional Markov chain $( X_n, X_n^*)_{n\geq 0}$ starting at $(X_0,X_0^*) = (x,x^*)$ is given by $\bb P_{x,x^*}=\bb P_{x} \times \bb P_{x^*}^*$.
Let $\bb E_{x,x^*}$ be the corresponding expectation. For any $l \geq 1$, denote by $\scr C^+ ( \bb X^l \times \bb R_+ )$ the set of non-negative functions $g$: $\bb X^l \times \bb R_+ \to \bb R_+$ 
satisfying the following properties:
\begin{itemize}
\item for any $(x_1,\dots,x_l) \in \bb X^l$, the function $z \mapsto g(x_1,\dots,x_l,z)$ is continuous,
\item there exists $\ee > 0$ such that $\max_{x_1,\dots x_l \in \bb X} \sup_{z \geq 0} g(x_1,\dots,x_l,z) (1+z)^{2+\ee} < +\infty$.
\end{itemize}

\begin{theorem}
\label{CAPE}
Assume Hypotheses \ref{primitive}-\ref{nondegenere}. For any $x \in \bb X$, $y \in \bb R$, $l \geq 1$, $m \geq 1$ and $g \in \scr C^+ \left( \bb X^{l+m} \times \bb R_+ \right)$,
\begin{align*}
&\lim_{n\to +\infty} n^{3/2} \bb E_x \left( g \left(X_1, \dots, X_l, X_{n-m+1}, \dots, X_n, y+S_n \right) \,;\, \tau_y > n \right) \\
&\qquad = \frac{2}{\sqrt{2\pi}\sigma^3} \int_0^{+\infty} \sum_{x^* \in \bb X} \bb E_{x,x^*} \left( g \left( X_1, \dots, X_l,X_m^{*},\dots,X_1^{*},z \right) \right. \\
&\hspace{4cm} \left. \times V \left( X_l, y+S_l \right) V^* \left( X_m^*, z+S_m^* \right) \,;\, \tau_y > l \,,\, \tau_z^* > m \right) \bs \nu(x^*) \dd z.
\end{align*}
\end{theorem}

As a consequence of Theorem \ref{CAPE}
we deduce the following asymptotic behaviour of the probability 
of the event $\left\{ \tau_y=n \right\}$ as $n\to +\infty$. 

\begin{theorem}
\label{CAPEBIS}
Assume Hypotheses \ref{primitive}-\ref{nondegenere}. For any $x\in \bb X$ and $y \in \bb R$,
\[
\lim_{n\to+\infty} n^{3/2}\bb P_x \left( \tau_y = n \right) = \frac{2V(x,y)}{\sqrt{2\pi} \sigma^3} \int_0^{+\infty} \bb E_{\bs \nu}^* \left( V^*(X_1^*,z) \,;\, S_1^* \geq z \right) \dd z.
\]
\end{theorem}

\section{Properties of the dual Markov chain}
\label{dual chain}

In this section we establish some properties of the dual Markov chain and of the corresponding Markov walk.

\begin{lemma}
Suppose that the operator $\bf P$ satisfies Hypotheses \ref{primitive}-\ref{nondegenere}. Then the dual operator $\bf P^*$ satisfies also \ref{primitive}-\ref{nondegenere}.
\end{lemma}

\begin{proof}
By the definition of $\bf P^*,$ for any $x^* \in \bb X$,
\[
\sum_{x \in \bb X} \bs \nu (x) \bf P^* \left( x,x^* \right) = \sum_{x \in \bb X} \bf P \left( x^*,x \right) \bs \nu \left( x^* \right) = \bs \nu (x^*),
\]
which proves that $\bs \nu$ is also $\bf P^*$-invariant. 
Thus Hypothesis \ref{nucentre}, $\bs \nu(f) = \bs \nu(-f)=0$, is satisfied for both chains.
Moreover, it is easy to see that for any $n \geq 1$, $(x,x^*) \in \bb X^2$, 
\[
\left(\bf  P^* \right)^n (x,x^*) = \bf P^n (x^*,x) \frac{\bs \nu(x^*)}{\bs \nu(x)}.
\]
This shows that $\bf P^*$ satisfies \ref{primitive} and \ref{nondegenere}. 
\end{proof}

Note that the operator $\bf P^*$ is the adjoint operator of $\bf P$ in the space $\LL^2 \left( \bs \nu \right) :$ for any functions $g$ and $h$ on $\bb X,$
\[
\bs \nu \left( g \left(\bf P^*\right)^n h \right) = \bs \nu \left( h \bf P^n g \right).
\]
In particular for any $n\geq 1$, $\bs \nu \left( f \left(\bf P^*\right)^n f \right) = \bs \nu \left( f \bf P^n f \right)$ and we note that
\[
\sigma^2 = \bs \nu \left((-f)^2 \right) + \sum_{n} \bs \nu \left((-f) \left( \bf P^* \right)^n (-f) \right).
\]

The following assertion plays a key role in the proofs. 
\begin{lemma}[Duality]
\label{duality}
For any probability measure $\mathfrak{m}$ on $\bb X$, any $n\geq 1$ and any function $F$ from $\bb X^n$ to $\bb R$,
\[
\bb E_{\mathfrak{m}} \left( F \left( X_1, \dots, X_{n-1}, X_n \right) \right) = \bb E_{\bs \nu}^* \left( F \left( X_n^*, X_{n-1}^*, \dots, X_1^* \right) \frac{\mathfrak{m} \left( X_{n+1}^* \right)}{\bs \nu \left( X_{n+1}^* \right)} \right).
\]
\end{lemma}

\begin{proof}
We write
\begin{align*}
	\bb E_{\mathfrak{m}} &\left( F \left( X_1, \dots, X_{n-1}, X_n \right) \right) \\
	&= \sum_{x_0,x_1, \dots, x_{n-1}, x_n, x_{n+1} \in \bb X} F \left( x_1, \dots, x_{n-1}, x_n \right) \mathfrak{m}(x_0) \\
	&\hspace{3cm} \bb P_{x_0} \left( X_1 = x_1, X_2 = x_2, \dots, X_{n-1} = x_{n-1}, X_n = x_n, X_{n+1} = x_{n+1} \right). 
\end{align*}
By the definition of $\bf P^*$, we have
\begin{align*}
	&\bb P_{x_0} \left( X_1 = x_1, X_2 = x_2, \dots, X_{n-1} = x_{n-1}, X_n = x_n, X_{n+1} = x_{n+1} \right) \\
	&= \bf P(x_0,x_1) \bf P(x_1,x_2) \dots \bf P(x_{n-1},x_n) \bf P(x_n,x_{n+1}) \\
	&= \bf P^*(x_1,x_0) \frac{\bs \nu (x_1)}{\bs \nu (x_0)} \bf P^*(x_2,x_1) \frac{\bs \nu (x_2)}{\bs \nu (x_1)} \dots \bf P^*(x_n,x_{n-1}) \frac{\bs \nu (x_n)}{\bs \nu (x_{n-1})}\bf P^*(x_{n+1},x_n)\frac{\bs \nu (x_{n+1})}{\bs \nu (x_n)} \\
	&= \frac{\bs \nu (x_{n+1})}{\bs \nu (x_0)} \bb P^*_{x_{n+1}} \left( X_1^* = x_n, X_2^* = x_{n-1}, \dots, X_n^* = x_1, X_{n+1}^* = x_0 \right)
\end{align*}
and the result of the lemma follows.
\end{proof}

\section{The perturbed operator}
\label{OPPET}

For any $t \in \bb R$, denote by $\bf P_t$ the perturbed transition operator defined by
\[
\bf P_tg(x) = \bf P \left( \e^{\bf itf} g \right)(x) = \bb E_x \left( \e^{\bf itf(X_1)} g(X_1) \right), \quad \text{for any } g \in \scr C,\; x \in \bb X,
\]
where $\bf i$ is the complex $\bf i^2 = -1$.
Let also $r_t$ be the spectral radius of $\bf P_t$. Note that for any $g \in \scr C$, $\norm{\bf P_t g}_{\infty} \leq \norm{\e^{\bf itf} g}_{\infty} = \norm{g}_{\infty}$ and so
\begin{equation}
\label{semaphore}
r_t \leq 1.
\end{equation}

We introduce the two following definitions:
\begin{itemize}
\item A sequence $x_0, x_1, \dots, x_n \in \bb X$, is a \textit{path} (between $x_0$ and $x_n$) if 
\[
\bf P(x_0,x_1) \cdots \bf P(x_{n-1},x_n) > 0.
\]
\item A sequence $x_0, x_1, \dots, x_n \in \bb X$, is an \textit{orbit} if $x_0, x_1, \dots, x_n, x_0$ is a path.
\end{itemize}
Note that under Hypothesis \ref{primitive}, for any $x_0, x\in \bb X$ it is always possible to connect $x_0$ and $x$ by a path $x_0, x_1, \dots, x_n, x$ in $\bb X$.

\begin{lemma}
Assume Hypothesis \ref{primitive}.
\label{lapin}
The following statements are equivalent:
\begin{enumerate}[ref=\arabic*, leftmargin=*, label=\arabic*.]
\item \label{lapin001} There exists $(\theta,a) \in \bb R^2$ such that for any orbit $x_0, \dots, x_n$ in $\bb X$, we have 
\[
f(x_0) + \cdots + f(x_n) - (n+1)\theta \in a\bb Z.
\]
\item \label{lapin002} There exist $t\in \bb R^*$, $h\in \scr C\setminus \{0\}$ and $\theta \in \bb R$ such that for any $(x,x') \in \bb X^2$,
\[
h(x')\e^{\bf itf(x')}\bf P(x,x') = h(x) \e^{\bf it\theta} \bf P(x,x').
\]
\item \label{lapin003} There exists $t \in \bb R^*$ such that
\[
r_t = 1.
\]
\end{enumerate}
\end{lemma}

\begin{proof}
\textit{The point \ref{lapin001} implies the point \ref{lapin002}.} Suppose that the point \ref{lapin001} holds. Fix $x_0 \in \bb X$ and set $h(x_0) = 1$. For any $x \in \bb X$, define $h(x)$ in the following way: for any path $x_0, \dots, x_n, x$ in $\bb X$ we set
\[
h(x) = \e^{\bf it\theta(n+1)} \e^{-\bf i t \left( f(x_1) + \dots + f(x_n) + f(x) \right)},
\]
where $t = \frac{2\pi}{a}$. Note that if $a=0$, then the point \ref{lapin001} holds also for $a=1$ and so, without lost of generality, we assume that $a\neq 0$. We first verify that $h$ is well defined on $\bb X$. Recall that under Hypothesis \ref{primitive}, for any $x\in \bb X$ it is always possible to connect $x_0$ and $x$ by a path. We have to check that the value of $h(x)$ does not depend on the choice of the path. Let $p,q \geq 1$ and $x_0,x_1, \dots, x_p, x$ in $\bb X$ and $x_0,y_1, \dots, y_q, x$ in $\bb X$ be two paths between $x_0$ and $x$. We complete these paths to orbits as follows. Under Hypothesis \ref{primitive}, there exist $n \geq 1$ and $z_1, \dots, z_n$ in $\bb X$ such that
\[
\bf P(x,z_1) \cdots \bf P(z_n,x_0) > 0,
\]
i.e.\ the sequence $x, z_1, \dots, z_n, x_0$ is a path. So, the sequences $x_0,x_1,\dots,x_p,x,z_1,\dots,z_n$ and $x_0,y_1,\dots,y_q,x,z_1,\dots,z_n$ are orbits. By the point \ref{lapin001}, there exist $l_1,l_2 \in \bb Z$ such that
\begin{align*}
f(x_1) + \dots + f(x_p) + f(x) &= al_1 - \left( f(z_1) + \dots + f(z_n) + f(x_0) \right) + (p+n+2)\theta \\
&= al_1 - al_2 + \left( f(y_1) + \dots + f(y_q) + f(x) \right) \\
&\hspace{5cm}- (q+n+2)\theta + (p+n+2)\theta.
\end{align*}
Therefore,
\[
\e^{\bf it\theta(p+1)} \e^{-\bf i t \left( f(x_1) + \dots + f(x_p) + f(x) \right)} = \e^{-\bf i t \left( al_1 - al_2 \right)} \e^{\bf it\theta(q+1)} \e^{-\bf i t \left( f(y_1) + \dots + f(y_q) + f(x) \right)}
\]
and since $ta=2\pi$ it proves that $h$ is well defined. Now let $(x,x') \in \bb X^2$ be such that $\bf P(x,x') > 0$. There exists a path $x_0, x_1, \dots, x_n, x$ between $x_0$ and $x$ and so
\[
h(x) = \e^{\bf it\theta(n+1)} \e^{-\bf i t \left( f(x_1) + \dots + f(x_n) + f(x) \right)}.
\]
Since $x_0,x_1, \dots, x_n,x,x'$ is a path between $x_0$ and $x'$, we have also
\[
h(x') = \e^{\bf it\theta(n+2)} \e^{-\bf i t \left( f(x_1) + \dots + f(x_n) + f(x)+f(x') \right)} = h(x) \e^{\bf it\theta} \e^{-\bf i t f(x')}.
\]
Note that since the modulus of $h$ is $1$, this function belongs to $\scr C \setminus\{0\}$.

\textit{The point \ref{lapin002} implies the point \ref{lapin001}.} Suppose that the point \ref{lapin002} holds and let $x_0, \dots, x_n$ be an orbit. Using the point \ref{lapin002} repeatedly, we have
\[
h(x_0) = h(x_n) \e^{\bf it\theta} \e^{-\bf i t f(x_0)} = \dots  = h(x_0) \e^{\bf it\theta(n+1)} \e^{-\bf i t \left( f(x_0)+\dots+f(x_n) \right)}.
\]
Since $h$ is a non-identically zero function with a constant modulus, necessarily, $h$ is never equal to $0$ and so $f(x_0)+\dots+f(x_n) -(n+1)\theta \in \frac{2\pi}{t} \bb Z$.

\textit{The point \ref{lapin002} implies the point \ref{lapin003}.} Suppose that the point \ref{lapin002} holds. Summing on $x'$ we have, for any $x \in \bb X$,
\[
\bf P \left( h \e^{itf} \right)(x) = \bf P_t h(x) = h(x)\e^{\bf i t \theta}.
\]
Therefore $h$ is an eigenvector of $\bf P_t$ associated to the eigenvalue $\e^{\bf i t \theta}$ which implies that $r_t \geq \abs{\e^{\bf i t \theta}} = 1$ and by \eqref{semaphore}, $r_t = 1$.

\textit{The point \ref{lapin003} implies the point \ref{lapin002}.} Suppose that the point \ref{lapin003} holds. 
There exist $h \in \scr C \setminus \{0\}$ and $\theta \in \bb R$ such that $\bf P_t h = h\e^{\bf i t \theta}$. 
Without loss of generality, we suppose that $\norm{h}_{\infty} = 1$. Since $\bf P_t^n h = h\e^{\bf i t n \theta}$ for any $n \geq 1$, by \eqref{decexpNLLT}, for any $x \in \bb X$, we have 
\begin{equation}
\abs{h(x)} = \abs{\bf P_t^n h(x)} \leq \bf P^n \abs{h}(x) \underset{n\to+\infty}{\longrightarrow} \bs \nu \left( \abs{h} \right).
\label{hhbound001}
\end{equation}
From \eqref{hhbound001}, letting $x_0 \in \bb X$ be such that $\abs{h(x_0)} = \norm{h}_{\infty} = 1$, 
it is easy to see that 
\[
\abs{h(x_0)} \leq \sum_{x \in \bb X} \abs{h(x)} \bs \nu (x) \leq \abs{h(x_0)}.
\] 
From this it follows that the modulus of $h$ is constant on $\bb X$: $\abs{h(x)} = \abs{h(x_0)} = 1$ for any $x \in \bb X$. 
Consequently, there exists $\alpha$: $\bb X \to \bb R$ such that for any $x \in \bb X$,
\begin{equation}
\label{saule}
h(x) = \e^{\bf i \alpha(x)}.
\end{equation}
With \eqref{saule} the equation $\bf P_t h = h\e^{\bf i t \theta}$ can be rewritten as
\[
\forall x \in \bb X, \qquad \sum_{x'\in \bb X} \e^{\bf i \alpha(x')} \e^{\bf i t f(x')} \bf P(x,x') = \e^{\bf i \alpha(x)} \e^{\bf i t \theta}.
\]
Since $\e^{\bf i \alpha(x)} \e^{\bf i t \theta} \in \left\{ z \in \bb C : \abs{z}=1 \right\}$ and $\e^{\bf i \alpha(x')} \e^{\bf i f(x')} \in \left\{ z \in \bb C : \abs{z}=1 \right\}$, for any $x' \in \bb X$, the previous equation holds only if $h(x')\e^{\bf i t f(x')} = \e^{\bf i \alpha(x')} \e^{\bf i t f(x')} = \e^{\bf i \alpha(x)} \e^{\bf i t \theta} = h(x) \e^{\bf i t \theta}$ for any $x' \in \bb X$ such that $\bf P(x,x') > 0$.
\end{proof}

Define the operator norm $\norm{\cdot}_{\scr C \to \scr C}$  on $\scr C$ as follows: for any operator $R$: $\scr C \to \scr C$, set
\[
\norm{R}_{\scr C \to \scr C} := \sup_{g \in \scr C \setminus \{0\}} \frac{\norm{R(g)}_{\infty}}{\norm{g}_{\infty}}.
\]

\begin{lemma}
\label{NOLA}
Assume Hypotheses \ref{primitive} and \ref{nondegenere}.
For any compact set $K$ included in $\bb R^*$ there exist constants $c_K > 0$ and $c_K' >0$ such that for any $n \geq 1$,
\[
\sup_{t\in K} \norm{\bf P_t^n}_{\scr C \to \scr C} \leq c_K \e^{-c_K'n}.
\]
\end{lemma}

\begin{proof}
By Lemma \ref{lapin}, under Hypotheses \ref{primitive} and \ref{nondegenere}, we have $r_t \neq 1$ for any $t\neq 0$ and  hence, using \eqref{semaphore},
\[
r_t < 1, \qquad \forall t \in \bb R^*.
\]
It is well known that
\[
r_t = \lim_{n\to+\infty} \norm{\bf P_t^n}_{\scr C \to \scr C}^{1/n}.
\]
Since $t \mapsto \bf P_t$ is continuous, the function $t \mapsto r_t$ is the infimum of the sequence of upper semi-continuous functions $t \mapsto \norm{\bf P_t^n}_{\scr C \to \scr C}^{1/n}$ and therefore is itself upper semi-continuous. In particular, for any compact set $K$ included in $\bb R^*$, there exists $t_0 \in K$ such that 
\[
\sup_{t\in K} r_t = r_{t_0} < 1.
\]
We deduce that for $\ee = (1- \sup_{t\in K} r_t)/2 >0$ there exists $n_0 \geq 1$ such that for any $n \geq n_0$,
\[
\norm{\bf P_t^n}_{\scr C \to \scr C}^{1/n} \leq \sup_{t\in K} r_t + \ee < 1.
\]
Choosing $c_{K'} = -\ln \left( \sup_{t\in K} r_t + \ee \right)$ and $c_K = \max_{n\leq n_0} \norm{\bf P_t^n}_{\scr C \to \scr C} \e^{c_{K'}n} +1$, the lemma is proved.
\end{proof}

In the proofs we make use of the following assertion which is a consequence of the perturbation theory of linear operator (see for example \cite{kato_perturbation_1976}). The point \ref{SPPT005} is proved in Lemma 2 of Guivarc'h and Hardy \cite{guivarch_theoremes_1988}.

\begin{proposition}
\label{SPPT}
Assume Hypotheses \ref{primitive} and \ref{nucentre}.
There exist a real $\ee_0>0$ and operator valued functions $\Pi_t$ and $Q_t$ acting from $[-\ee_0,\ee_0]$ to the set of operators onto $\scr C $ such that
\begin{enumerate}[ref=\arabic*, leftmargin=*, label= \arabic*.]
	\item \label{SPPT001} the maps $t \mapsto \Pi_t$, $t \mapsto Q_t$ and $t \mapsto \lambda_t$ are analytic at $0$,
	\item \label{SPPT002} the operator $\bf P_t$ has the following decomposition,
	\[
	\bf P_t = \lambda_t \Pi_t +Q_t, \qquad \forall t \in [-\ee_0,\ee_0],
	\]
	\item \label{SPPT003} for any $t\in [-\ee_0,\ee_0]$, $\Pi_t$ is a one-dimensional projector and $\Pi_t Q_t = Q_t \Pi_t = 0$,
	\item \label{SPPT004} there exist $c_1>0$ and $c_2>0$ such that, for any $n\in \bb N^*$,
	\[
	\sup_{t\in [-\ee_0,\ee_0]} \norm{Q_t^n}_{\scr C \to \scr C} \leq c_1 \e^{-c_2 n},
	\]
	\item \label{SPPT005} the function $\lambda_t$ has the following expansion at $0$: for any $t \in [-\ee_0,\ee_0]$,
	\[
	\abs{\lambda_t - 1 + \frac{t^2 \sigma^2}{2}} \leq c \abs{t}^3.
	\]
\end{enumerate}
\end{proposition}

Note that $\lambda_0=1$ and $\Pi_0(\cdot) = \Pi(\cdot) = \bs \nu (\cdot) e$, where $e$ is the unit function of $\bb X$: $e(x) = 1$, for any $x\in \bb X$. 

\begin{lemma}
\label{Cameleon}
Assume Hypotheses \ref{primitive} and \ref{nucentre}.
There exists $\ee_0 > 0$ such that for any $n \geq 1$ and  $t \in [-\ee_0 \sqrt{n}, \ee_0\sqrt{n}]$,
\[
\norm{\bf P_{\frac{t}{\sqrt{n}}}^n - \e^{-\frac{t^2 \sigma^2}{2}} \Pi}_{\scr C \to \scr C} \leq \frac{c}{\sqrt{n}}\e^{-\frac{t^2 \sigma^2}{4}} + c \e^{-c n}.
\]
\end{lemma}

\begin{proof}
By the points \ref{SPPT002} and \ref{SPPT003} of Proposition \ref{SPPT}, for any $t/\sqrt{n} \in [-\ee_0,\ee_0]$,
\[
\bf P_{\frac{t}{\sqrt{n}}}^n = \lambda_{\frac{t}{\sqrt{n}}}^n \Pi_{\frac{t}{\sqrt{n}}} + Q_{\frac{t}{\sqrt{n}}}^n.
\]
By the points \ref{SPPT001} and \ref{SPPT004} of Proposition \ref{SPPT}, for $n\geq 1$,
\begin{align}
	\label{semaphore001}
	\norm{\Pi_{\frac{t}{\sqrt{n}}} - \Pi}_{\scr C \to \scr C} &\leq \sup_{u\in [-\ee_0,\ee_0]} \norm{\Pi_{u}'}_{\scr C \to \scr C} \frac{\abs{t}}{\sqrt{n}} \leq c\frac{\abs{t}}{\sqrt{n}},\\
	\label{semaphore002}
	\sup_{t\in [-\ee_0,\ee_0]} \norm{Q_{\frac{t}{\sqrt{n}}}^n}_{\scr C \to \scr C} &\leq c \e^{-c n}.
\end{align}
Let $\alpha$ be the complex valued function defined on $[-\ee_0,\ee_0]$ by $\alpha(t) = \frac{1}{t^3} \left( \lambda_t - 1 + \frac{t^2 \sigma^2}{2} \right)$ for any $t \in [-\ee_0,\ee_0] \setminus \{0\}$ and $\alpha(0) = 0$. By the point \ref{SPPT005} of Proposition \ref{SPPT}, there exists $c >0$ such that
\begin{equation}
	\label{aurore}
	\forall t \in [-\ee_0,\ee_0], \qquad \abs{\alpha(t)} \leq c.
\end{equation}
With this notation, we have for any $t/\sqrt{n} \in [-\ee_0,\ee_0]$,
\begin{align}
	\abs{\lambda_{\frac{t}{\sqrt{n}}}^n - \e^{-\frac{t^2 \sigma^2}{2}}} &\leq \underbrace{\abs{\left( 1 - \frac{t^2 \sigma^2}{2n} + \frac{t^3}{n^{3/2}} \alpha\left( \frac{t}{\sqrt{n}} \right) \right)^n - \left( 1-\frac{t^2 \sigma^2}{2n} \right)^n}}_{=: I_1} \nonumber\\
	&\qquad + \underbrace{\abs{\left( 1-\frac{t^2 \sigma^2}{2n} \right)^n - \e^{-\frac{t^2 \sigma^2}{2}}}}_{=:I_2}.
	\label{citron}
\end{align}
Without loss of generality, the value of $\ee_0> 0$ can be chosen such that $\ee_0^2 \sigma^2 \leq 1$ and so for any $t/\sqrt{n} \in [-\ee_0,\ee_0]$, we have $1-\frac{t^2 \sigma^2}{2n} \geq 1/2$. Therefore,
\begin{align*}
I_1 &\leq \left( 1-\frac{t^2 \sigma^2}{2n} \right)^n \abs{\left( 1 + \frac{t^3}{n^{3/2}\left( 1-\frac{t^2 \sigma^2}{2n} \right)} \alpha\left( \frac{t}{\sqrt{n}} \right) \right)^n - 1} \\
&\leq \left( 1-\frac{t^2 \sigma^2}{2n} \right)^n \sum_{k=1}^n \begin{pmatrix} n \\ k \end{pmatrix} \abs{\frac{t^3}{n^{3/2}\left( 1-\frac{t^2 \sigma^2}{2n} \right)} \alpha\left( \frac{t}{\sqrt{n}} \right)}^k \\
&= \left( 1-\frac{t^2 \sigma^2}{2n} \right)^n \left[ \left( 1 + \frac{\abs{t}^3}{n^{3/2}\left( 1-\frac{t^2 \sigma^2}{2n} \right)} \abs{\alpha\left( \frac{t}{\sqrt{n}} \right)} \right)^n - 1 \right].
\end{align*}
Using the inequality $1+u \leq \e^{u}$ for $u \in \bb R$, the fact that $1-\frac{t^2 \sigma^2}{2n} \geq 1/2$ and the bound \eqref{aurore}, we have
\[
I_1 \leq \e^{-\frac{t^2 \sigma^2}{2}} \left( \e^{\frac{c\abs{t}^3}{\sqrt{n}}} - 1 \right).
\]
Next, using the inequality $\e^{u}-1 \leq u \e^{u}$ for $u \geq 0$ and the fact that $\abs{t}/\sqrt{n} \leq \ee_0$,
\begin{equation}
I_1 \leq \e^{-\frac{t^2 \sigma^2}{2}} \frac{c}{\sqrt{n}} \abs{t}^3 \e^{c \ee_0 t^2}.
\label{I1bound001}
\end{equation}
Again, without loss of generality, the value of $\ee_0> 0$ can be chosen such that $c \ee_0^2 \leq \sigma^2/8$ (this have no impact on \eqref{aurore} which holds for any $[-\ee_0',\ee_0'] \subseteq [-\ee_0,\ee_0]$). Thus, from \eqref{I1bound001} it follows that
\begin{equation}
	\label{citron001}
	I_1 \leq \frac{c}{\sqrt{n}} \e^{-\frac{t^2 \sigma^2}{4}}.
\end{equation}
Using the inequalities $1-u \leq \e^{-u}$ for $u \in \bb R$ and $\ln(1-u) \geq -u-u^2$ for $u \leq 1$, we have
\begin{equation}
	\label{citron002}
	I_2 = \e^{-\frac{t^2 \sigma^2}{2}} - \left( 1-\frac{t^2 \sigma^2}{2n} \right)^n \leq \e^{-\frac{t^2 \sigma^2}{2}} - \e^{-\frac{t^2 \sigma^2}{2}-\frac{t^4 \sigma^4}{4n}} \leq \frac{t^4 \sigma^4}{4n}\e^{-\frac{t^2 \sigma^2}{2}} \leq \frac{c}{\sqrt{n}}\e^{-\frac{t^2 \sigma^2}{4}}.
\end{equation}
Putting together \eqref{citron}, \eqref{citron001} and \eqref{citron002}, we obtain that, for any $t/\sqrt{n} \in [-\ee_0,\ee_0]$,
\begin{equation}
	\label{semaphore003}
\abs{\lambda_{\frac{t}{\sqrt{n}}}^n - \e^{-\frac{t^2 \sigma^2}{2}}} \leq \frac{c}{\sqrt{n}}\e^{-\frac{t^2 \sigma^2}{4}}.
\end{equation}
In the same way, one can prove that
\begin{equation}
	\label{semaphore004}
	\abs{t}\abs{\lambda_{\frac{t}{\sqrt{n}}}^n} \leq \e^{-\frac{t^2 \sigma^2}{4}}.
\end{equation}

The right hand side in the assertion of the lemma can be bounded as follows: 
\[
\norm{\bf P_{\frac{t}{\sqrt{n}}}^n - \e^{-\frac{t^2 \sigma^2}{2}} \Pi}_{\scr C \to \scr C} \leq \abs{\lambda_{\frac{t}{\sqrt{n}}}^n} \norm{\Pi_{\frac{t}{\sqrt{n}}} - \Pi}_{\scr C \to \scr C} + \abs{\lambda_{\frac{t}{\sqrt{n}}}^n - \e^{-\frac{t^2 \sigma^2}{2}}} \norm{\Pi}_{\scr C \to \scr C} + \norm{Q_{\frac{t}{\sqrt{n}}}^n}_{\scr C \to \scr C}.
\]
Using \eqref{semaphore001}, \eqref{semaphore002}, \eqref{semaphore003} and \eqref{semaphore004}, we obtain that, for any $t/ \sqrt{n} \in [\ee_0,\ee_0]$,
\[
\norm{\bf P_{\frac{t}{\sqrt{n}}}^n - \e^{-\frac{t^2 \sigma^2}{2}} \Pi}_{\scr C \to \scr C} \leq \frac{c}{\sqrt{n}} \e^{-\frac{t^2 \sigma^2}{4}} + c \e^{-c n}.
\]
\end{proof}

\section{A non asymptotic local limit theorem}
\label{SecLLTWC}

In this section we establish a local limit theorem for the Markov walk jointly with the Markov chain. 
Our result is similar to that in 
Grama and Le Page \cite{grama_bounds_2017} where the case of sums of independent random variables is considered under the Cram\'er
condition.  
We refer to Guivarc'h and Hardy \cite{guivarch_theoremes_1988} for local limit theorem 
for a Markov chain with compact state space. 
In contrast to \cite{guivarch_theoremes_1988} our local limit theorem gives a control of the remainder term.

We first establish a local limit theorem for integrable functions with Fourier transforms with compact supports.
For any integrable  function $h$: $\bb R \to \bb R$ denote by $\hat{h}$ its Fourier transform:
\[
\hat{h}(t) = \int_{\bb R} \e^{-itu} h(u) \dd u, \quad \forall t \in \bb R.
\]
When $\hat{h}$ is integrable, by the inversion formula,
\[
h(u) = \frac{1}{2\pi} \int_{\bb R} \e^{itu} \hat{h}(t) \dd t, \quad \forall u \in \bb R.
\]
For any integrable functions $h$ and $g$, let
\[
h*g(u) = \int_{\bb R} h(v)g(u-v) \dd v
\]
be the convolution of $h$ and $g$. Denote by $\varphi_{\sigma}$ the density of the centred normal law with variance $\sigma^2$:
\begin{equation}
	\label{normal}
	\varphi_\sigma(u) = \frac{1}{\sqrt{2\pi}\sigma}e^{-\frac{u^2}{2\sigma^2}}, \quad \forall u \in \bb R.
\end{equation}

\begin{lemma}
\label{Buisson1}
Assume Hypotheses \ref{primitive}-\ref{nondegenere}.
For any $A > 0$, any integrable function $h$ on $\bb R$ whose Fourier transform $\hat{h}$ has a compact support included in $[-A,A]$, any real function $\psi$ defined on $\bb X$ and any $n \geq 1$,
\begin{align*}
	\underset{y\in \bb R}{\sup} \sqrt{n} &\abs{\bb E_x \left( h\left( y+S_n \right) \psi \left( X_n \right) \right) - h*\varphi_{\sqrt{n}\sigma}(y) \bs \nu \left( \psi \right)} \\
&\hspace{3cm} \leq \norm{\psi}_{\infty} \left( \frac{c}{\sqrt{n}} \norm{h}_{\LL^1} + \norm{\hat{h}}_{\LL^1} c_{A}\e^{-c_{A}n} \right).
\end{align*}
\end{lemma}

\begin{proof}
By the inversion formula and the Fubini theorem,
\begin{align*}
	I_0 &:= \sqrt{n}\abs{\bb E_x \left( h\left( y+S_n \right) \psi \left( X_n \right) \right) - h*\varphi_{\sqrt{n}\sigma}(y) \bs \nu \left( \psi \right)}\\
	&= \frac{\sqrt{n}}{2\pi}\abs{\bb E_x \left( \int_{\bb R} \e^{it\left( y+S_n \right)} \hat{h}(t) \dd t \psi \left( X_n \right) \right) - \int_{\bb R} \hat{h}(t) 
	\hat{\varphi}_{\sqrt{n}\sigma}(t) \e^{ity} \dd t \bs \nu \left( \psi \right)}\\
	&= 	\frac{\sqrt{n}}{2\pi}\abs{\int_{\bb R} \e^{ity} \left( \bf P_{t}^n \psi (x) - \e^{-\frac{t^2\sigma^2 n}{2}} \bs \nu \left( \psi \right) \right) \hat{h}(t) \dd t}.
\end{align*}
Since $\hat{h}( t ) = 0$ for any $t \notin [-A,A]$, we write
\begin{align}
	I_0 \leq\; &\underbrace{\frac{\sqrt{n}}{2\pi}\abs{\int_{\ee_0 \leq \abs{t}\leq A} \e^{ity} \left( \bf P_{t}^n \psi (x) - \e^{-\frac{t^2\sigma^2 n}{2}} \bs \nu \left( \psi \right) \right) \hat{h}(t) \dd t}}_{=:I_1} \nonumber\\
	&+ \underbrace{\frac{\sqrt{n}}{2\pi}\abs{\int_{\abs{t}\leq \ee_0} \e^{ity} \left( \bf P_{t}^n \psi (x) - \e^{-\frac{t^2\sigma^2 n}{2}} \bs \nu \left( \psi \right) \right) \hat{h}(t) \dd t}}_{=:I_2},
	\label{decI0}
\end{align}
where $\ee_0$ is defined by Lemma \ref{Cameleon}.

\textit{Bound of $I_1$.} By Lemma \ref{NOLA}, for any $\ee_0 \leq \abs{t} \leq A$, we have
\[
\norm{\bf P_{t}^n \psi}_{\infty} \leq \norm{\psi}_{\infty} c_{A,\ee_0} \e^{-c_{A,\ee_0}n}.
\]
Consequently,
\begin{align}
	I_1 &\leq \frac{\sqrt{n}}{2\pi} \left( \norm{\psi}_{\infty} c_{A,\ee_0} \e^{-c_{A,\ee_0}n} + \e^{-\frac{\ee_0^2 \sigma^2 n}{2}} \abs{\bs \nu(\psi)} \right) \norm{\hat{h}}_{\LL^1} \nonumber\\
	\label{MajI1LLT}
	&\leq \norm{\psi}_{\infty} \norm{\hat{h}}_{\LL^1} c_{A,\ee_0}\e^{-c_{A,\ee_0}n}.
\end{align}

\textit{Bound of $I_2$.} Substituting $s=t\sqrt{n}$, we write
\begin{align*}
	I_2 &= \frac{1}{2\pi}\abs{\int_{\abs{s}\leq \ee_0 \sqrt{n}} \e^{i\frac{sy}{\sqrt{n}}} \left( \bf P_{\frac{s}{\sqrt{n}}}^n \psi (x) - \e^{-\frac{s^2\sigma^2}{2}} \bs \nu \left( \psi \right) \right) \hat{h} \left( \frac{s}{\sqrt{n}} \right) \dd s} \\
	&\leq \frac{1}{2\pi}\int_{\abs{s}\leq \ee_0 \sqrt{n}} \abs{\bf P_{\frac{s}{\sqrt{n}}}^n \psi (x) - \e^{-\frac{s^2\sigma^2}{2}} \bs \nu \left( \psi \right)} \abs{\hat{h} \left( \frac{s}{\sqrt{n}} \right)} \dd s.
\end{align*}
By Lemma \ref{Cameleon}, for any $\abs{s}\leq \ee_0 \sqrt{n}$, we have 
\begin{align*}
	\abs{\bf P_{\frac{s}{\sqrt{n}}}^n \psi (x) - \e^{-\frac{s^2\sigma^2}{2}} \bs \nu \left( \psi \right)} &\leq \norm{\bf P_{\frac{s}{\sqrt{n}}}^n \left(\psi\right) - \e^{-\frac{s^2\sigma^2}{2}} \Pi \left( \psi \right)}_{\infty} \\
	&\leq \norm{\psi}_{\infty} \norm{\bf P_{\frac{s}{\sqrt{n}}}^n - \e^{-\frac{s^2\sigma^2}{2}} \Pi}_{\scr C \to \scr C} \\
	&\leq \norm{\psi}_{\infty} \left( \frac{c}{\sqrt{n}}\e^{-\frac{s^2 \sigma^2}{4}} + c \e^{-c n} \right).
\end{align*}
Therefore,
\begin{align}
	I_2 &\leq \norm{\psi}_{\infty} \left( \frac{c}{\sqrt{n}} \int_{\bb R} \e^{-\frac{s^2 \sigma^2}{4}} \norm{\hat{h}}_{\infty} \dd s + c \e^{-c n} \norm{\hat{h}}_{\LL^1} \right) \nonumber\\
	&\leq \norm{\psi}_{\infty} \left( \frac{c}{\sqrt{n}} \norm{h}_{\LL^1} + c \e^{-c n} \norm{\hat{h}}_{\LL^1} \right).
	\label{MajI2}
\end{align}
Putting together \eqref{decI0}, \eqref{MajI1LLT} and \eqref{MajI2}, concludes the proof.
\end{proof}

We extend the result of Lemma \ref{Buisson1} for any integrable function (with not necessarily integrable Fourier transform). As in Stone \cite{stone_local_1965}, we introduce the kernel $\kappa$ defined on $\bb R$ by
\[
\kappa (u) = \frac{1}{2\pi} \left( \frac{\sin \left( \frac{u}{2} \right)}{\frac{u}{2}} \right)^2, \quad \forall u \in \bb R^* \qquad \text{and} \qquad \kappa(0) = \frac{1}{2\pi}.
\]
The function $\kappa$ is integrable and its Fourier transform is given by
\[
\hat{\kappa}(t) = 1-\abs{t}, \quad \forall t \in [-1,1], \qquad \text{and} \qquad \hat{\kappa}(t) = 0 \quad \text{otherwise.}
\]
Note that
\[
\int_{\bb R} \kappa (u) \dd u  = \hat{\kappa}(0) = 1 = \int_{\bb R} \hat{\kappa}(t) \dd t.
\]
For any $\ee >0$, we define the function $\kappa_{\ee}$ on $\bb R$ by
\[
\kappa_{\ee} (u) = \frac{1}{\ee} \kappa \left( \frac{u}{\ee} \right).
\]
Its Fourier transform is given by $\hat{\kappa}_{\ee} (t) = \hat{\kappa}(\ee t)$. Note also that, for any $\ee > 0$, we have
\begin{equation}
	\label{phare}
	\int_{\abs{u} \geq \frac{1}{\ee}} \kappa (u) \dd u \leq \frac{1}{\pi} \int_{\frac{1}{\ee}}^{+\infty} \frac{4}{u^2} \dd u = \frac{4}{\pi} \ee.
\end{equation}

For any non-negative and locally bounded function $h$ defined on $\bb R$ and any $\ee >0$, let $\overline{h}_{\ee}$ and $\underline{h}_{\ee}$ be the "thickened" functions: for any $u \in \bb R$,
\[
\overline{h}_{\ee}(u) = \sup_{v \in [u-\ee,u+\ee]} h(v) \qquad \text{and} \qquad \underline{h}_{\ee}(u) = \inf_{v \in [u-\ee,u+\ee]} h(v).
\]
For any $\ee > 0$, denote by $\scr H_{\ee}$ the set of non-negative and locally bounded functions $h$ such that $h$, $\overline{h}_{\ee}$ and $\underline{h}_{\ee}$ are measurable from $\left( \bb R, \scr B \left( \bb R \right) \right)$ to $\left( \bb R_+, \scr B \left( \bb R_+ \right) \right)$ and Lebesgue-integrable (where $\scr B \left( \bb R \right)$, $\scr B \left( \bb R_+ \right)$ are the Borel $\sigma$-algebras).

\begin{lemma}
\label{soleil1}
For any function $h \in \scr H_{\ee}$, $\ee \in (0,1/4)$ and $u \in \bb R$,
\[
\underline{h}_{\ee}*\kappa_{\ee^2} (u) - \int_{\abs{v} \geq \ee} \underline{h}_{\ee} \left( u- v \right) \kappa_{\ee^2} (v) \dd v \leq h(u) \leq  \left( 1+4\ee \right) \overline{h}_{\ee}*\kappa_{\ee^2} (u).
\]
\end{lemma}

\begin{proof}
Note that for any $\abs{v} \leq \ee$ and $u\in \bb R$, we have $u\in [ u- v - \ee, u- v+\ee]$. So,
\begin{equation}
	\label{armure1}
	\underline{h}_{\ee} \left( u- v \right) \leq h (u) \leq \overline{h}_{\ee} \left( u- v \right).
\end{equation}
Using the fact that $\int_{\bb R} \kappa_{\ee^2} (u) \dd u = 1$ and \eqref{phare}, we write
\begin{align*}
	h (u) &= \int_{\abs{v} \leq \ee} h(u) \kappa_{\ee^2} (v) \dd v + h(u) \int_{\abs{v} \geq \ee}  \kappa_{\ee^2} (v) \dd v \\
	&\leq \int_{\abs{v} \leq \ee} \overline{h}_{\ee} \left( u- v \right) \kappa_{\ee^2} (v) \dd v + h(u) \frac{4}{\pi} \ee.
\end{align*}
Therefore,
\[
h(u) \left( 1- \frac{4}{\pi} \ee \right) \leq \int_{\bb R} \overline{h}_{\ee} \left( u-v \right) \kappa_{\ee^2} (v) \dd v = \overline{h}_{\ee}*\kappa_{\ee^2} (u).
\]
For any $\ee \in (0,1/4)$,
\[
h(u) \leq \frac{1}{1-2\ee} \overline{h}_{\ee}*\kappa_{\ee^2} (u) \leq \left( 1+4\ee \right) \overline{h}_{\ee}*\kappa_{\ee^2} (u).
\]
Moreover, from \eqref{armure1},
\begin{align*}
	h(u) &\geq \int_{\abs{v} \leq \ee} h(u) \kappa_{\ee^2} (v) \dd v \\
	&\geq \int_{\abs{v} \leq \ee} \underline{h}_{\ee} \left( u- v \right) \kappa_{\ee^2} (v) \dd v \\
	&= \underline{h}_{\ee}*\kappa_{\ee^2} (u) - \int_{\abs{v} \geq \ee} \underline{h}_{\ee} \left( u- v \right) \kappa_{\ee^2} (v) \dd v.
\end{align*}
\end{proof}

\begin{lemma}
\label{soleil2}
Let $\ee >0$ and $h \in \scr H_{\ee}$.
\begin{enumerate}[ref=\arabic*, leftmargin=*, label=\arabic*.]
	\item \label{soleil21} For any $y \in \bb R$ and $n\geq 1$,
	\[
	\sqrt{n} \left( \overline{h}_{\ee}*\kappa_{\ee^2} \right)*\varphi_{\sqrt{n}\sigma}(y) \leq \sqrt{n}\left( h*\varphi_{\sqrt{n}\sigma} \right)(y) + c \norm{\overline{h}_{2\ee} - h}_{\LL^1}  + c \ee \norm{h}_{\LL^1},
	\]
where $\varphi_{\sqrt{n}\sigma}(\cdot)$ is defined by \eqref{normal}.
	\item \label{soleil22} For any $y \in \bb R$ and $n\geq 1$,
	\[
	\sqrt{n}\left( \overline{h}_{\ee}*\kappa_{\ee^2} \right)*\varphi_{\sqrt{n}\sigma}(y) \leq c \norm{\overline{h}_{\ee}}_{\LL^1}.
	\]
	\item \label{soleil23} For any $y\in \bb R$ and $n\geq 1$,
	\[
	\sqrt{n} \left( \underline{h}_{\ee}*\kappa_{\ee^2} \right)*\varphi_{\sqrt{n}\sigma}(y) \geq \sqrt{n} \left(h*\varphi_{\sqrt{n}\sigma}\right)(y) - c \norm{h-\underline{h}_{2\ee}}_{\LL^1} - c \ee \norm{h}_{\LL^1}.
	\]
\end{enumerate}
\end{lemma}

\begin{proof} 
For any $\ee > 0$, $\abs{v} \leq \ee$ and $u\in \bb R$ it holds $[u-v-\ee,u-v+\ee] \subset [u-2\ee,u+2\ee]$. 
Therefore,
\begin{equation}
	\label{armure2}
	\underline{h}_{\ee}(u-v) \geq \underline{h}_{2\ee}(u) \qquad \text{and} \qquad \overline{h}_{\ee}(u-v) \leq \overline{h}_{2\ee}(u).
\end{equation}
Consequently, for any $u\in \bb R$,
\begin{align*}
	\overline{h}_{\ee}*\kappa_{\ee^2} (u) &\leq \overline{h}_{2\ee}(u) \int_{\abs{v}\leq \ee} \kappa_{\ee^2}(v) \dd v + \int_{\abs{v}\geq \ee} \overline{h}_{\ee}(u-v)\kappa_{\ee^2}(v) \dd v \\
	&\leq \overline{h}_{2\ee}(u) + \int_{\abs{v}\geq \ee} \overline{h}_{\ee}(u-v)\kappa_{\ee^2}(v) \dd v.
\end{align*}
From this, using the bound $\sqrt{n}\varphi_{\sqrt{n}\sigma}(\cdot) \leq 1/(\sqrt{2\pi}\sigma)$ and \eqref{phare}, we obtain that
\begin{align*}
	\sqrt{n} \left( \overline{h}_{\ee}*\kappa_{\ee^2} \right)*\varphi_{\sqrt{n}\sigma}(y) &\leq \sqrt{n} \left( \overline{h}_{2\ee}*\varphi_{\sqrt{n}\sigma} \right)(y) \\
	&\qquad+ \frac{1}{\sqrt{2\pi}\sigma} \int_{\bb R} \int_{\abs{v}\geq \ee} \overline{h}_{\ee}(u-v)\kappa_{\ee^2}(v) \dd v \dd u \\
	&= \sqrt{n} \left( \overline{h}_{2\ee}*\varphi_{\sqrt{n}\sigma} \right)(y) + \frac{2\sqrt{2}}{\pi^{3/2}\sigma} \ee \norm{\overline{h}_{\ee}}_{\LL^1}.
\end{align*}
Using again the bound $\sqrt{n}\varphi_{\sqrt{n}\sigma}(\cdot) \leq 1/(\sqrt{2\pi}\sigma)$, we get
\begin{align*}
	\sqrt{n} \left( \overline{h}_{\ee}*\kappa_{\ee^2} \right)*\varphi_{\sqrt{n}\sigma}(y) &\leq \sqrt{n} \left( h*\varphi_{\sqrt{n}\sigma} \right)(y) + \int_{\bb R}  \abs{\overline{h}_{2\ee}(u) - h(u)} \frac{\dd u}{\sqrt{2\pi}\sigma} + c \ee \norm{\overline{h}_{\ee}}_{\LL^1} \\
	&\leq \sqrt{n} \left( h*\varphi_{\sqrt{n}\sigma} \right)(y) + c \norm{\overline{h}_{2\ee} - h}_{\LL^1} + c \ee \norm{\overline{h}_{2\ee}}_{\LL^1} \\
	&\leq \sqrt{n} \left( h*\varphi_{\sqrt{n}\sigma} \right)(y) + \left( c+ c\ee \right) \norm{\overline{h}_{2\ee} - h}_{\LL^1} + c \ee \norm{h}_{\LL^1},
\end{align*}
which proves the claim \ref{soleil21}.

In the same way,
\[
\sqrt{n} \left( \overline{h}_{\ee}*\kappa_{\ee^2} \right)*\varphi_{\sqrt{n}\sigma}(y) \leq \frac{1}{\sqrt{2\pi}\sigma} \norm{\overline{h}_{\ee}*\kappa_{\ee^2}}_{\LL^1} = \frac{1}{\sqrt{2\pi}\sigma} \norm{\overline{h}_{\ee}}_{\LL^1},
\]
which establishes the claim \ref{soleil22}.

By \eqref{armure2} and \eqref{phare},
\[
\underline{h}_{\ee}*\kappa_{\ee^2}(u) \geq \underline{h}_{2\ee}(u)\int_{\abs{v}\leq \ee} \kappa_{\ee^2}(v) \dd v \geq \left( 1-\frac{4}{\pi} \ee \right) \underline{h}_{2\ee}(u).
\]
Integrating this inequality and using once again the bound $\sqrt{n} \varphi_{\sqrt{n}\sigma}(\cdot) \leq \frac{1}{\sqrt{2\pi}\sigma}$, we have
\begin{align*}
	\sqrt{n} \left( \underline{h}_{\ee}*\kappa_{\ee^2} \right)*\varphi_{\sqrt{n}\sigma}(y) &\geq \sqrt{n} \left( 1-\frac{4}{\pi} \ee \right) \underline{h}_{2\ee}*\varphi_{\sqrt{n}\sigma}(y) \\
	&\geq \sqrt{n} \left(\underline{h}_{2\ee}*\varphi_{\sqrt{n}\sigma}\right)(y) - \frac{4}{\pi} \ee \frac{1}{\sqrt{2\pi}\sigma} \norm{\underline{h}_{2\ee}}_{\LL^1}.
\end{align*}
Inserting $h$, we conclude that
\begin{align*}
	\sqrt{n} \left( \underline{h}_{\ee}*\kappa_{\ee^2} \right)*\varphi_{\sqrt{n}\sigma}(y) &\geq \sqrt{n} \left(h*\varphi_{\sqrt{n}\sigma}\right)(y) - \frac{1}{\sqrt{2\pi}\sigma} \norm{h-\underline{h}_{2\ee}}_{\LL^1} - c \ee \norm{\underline{h}_{2\ee}}_{\LL^1} \\
	&\geq \sqrt{n} \left(h*\varphi_{\sqrt{n}\sigma}\right)(y) - c \norm{h-\underline{h}_{2\ee}}_{\LL^1} - c \ee \norm{h}_{\LL^1}.
\end{align*}
\end{proof}

We are now equipped to prove a non-asymptotic theorem for a large class of functions $h$.
\begin{lemma}
\label{Bosquet}
Assume Hypotheses \ref{primitive}-\ref{nondegenere}.
Let $\ee \in (0,1/4)$. For any function $h \in \scr H_{\ee}$, any non-negative function $\psi \in \scr C$ and any $n \geq 1$,
\begin{align*}
	&\underset{x \in \bb X, \,y\in \bb R}{\sup} \sqrt{n} \abs{ \bb E_x \left( h\left( y+S_n \right) \psi \left( X_n \right) \right) - h*\varphi_{\sqrt{n}\sigma}(y) \bs \nu \left( \psi \right) } \\
&\qquad \leq c \norm{\psi}_{\infty} \left( \norm{h-\underline{h}_{2\ee}}_{\LL^1} + \norm{\overline{h}_{2\ee}-h}_{\LL^1} \right) + c \norm{\psi}_{\infty} \norm{\overline{h}_{2\ee}}_{\LL^1} \left( \frac{1}{\sqrt{n}} + \ee + c_{\ee}\e^{-c_{\ee}n} \right),
\end{align*}
where $\varphi_{\sqrt{n}\sigma}(\cdot)$ is defined by $\eqref{normal}$. Moreover,
\[
\underset{x \in \bb X, \,y\in \bb R}{\sup} \sqrt{n} \bb E_x \left( h\left( y+S_n \right) \psi \left( X_n \right) \right) \leq c \norm{\psi}_{\infty} \norm{\overline{h}_{2\ee}}_{\LL^1} \left( 1 + c_{\ee}\e^{-c_{\ee}n} \right).
\]
\end{lemma}

\begin{proof}
We prove upper and lower bounds for $\sqrt{n}\bb E_x \left( h\left( y+S_n \right) \psi \left( X_n \right) \right)$ from which the claim wills follow.

\textit{The upper bound.} By Lemma \ref{soleil1}, we have, for any $x\in \bb X$, $n\geq 1$, $y\in \bb R$ and $\ee\in(0,1/4)$,
\[
\bb E_x \left( h\left( y+S_n \right) \psi \left( X_n \right) \right) \leq \left( 1+4\ee \right) \bb E_x \left( \overline{h}_{\ee}*\kappa_{\ee^2}\left( y+S_n \right) \psi \left( X_n \right) \right)
\]
Since $\overline{h}_{\ee}$ is integrable, the function $u\mapsto \overline{h}_{\ee}*\kappa_{\ee^2}(u)$ is integrable and its Fourier transform $u\mapsto \hat{\overline{h}}_{\ee}(u) \hat{\kappa}_{\ee^2}(u)$ has a support included in $[-1/\ee^2,1/\ee^2]$. Consequently, by Lemma \ref{Buisson1},
\begin{align*}
	I_0 &:= \sqrt{n} \bb E_x \left( h\left( y+S_n \right) \psi \left( X_n \right) \right) \\
	&\leq \sqrt{n}\left( 1+4\ee \right)\left( \overline{h}_{\ee}*\kappa_{\ee^2} \right)*\varphi_{\sqrt{n}\sigma}(y) \bs \nu \left( \psi \right) \\
	&\qquad + 2\norm{\psi}_{\infty} \left( \frac{c}{\sqrt{n}} \norm{\overline{h}_{\ee}*\kappa_{\ee^2}}_{\LL^1} + \norm{\hat{\overline{h}}_{\ee} \hat{\kappa}_{\ee^2}}_{\LL^1} c_{\ee}\e^{-c_{\ee}n} \right).
\end{align*}
Using the points \ref{soleil21} and \ref{soleil22} of Lemma \ref{soleil2} and the fact that 
$\abs{\bs \nu \left( \psi \right)} \leq \norm{\psi}_{\infty}$,  
we deduce that
\begin{align*}
	I_0 &\leq \sqrt{n} \left( h*\varphi_{\sqrt{n}\sigma} \right)(y) \bs \nu \left( \psi \right) + \norm{\psi}_{\infty} \left( c \norm{\overline{h}_{2\ee} - h}_{\LL^1}  + c \ee \norm{h}_{\LL^1} \right) + 4\ee c \norm{\overline{h}_{\ee}}_{\LL^1} \norm{\psi}_{\infty} \\
	&\qquad + 2 \norm{\psi}_{\infty} \left( \frac{c}{\sqrt{n}} \norm{\overline{h}_{\ee}*\kappa_{\ee^2}}_{\LL^1} + \norm{\hat{\overline{h}}_{\ee} \hat{\kappa}_{\ee^2}}_{\LL^1} c_{\ee}\e^{-c_{\ee}n} \right).
\end{align*}
Note that $\norm{\overline{h}_{\ee}*\kappa_{\ee^2}}_{\LL^1} = \norm{\overline{h}_{\ee}}_{\LL^1}$ and 
\[
\norm{\hat{\overline{h}}_{\ee} \hat{\kappa}_{\ee^2}}_{\LL^1} \leq \norm{\overline{h}_{\ee}}_{\LL^1} \int_{\bb R} \hat{\kappa}_{\ee^2} (t) \dd t
=\norm{\overline{h}_{\ee}}_{\LL^1} \int_{\bb R} \hat{\kappa} (\ee^2 t) \dd t = \frac{1}{\ee^2} \norm{\overline{h}_{\ee}}_{\LL^1}.
\]
Consequently,
\begin{align}
	I_0 \leq \sqrt{n} \left( h*\varphi_{\sqrt{n}\sigma} \right)(y) \bs \nu \left( \psi \right) &+ c\norm{\psi}_{\infty} \norm{\overline{h}_{2\ee} - h}_{\LL^1} \nonumber\\
	&+ c\norm{\psi}_{\infty}  \norm{\overline{h}_{\ee}}_{\LL^1} \left( \frac{1}{\sqrt{n}} + \ee + c_{\ee}\e^{-c_{\ee}n} \right).
	\label{lac001}
\end{align}
From \eqref{lac001}, taking into account that $\sqrt{n} \left( h*\varphi_{\sqrt{n}\sigma} \right)(y) \leq c\norm{h}_{\LL^1}$, we deduce, in addition, that
\begin{equation}
	\label{lac002}
	I_0 \leq c\norm{\psi}_{\infty} \norm{\overline{h}_{2\ee}}_{\LL^1} \left( 1 + c_{\ee}\e^{-c_{\ee}n} \right).
\end{equation}

\textit{The lower bound.} By Lemma \ref{soleil1}, we write that
\begin{align}
	\label{baleine1}
	I_0 &\geq \underbrace{\sqrt{n} \bb E_x \left( \underline{h}_{\ee}*\kappa_{\ee^2} \left( y+S_n \right) \psi \left( X_n \right) \right)}_{=:I_1} \nonumber\\
	&\qquad- \underbrace{\sqrt{n} \bb E_x \left( \int_{\abs{v} \geq \ee} \underline{h}_{\ee} \left( y+S_n - v \right) \kappa_{\ee^2} (v) \dd v \psi \left( X_n \right) \right)}_{=:I_2}.
\end{align}

\textit{Bound of $I_1$.}
The Fourier transform of $\underline{h}_{\ee}*\kappa_{\ee^2}$ has a compact support included in $[-1/\ee^2,1/\ee^2]$. So by Lemma \ref{Buisson1},
\[
I_1 \geq \sqrt{n} \left( \underline{h}_{\ee}*\kappa_{\ee^2} \right)*\varphi_{\sqrt{n}\sigma}(y) \bs \nu \left( \psi \right) - \norm{\psi}_{\infty} \left( \frac{c}{\sqrt{n}}\norm{\underline{h}_{\ee}*\kappa_{\ee^2}}_{\LL^1} + \norm{\hat{\underline{h}_{\ee}*\kappa_{\ee^2}}}_{\LL^1} c_{\ee}\e^{-c_{\ee}n} \right),
\]
Using the point \ref{soleil23} of Lemma \ref{soleil2} and the fact that $\abs{\bs \nu \left( \psi \right)} \leq \norm{\psi}_{\infty}$,
\begin{align*}
	I_1 &\geq \sqrt{n} \left( h*\varphi_{\sqrt{n}\sigma} \right)(y) \bs \nu \left( \psi \right) - c \norm{\psi}_{\infty} \left( \norm{h-\underline{h}_{2\ee}}_{\LL^1} + \ee \norm{h}_{\LL^1} \right) \\
	&\qquad - \norm{\psi}_{\infty} \left( \frac{c}{\sqrt{n}} \norm{\underline{h}_{\ee}*\kappa_{\ee^2}}_{\LL^1} + \norm{\hat{\underline{h}_{\ee}*\kappa_{\ee^2}}}_{\LL^1} c_{\ee}\e^{-c_{\ee}n} \right).
\end{align*}
Since $\norm{\underline{h}_{\ee}*\kappa_{\ee^2}}_{\LL^1} = \norm{\underline{h}_{\ee}}_{\LL^1} \leq \norm{h}_{\LL^1}$ and since $\norm{\hat{\underline{h}_{\ee}*\kappa_{\ee^2}}}_{\LL^1} \leq \norm{\underline{h}_{\ee}}_{\LL^1} \norm{\hat{\kappa}_{\ee^2}}_{\LL^1} = \frac{1}{\ee^2}\norm{\underline{h}_{\ee}}_{\LL^1} \leq \frac{1}{\ee^2} \norm{h}_{\LL^1}$, we deduce that
\begin{align}
	I_1 \geq \sqrt{n} \left( h*\varphi_{\sqrt{n}\sigma} \right)(y) \bs \nu \left( \psi \right) &- c \norm{\psi}_{\infty} \norm{h-\underline{h}_{2\ee}}_{\LL^1} \nonumber\\
	&- c \norm{\psi}_{\infty} \norm{h}_{\LL^1} \left( \frac{1}{\sqrt{n}} + \ee + c_{\ee}\e^{-c_{\ee}n} \right).
	\label{baleine2}
\end{align}

\textit{Bound of $I_2$.} 
With the notation $g_{\ee,v}(u) = \underline{h}_{\ee} \left( u - v \right)$, we have
\[
I_2 = \int_{\abs{v} \geq \ee} \sqrt{n} \bb E_x \left( g_{\ee,v} \left( y+S_n \right) \psi \left( X_n \right) \right) \kappa_{\ee^2} (v) \dd v.
\]
Consequently, using \eqref{lac002}, we find that
\begin{align*}
I_2 \leq c\norm{\psi}_{\infty} \left( 1 + c_{\ee}\e^{-c_{\ee}n} \right) \int_{\abs{v} \geq \ee} \norm{\overline{\left( g_{\ee,v} \right)}_{2\ee}}_{\LL^1} \kappa_{\ee^2} (v) \dd v.
\end{align*}
Note that, for any $u$ and $v \in \bb R$, 
\[
\overline{\left( g_{\ee,v} \right)}_{2\ee} (u) = \sup_{w\in[u-2\ee,u+2\ee]} \underline{h}_{\ee} \left( w - v \right) \leq \sup_{w\in[u-2\ee,u+2\ee]} h \left( w - v \right) = \overline{h}_{2\ee}(u-v).
\]
So, $\norm{\overline{\left( g_{\ee,v} \right)}_{2\ee}}_{\LL^1} \leq \norm{\overline{h}_{2\ee}}_{\LL^1}$ and
\[
I_2 \leq c\norm{\psi}_{\infty} \norm{\overline{h}_{2\ee}}_{\LL^1} \left( 1 + c_{\ee}\e^{-c_{\ee}n} \right) \int_{\abs{v} \geq \ee} \kappa_{\ee^2} (v) \dd v.
\]
By \eqref{phare},
\begin{equation}
	\label{baleine3}
	I_2 \leq c\norm{\psi}_{\infty} \norm{\overline{h}_{2\ee}}_{\LL^1} \left( \ee + c_{\ee}\e^{-c_{\ee}n} \right).
\end{equation}
Putting together \eqref{baleine1}, \eqref{baleine2} and \eqref{baleine3}, we obtain that
\begin{align}
	I_0 \geq \sqrt{n} \left( h*\varphi_{\sqrt{n}\sigma} \right)(y) \bs \nu \left( \psi \right) &- c \norm{\psi}_{\infty} \norm{h-\underline{h}_{2\ee}}_{\LL^1} \nonumber\\
	&- c \norm{\psi}_{\infty} \norm{\overline{h}_{2\ee}}_{\LL^1} \left( \frac{1}{\sqrt{n}} + \ee + c_{\ee}\e^{-c_{\ee}n} \right).
	\label{lac003}
\end{align}

Putting together the upper bound \eqref{lac001} and the lower bound \eqref{lac003}, the first inequality of the lemma follows. 
The second inequality is proved in \eqref{lac002}.
\end{proof}

We now apply Lemma \ref{Bosquet} when the function $h$ is an indicator of an interval.

\begin{corollary}
\label{papillon}
Assume Hypotheses \ref{primitive}-\ref{nondegenere}.
For any $a>0$, $\ee \in (0,1/4)$, any non-negative function $\psi \in \scr C$ and any $n\geq 1$,
\begin{align*}
	&\sup_{x\in \bb X,\,y\in \bb R,\, z \geq 0} \sqrt{n} \abs{\bb E_x \left( \psi\left( X_n \right) \,;\, y+S_n \in [z,z+a] \right) - a\varphi_{\sqrt{n}\sigma} (z-y) \bs \nu \left( \psi \right)} \\
	&\hspace{1cm}\leq c (a+\ee) \norm{\psi}_{\infty} \left( \frac{1}{\sqrt{n}} + \frac{a}{n} + \ee + c_{\ee}\e^{-c_{\ee}n} \right),
\end{align*}
where $\varphi_{\sqrt{n}\sigma}(\cdot)$ is defined by $\eqref{normal}$. In particular, there exists $c> 0$ such that for any $a >0$,
\begin{equation}
	\label{papillon001}
\sup_{x\in \bb X,\, y\in \bb R,\, z \geq 0} \sqrt{n} \bb E_x \left( \psi\left( X_n \right) \,;\, y+S_n \in [z,z+a] \right) \leq c (1+a^2) \norm{\psi}_{\infty}.
\end{equation}
\end{corollary}

\begin{proof}
Let $z \geq 0$, $a>0$, $\ee \in (0,1/4)$. For any $y \in \bb R$ set  
\[
h(y) = \bbm 1_{[z,z+a]}(y).
\]
It is clear that 
\[
\overline{h}_{\ee}(y) = \bbm 1_{[z-\ee,z+a+\ee]}(y) \qquad \text{and} \qquad \underline{h}_{\ee}(y) = \bbm 1_{[z+\ee,z+a-\ee]}(y),
\]
where by convention $\bbm 1_{[z+\ee,z+a-\ee]}(y) = 0$ when $a \leq 2\ee$. 
It is also easy to see that
\[
\norm{h-\underline{h}_{2\ee}}_{\LL^1} = \norm{\overline{h}_{2\ee}-h}_{\LL^1} = 4\ee  \qquad \text{and} \qquad \norm{\overline{h}_{2\ee}}_{\LL^1} = a+4\ee.
\] 
Taking into account these last equalities and using Lemma \ref{Bosquet}, we find that
\begin{align}
	&\abs{\bb E_x \left( \psi\left( X_n \right) \,;\, y+S_n \in [z,z+a] \right) - \bbm 1_{[z,z+a]}*\varphi_{\sqrt{n}\sigma}(y) \bs \nu \left( \psi \right)} \nonumber\\
	&\hspace{2cm} \leq c (a+\ee) \norm{\psi}_{\infty} \left( \frac{1}{\sqrt{n}} + \ee + c_{\ee}\e^{-c_{\ee}n} \right).
	\label{voiture1}
\end{align}
Moreover, the convolution $\bbm 1_{[z,z+a]}*\varphi_{\sqrt{n}\sigma}$ is equal to
\[
\bbm 1_{[z,z+a]}*\varphi_{\sqrt{n}\sigma}(y) = \int_{\bb R} \bbm 1_{\left\{ z \leq y-u \leq z+a \right\}} \frac{e^{-\frac{u^2}{2n\sigma^2}}}{\sqrt{2\pi n}\sigma} \dd u = \Phi_{\sqrt{n}\sigma}(y-z) - \Phi_{\sqrt{n}\sigma}(y-z-a),
\]
where $\Phi_{\sqrt{n}\sigma}(t) = \int_{-\infty}^{t} \frac{e^{-\frac{u^2}{2n\sigma^2}}}{\sqrt{2\pi n}\sigma} \dd u$ is the distribution function of the centred normal law of variance $n\sigma^2$. 
By the Taylor-Lagrange formula, there exists $\xi \in (y-z-a,y-z)$ such that
\[
\Phi_{\sqrt{n}\sigma}(y-z-a) = \Phi_{\sqrt{n}\sigma}(y-z) - a\varphi_{\sqrt{n}\sigma}(y-z) + \frac{a^2}{2} \varphi_{\sqrt{n}\sigma}'(\xi).
\]
Using the fact that $\sup_{u\in \bb R} \abs{u}\e^{-u^2} \leq c$,
\begin{equation}
	\label{voiture2}
	\abs{\bbm 1_{[z,z+a]}*\varphi_{\sqrt{n}\sigma}(y) - a\varphi_{\sqrt{n}\sigma}(z-y)} \leq \frac{c a^2}{n}.
\end{equation}
Putting together \eqref{voiture1} and \eqref{voiture2}, we conclude that
\begin{align*}
	&\abs{\bb E_x \left( \psi\left( X_n \right) \,;\, y+S_n \in [z,z+a] \right) - a\varphi_{\sqrt{n}\sigma}(z-y) \bs \nu \left( \psi \right)} \\
	&\hspace{2cm} \leq c (a+\ee) \norm{\psi}_{\infty} \left( \frac{1}{\sqrt{n}} + \frac{a}{n} + \ee + c_{\ee}\e^{-c_{\ee}n} \right).
\end{align*}
\end{proof}

\section{Auxiliary bounds} \label{auxbounds}

We state two bounds on the expectation $\bb E_x \left( \psi(X_n) \,;\, y+S_n \in [z,z+a] \,,\, \tau_y > n \right)$. The first one is of order $1/n$ and independent of $z$. 
Then we reverse the Markov chain to improve it to a bound of order  $1/n^{3/2}.$  
 We refer to Denisov and Wachtel \cite{denisov_random_2015} for related results in the case of lattice valued independent random variables. 

\begin{lemma}
\label{gorille001}
Assume Hypotheses \ref{primitive}-\ref{nondegenere}. There exists $c > 0$ such that for any $a >0$, non-negative function $\psi \in \scr C$, $y \in \bb R$ and $n \geq 1$
\[
\sup_{x\in \bb X,\, z \geq 0} \bb E_x \left( \psi \left( X_{n} \right) \,;\, y+S_{n} \in [z,z+a] \,,\, \tau_y > n \right) \leq \frac{c}{n} \norm{\psi}_{\infty} (1+a^2) \left(1 + \max(y,0) \right).
\]
\end{lemma}

\begin{proof}
We split the time $n$ into two parts $k := \pent{n/2}$ and $n-k$. By the Markov property,
\begin{align*}
	E_0 &:= \bb E_x \left( \psi \left( X_n \right) \,;\, y+S_n \in [z,z+a] \,,\, \tau_y > n \right) \\
	&= \sum_{x'\in \bb X} \int_{0}^{+\infty} \bb E_{x'} \left( \psi \left( X_k \right) \,;\, y'+S_k \in [z,z+a] \,,\, \tau_{y'} > k \right) \\
	&\hspace{3cm}\times \bb P_x \left( X_{n-k} = x' \,,\, y+S_{n-k} \in \dd y' \,,\, \tau_y > n-k \right) \\
	&\leq \sum_{x'\in \bb X} \int_{0}^{+\infty} \bb E_{x'} \left( \psi \left( X_k \right) \,;\, y'+S_k \in [z,z+a] \right) \\
	&\hspace{3cm}\times \bb P_x \left( X_{n-k} = x' \,,\, y+S_{n-k} \in \dd y' \,,\, \tau_y > n-k \right).
\end{align*}
Using the uniform bound \eqref{papillon001} in Corollary \ref{papillon}, we obtain that
\[
E_0 \leq \frac{c \norm{\psi}_{\infty}}{\sqrt{k}} (1+a^2) \bb P_x \left( \tau_y > n-k \right).
\]
By the point \ref{RESUII002} of Proposition \ref{RESUII}, we get
\[
E_0 \leq \frac{c \norm{\psi}_{\infty} (1+a^2) \left( 1+\max(y,0) \right)}{\sqrt{k}\sqrt{n-k}}.
\]
Since $n-k \geq n/2$ and $k \geq n/4$ for any $n \geq 4$, the lemma is proved (the case when $n\leq 4$ is trivial).
\end{proof}

\begin{lemma}
\label{gorille002}
Assume Hypotheses \ref{primitive}-\ref{nondegenere}. There exists $c > 0$ such that for any $a >0$, non-negative function $\psi \in \scr C$, $y \in \bb R$, $z \geq 0$ and $n \geq 1$
\[
\sup_{x\in \bb X} \bb E_x \left( \psi \left( X_{n} \right) \,;\, y+S_{n} \in [z,z+a] \,,\, \tau_y > n \right) \leq \frac{c \norm{\psi}_{\infty}}{n^{3/2}} (1+a^3)\left( 1+z \right)\left( 1+\max(y,0) \right).
\]
\end{lemma}

\begin{proof}
Set again $k=\pent{n/2}$. By the Markov property
\begin{align}
	E_0 &:= \bb E_x \left( \psi \left( X_n \right) \,;\, y+S_n \in [z,z+a] \,,\, \tau_y > n \right) \nonumber\\
	\label{brebis}
	&= \sum_{x'\in \bb X} \int_{0}^{+\infty} \underbrace{\bb E_{x'} \left( \psi \left( X_k \right) \,;\, y'+S_k \in [z,z+a] \,,\, \tau_{y'} > k \right)}_{=:E_0'} \\
	&\hspace{3cm}\times \bb P_x \left( X_{n-k} = x' \,,\, y+S_{n-k} \in \dd y' \,,\, \tau_y > n-k \right). \nonumber
\end{align}
Using Lemma \ref{duality} with $\mathfrak{m} = \bs \delta_{x'}$ and
\[
	F(x_1,\dots,x_k) = \psi(x_k) \bbm 1_{\left\{ y'+f(x_1) \dots+f(x_k) \in [z,z+a] \,,\, \forall i \in \{ 1, \dots, k \},\, y'+f(x_1)+\dots+f(x_i) > 0 \right\}},
\]
we have
\begin{align*}
	E_0' &= \bb E_{\bs \nu}^* \left( \psi \left( X_1^* \right) \frac{\bbm 1_{\{x'\}}\left( X_{k+1}^* \right)}{\bs \nu \left( X_{k+1}^* \right)} \,;\, y'+f\left( X_k^* \right)+ \dots + f\left( X_1^* \right) \in [z,z+a] \,,\, \right. \\
	&\qquad\left. \phantom{\frac{\left( X_{n_2}^* \right)}{\bs \nu \left( X_{n_2}^* \right)}} \forall i \in \{ 1, \dots, k \},\, y'+f\left( X_k^* \right)+\dots+f \left(X_{k-i+1}^* \right) > 0 \right).
\end{align*}
By the Markov property,
\begin{align*}
	E_0' &= \bb E_{\bs \nu}^* \left( \psi \left( X_1^* \right) \psi_{x'}^*\left( X_k^* \right) \,;\, y'+f\left( X_k^* \right)+ \dots +f\left( X_1^* \right) \in [z,z+a] \,,\, \right. \\
	&\hspace{3cm} \left. \forall i \in \{ 1, \dots, k \},\, y'+f\left( X_k^* \right)+\dots+f \left(X_{k-i+1}^* \right) > 0 \right).
\end{align*}
where
\begin{equation}
	\label{hirondelle}
	\psi_{x'}^*(x^*) = \bb E_{x^*}^* \left( \frac{\bbm 1_{\{x'\}}\left( X_1^* \right)}{\bs \nu \left( X_1^* \right)} \right) = \frac{ \bf P^*(x^*,x')}{\bs \nu(x')} = \frac{ \bf P(x',x^*)}{\bs \nu(x^*)} \leq \frac{1}{\inf_{x\in \bb X} \bs \nu(x)}.
\end{equation}
On the event $\left\{ y'+f\left( X_k^* \right)+ \dots + f\left( X_1^* \right) \in \left[ z,z+a \right] \right\} = \left\{ z+a + S_k^* \in \left[ y',y'+a \right] \right\}$, we have
\begin{align*}
	&\left\{ \forall i \in \{ 1, \dots, k \},\, y'+f\left( X_k^* \right)+\dots+f \left( X_{k-i+1}^* \right) > 0, y' >0 \right\} \\
	&\qquad \subset \left\{ \forall i \in \{ 1, \dots, k-1 \},\, z+a-f \left( X_{k-i}^* \right)-\dots-f\left( X_1^* \right) > 0, z+a+S_k^* > 0 \right\} \\
	&\qquad = \left\{ \tau_{z+a}^* >  k \right\}.
\end{align*}
So, for any $y' > 0$,
\[
E_0' \leq c\norm{\psi}_{\infty} \bb P_{\bs \nu}^* \left( z+a+S_k^* \in [y',y'+a] \,,\, \tau_{z+a}^* > k \right).
\]
Using Lemma \ref{gorille001} we have uniformly in $y' >0$,
\begin{equation}
	\label{rossignol}
	E_0' \leq \frac{c \norm{\psi}_{\infty}}{k} (1+a^2) \left( 1+\max(z+a,0) \right) \leq \frac{c \norm{\psi}_{\infty}}{k} (1+a^3) \left( 1+z \right).
\end{equation}
Putting together \eqref{rossignol} and \eqref{brebis} and using the point \ref{RESUII002} of Proposition \ref{RESUII},
\[
	E_0 \leq \frac{c \norm{\psi}_{\infty}}{k} (1+a^3) \left( 1+z \right) \bb P_x \left( \tau_y > n-k \right) \leq \frac{c \norm{\psi}_{\infty}}{k\sqrt{n-k}} (1+a^3) \left( 1+z \right)\left( 1+\max(y,0) \right).
\]
Since $n-k \geq n/2$ and $k \geq n/4$ for any $n \geq 4$, the lemma is proved.
\end{proof}

\section{Proof of Theorem \ref{GNLLT}}
\label{heron}

The aim of this section is to bound
\begin{equation}
	\label{otarie}
	E_0 := \bb E_x \left( \psi \left( X_{n} \right) \,;\, y+S_{n} \in [z,z+a] \,,\, \tau_y > n \right)
\end{equation}
uniformly in the end point $z$. The point is to split the time $n$ into $n=n_1+n_2$, where $n_2 = \pent{\ee^3 n}$ and $n_1 = n- \pent{\ee^3 n}$, and $\ee \in (0,1)$. Using the Markov property, we shall bound the process between $n_1$ and $n$ by the local limit theorem (Corollary \ref{papillon}) and between $1$ and $n_1$ by the integral theorem (Proposition \ref{RESUIII}). Following this idea we write
\begin{align}
	E_0 &= \underbrace{\bb E_x \left( \psi \left( X_{n} \right) \,;\, y+S_{n} \in [z,z+a] \,,\, \tau_y > n_1 \right)}_{=:E_1} \nonumber\\
	&\qquad - \underbrace{\bb E_x \left( \psi \left( X_{n} \right) \,;\, y+S_{n} \in [z,z+a] \,,\, n_1 < \tau_y \leq n \right)}_{=:E_2}.
	\label{chatLLT}
\end{align}

For the ease of reading the bounds of $E_1$ and $E_2$ are given in separate sections.

\subsection{Control of \texorpdfstring{$E_1$}{}}
\label{perroquet}

\begin{lemma}
\label{DOSP}
Assume Hypotheses \ref{primitive}-\ref{nondegenere}.
For any $a>0$ and $\ee \in (0, 1/4)$ there exist $c = c_a >0$ depending only on $a$  and $c_{\ee}>0$ such that for any non-negative function $\psi \in \scr C$, any $y \in \bb R$ and $n\ \in \bb N$, such that $\ee^3 n\geq1$ we have
\begin{align*}
	\sup_{x\in \bb X,z\geq 0} &n\abs{E_1 - \frac{a}{\sqrt{n_2}\sigma} \bs \nu \left( \psi \right) \bb E_x \left( \varphi \left( \frac{y-z+S_{n_1}}{\sqrt{n_2}\sigma} \right) \,;\, \tau_y > n_1 \right)} \\
	&\hspace{4cm} \leq c \left(1 + \max(y,0) \right) \norm{\psi}_{\infty} \left( \ee + \frac{c_{\ee}}{\sqrt{n}} \right).
\end{align*}
where $E_1 = \bb E_x \left( \psi \left( X_{n} \right) \,;\, y+S_{n} \in [z,z+a] \,,\, \tau_y > n_1 \right)$, $n_2= \pent{\ee^3 n}$, $n_1 = n- \pent{\ee^3 n}$ and $\varphi(t) = \e^{-\frac{t^2}{2}}/\sqrt{2\pi}$.
\end{lemma}

\begin{proof}
By the Markov property,
\begin{align}
	E_1 &= \sum_{x' \in \bb X} \int_{0}^{+\infty} \underbrace{\bb E_{x'} \left( \psi\left( X_{n_2} \right) \,;\, y'+S_{n_2} \in [z,z+a] \right)}_{=:E_1'} \nonumber\\
	&\hspace{5cm} \times \bb P_x \left( y+S_{n_1} \in \dd y' \,,\, X_{n_1} = x' \,,\, \tau_y > n_1 \right).
	\label{canard1}
\end{align}
From now on we consider that the real $a>0$ is fixed. By Corollary \ref{papillon}, for any $\ee^{5/2} \leq \ee \in (0,1/4)$,
\begin{align*}
	\sqrt{n_2} \abs{E_1' - a\varphi_{\sqrt{n_2}\sigma} (z-y') \bs \nu \left( \psi \right)} \leq c \norm{\psi}_{\infty} \left( \frac{1}{\sqrt{n_2}} + \ee^{5/2} + c_{\ee}\e^{-c_{\ee}n_2} \right),
\end{align*}
with $c$ depending only on $a$. Consequently, using \eqref{canard1} and the fact that $n_2 = \pent{\ee^3 n} \geq c_{\ee} n$,
\begin{align*}
	&\abs{E_1 - a\bs \nu \left( \psi \right)  \bb E_x \left( \varphi_{\sqrt{n_2}\sigma} \left( y-z+S_{n_1} \right) \,;\, \tau_y > n_1 \right)} \\
	&\qquad \leq \frac{c\norm{\psi}_{\infty}}{\sqrt{n_2}} \left( \frac{c_{\ee}}{\sqrt{n}} + \ee^{5/2} + c_{\ee}\e^{-c_{\ee}n} \right) \bb P_x \left( \tau_y > n_1 \right).
\end{align*}
Therefore, by \eqref{normal} and the point \ref{RESUII002} of Proposition \ref{RESUII}, we obtain that
\begin{align*}
	&\abs{E_1 - \frac{a}{\sqrt{n_2}\sigma} \bs \nu \left( \psi \right) \bb E_x \left( \varphi \left( \frac{y-z+S_{n_1}}{\sqrt{n_2}\sigma} \right) \,;\, \tau_y > n_1 \right)} \leq c \norm{\psi}_{\infty} \frac{1 + \max(y,0)}{\sqrt{n_2}\sqrt{n_1}} \left( \frac{c_{\ee}}{\sqrt{n}} + \ee^{5/2} \right).
\end{align*}
Since $n_2 \geq \ee^3 n \left( 1 - \frac{1}{\ee^3 n} \right)$ and $n_1 \geq \frac{n}{2}$, we have
\begin{align*}
	c \norm{\psi}_{\infty} \frac{1 + \max(y,0)}{\sqrt{n_2}\sqrt{n_1}} \left( \frac{c_{\ee}}{\sqrt{n}} + \ee^{5/2} \right) &\leq c \norm{\psi}_{\infty} \frac{1 + \max(y,0)}{\ee^{3/2} n} \left( 1 + \frac{c_{\ee}}{n} \right)  \left( \frac{c_{\ee}}{\sqrt{n}} + \ee^{5/2} \right) \\
	&\leq c \norm{\psi}_{\infty} \frac{1 + \max(y,0)}{n} \left( \ee + \frac{c_{\ee}}{\sqrt{n}} \right)
\end{align*}
and the lemma follows.
\end{proof}

To find the limit behaviour of $E_1$, we will develop $\frac{1}{\sqrt{n_2}}\bb E_x \left( \varphi \left( \frac{y+S_{n_1}-z}{\sqrt{n_2}\sigma} \right) \,;\, \tau_y > n_1 \right)$. To this aim, we prove the following lemma which we will apply first with the standard normal density function $\varphi$, and later on with the Rayleigh density $\varphi_+$.
\begin{lemma}
\label{couleuvre}
Assume Hypotheses \ref{primitive}-\ref{nondegenere}.
Let $\Psi$ : $\bb R \to \bb R$ be a non-negative derivable function such that $\Psi(t) \to 0$ as $t \to +\infty$. Moreover we suppose that $\Psi'$ is a continuous function on $\bb R$ such that $\max(\abs{\Psi(t)},\abs{\Psi'(t)}) \leq c \e^{-\frac{t^2}{4}}$. There exists $\ee_0 \in (0,1/2)$ such that for any $\ee \in (0,\ee_0)$, $y \in \bb R$, $m_1 \geq 1$ and $m_2 \geq 1$, we have
\begin{align*}
	\sup_{x\in \bb X, \, z \geq 0} &\abs{\bb E_x \left( \Psi \left( \frac{y+S_{m_1}-z}{\sqrt{m_2}\sigma} \right) \,;\, \tau_y > m_1 \right) - \frac{2V(x,y)}{\sqrt{2\pi m_1}\sigma} \int_{0}^{+\infty} \Psi \left( \sqrt{\frac{m_1}{m_2}}t- \frac{z}{\sqrt{m_2}\sigma} \right) \varphi_+ (t) \dd t} \\
	&\qquad \leq c_{\ee} \frac{\left( 1+\max(y,0) \right)^2}{m_1^{\ee} \sqrt{m_2}} + c \frac{1+\max(y,0)}{\sqrt{m_1}} \left( \e^{-c\frac{m_1}{m_2} } + \ee^4 \right),
\end{align*}
where $\varphi_+(t) = t\e^{-\frac{t^2}{2}}$.
\end{lemma}

\begin{proof} Let $x \in \bb X$, $y \in \bb R$, $z \geq 0$, $m_1 \geq 1$ and $m_2 \geq 1$ and fix $\ee_1 \in (0,1)$. We consider two cases. Assume first that $z \leq \sqrt{m_1}\sigma/\ee_1$. Using the regularity of the function $\Psi$, we note that
\begin{align*}
	J_0 &:= \bb E_x \left( \Psi \left( \frac{y+S_{m_1}-z}{\sqrt{m_2}\sigma} \right) \,;\, \tau_y > m_1 \right) \\
	&= -\int_0^{+\infty} \sqrt{\frac{m_1}{m_2}}  \Psi' \left( \sqrt{\frac{m_1}{m_2}} t - \frac{z}{\sqrt{m_2}\sigma} \right) \bb P_x \left( \frac{y+S_{m_1}}{\sqrt{m_1}\sigma} \leq t \,,\, \tau_y > m_1 \right) \dd t.
\end{align*}
Denote by $J_1$ the following integral:
\begin{equation}
	\label{fourmi}
	J_1 := -\frac{2V(x,y)}{\sqrt{2\pi m_1}\sigma} \int_0^{+\infty} \sqrt{\frac{m_1}{m_2}}  \Psi' \left( \sqrt{\frac{m_1}{m_2}} t - \frac{z}{\sqrt{m_2}\sigma} \right) \left( 1 - \e^{-\frac{t^2}{2}} \right) \dd t.
\end{equation}
Using the point \ref{RESUIII002} of Proposition \ref{RESUIII}, with $t_0 = 2/\ee_1$, there exists $\ee_0 > 0$ such that for any $\ee \in (0,\ee_0)$,
\begin{align*}
	\abs{J_0 - J_1} &\leq c_{\ee,\ee_1} \frac{\left( 1+\max(y,0) \right)^2}{m_1^{1/2+\ee}} \int_0^{\frac{2}{\ee_1}} \sqrt{\frac{m_1}{m_2}}  \abs{\Psi' \left( \sqrt{\frac{m_1}{m_2}} t - \frac{z}{\sqrt{m_2}\sigma} \right)} \dd t \\
	&\qquad+ \left( \frac{2V(x,y)}{\sqrt{2\pi m_1}\sigma} + \bb P_x \left( \tau_y > m_1 \right) \right) \int_{\frac{2}{\ee_1}}^{+\infty} \sqrt{\frac{m_1}{m_2}}  \abs{\Psi' \left( \sqrt{\frac{m_1}{m_2}} t - \frac{z}{\sqrt{m_2}\sigma} \right)} \dd t.
\end{align*}
Using the point \ref{RESUI002} of Proposition \ref{RESUI} and the point \ref{RESUII002} of Proposition \ref{RESUII}, with $\norm{\Psi'}_{\infty} = \sup_{t\in \bb R} \abs{\Psi'(t)}$,
\begin{align*}
	\abs{J_0 - J_1} &\leq c_{\ee,\ee_1} \frac{\left( 1+\max(y,0) \right)^2}{m_1^{\ee} \sqrt{m_2}} \norm{\Psi'}_{\infty} + c \frac{1+\max(y,0)}{\sqrt{m_1}} \sqrt{\frac{m_1}{m_2}} \int_{\frac{2}{\ee_1}}^{+\infty} \e^{-\frac{ \left(\sqrt{\frac{m_1}{m_2}} t - \frac{z}{\sqrt{m_2}\sigma} \right)^2}{4} } \dd t \\
	&\leq c_{\ee,\ee_1} \frac{\left( 1+\max(y,0) \right)^2}{m_1^{\ee} \sqrt{m_2}} + c \frac{1+\max(y,0)}{\sqrt{m_1}}  \int_{\sqrt{\frac{m_1}{m_2}} \left( \frac{2}{\ee_1}-\frac{z}{\sqrt{m_1}\sigma} \right)}^{+\infty} \e^{-\frac{s^2}{4} } \dd s.
\end{align*}
Since $z \leq \frac{\sqrt{m_1}\sigma}{\ee_1}$, we have $\frac{2}{\ee_1} - \frac{z}{\sqrt{m_1}\sigma} \geq \frac{1}{\ee_1} \geq 1$ and so
\begin{equation}
	\abs{J_0 - J_1} \leq c_{\ee,\ee_1} \frac{\left( 1+\max(y,0) \right)^2}{m_1^{\ee} \sqrt{m_2}} + c \frac{1+\max(y,0)}{\sqrt{m_1}} \e^{-\frac{m_1}{8m_2} } \int_{\bb R} \e^{-\frac{s^2}{8} } \dd s.
	\label{bourgeon}
\end{equation}
Moreover, by the definition of $J_1$ in \eqref{fourmi}, we have
\begin{align}
	J_1 &= \frac{2V(x,y)}{\sqrt{2\pi m_1}\sigma} \left[ - \Psi \left( \sqrt{\frac{m_1}{m_2}} t - \frac{z}{\sqrt{m_2}\sigma} \right) \left( 1 - \e^{-\frac{t^2}{2}} \right) \right]_{t=0}^{t=+\infty} \nonumber\\
	&\qquad + \frac{2V(x,y)}{\sqrt{2\pi m_1}\sigma} \int_{0}^{+\infty}  \Psi \left( \sqrt{\frac{m_1}{m_2}} t - \frac{z}{\sqrt{m_2}\sigma} \right) t \e^{-\frac{t^2}{2}} \dd t \nonumber\\
	&=\frac{2V(x,y)}{\sqrt{2\pi m_1}\sigma} \int_{0}^{+\infty}  \Psi \left( \sqrt{\frac{m_1}{m_2}} t - \frac{z}{\sqrt{m_2}\sigma} \right) \varphi_+(t) \dd t.
	\label{bourgeon001}
\end{align}

Now, assume that $z > \frac{\sqrt{m_1}\sigma}{\ee_1}$. We write
\begin{align*}
	J_0 &\leq c\bb E_x \left( \e^{-\frac{\left( y+S_{m_1}-z \right)^2}{4m_2\sigma^2}} \,;\, y+S_{m_1} \leq \frac{\sqrt{m_1}\sigma}{2\ee_1} \,,\, \tau_y > m_1 \right) \\
	&\hspace{5cm} + \norm{\Psi}_{\infty} \bb P_x \left( y+S_{m_1} > \frac{\sqrt{m_1}\sigma}{2\ee_1} \,,\, \tau_y > m_1 \right) \\
	&\leq c \e^{-\frac{m_1}{16m_2\ee_1^2}} \bb P_x \left( \tau_y > m_1 \right) + \norm{\Psi}_{\infty} \frac{2\ee_1}{\sqrt{m_1}\sigma} \bb E_x \left( y+S_{m_1}  \,;\, \tau_y > m_1 \right).
\end{align*}
Using the points \ref{RESUI003} and \ref{RESUI001} of Proposition \ref{RESUI}, we can verify that
\[
\bb E_x \left( y+S_{m_1}  \,;\, \tau_y > m_1 \right) \leq \bb E_x \left( 2V \left(y+S_{m_1}, X_{m_1} \right) + c  \,;\, \tau_y > m_1 \right)  \leq 2V(x,y) + c. 
\]
So by the point \ref{RESUII002} of Proposition \ref{RESUII} and the point \ref{RESUI002} of Proposition \ref{RESUI},
\[
J_0 \leq c \frac{1+\max(y,0)}{\sqrt{m_1}} \e^{-\frac{cm_1}{m_2}} + \frac{c\ee_1}{\sqrt{m_1}} \left( 1+\max(y,0) \right).
\]
In the same way,
\begin{align*}
	J_1 &= \frac{2V(x,y)}{\sqrt{2\pi m_1}\sigma} \int_{0}^{+\infty}  \Psi \left( \sqrt{\frac{m_1}{m_2}} t - \frac{z}{\sqrt{m_2}\sigma} \right) \varphi_+(t) \dd t \\
	&\leq \frac{c\left( 1+\max(y,0) \right)}{\sqrt{m_1}} \left[ \int_0^{\frac{1}{2\ee_1}}  \e^{-\frac{m_1}{4m_2} \left( t-\frac{z}{\sqrt{m_1}\sigma} \right)^2} \varphi_+(t) \dd t + \norm{\Psi}_{\infty} \int_{\frac{1}{2\ee_1}}^{+\infty} t\e^{-\frac{t^2}{2}} \dd t \right] \\
	&\leq \frac{c\left( 1+\max(y,0) \right)}{\sqrt{m_1}} \left[  \e^{-\frac{m_1}{16m_2\ee_1^2}}  \int_0^{+\infty} \varphi_+(t) \dd t + \norm{\Psi}_{\infty} \e^{-\frac{1}{16 \ee_1^2}} \int_0^{+\infty} t\e^{-\frac{t^2}{4}} \dd t \right] \\
	&\leq \frac{c\left( 1+\max(y,0) \right)}{\sqrt{m_1}} \left(  \e^{-\frac{cm_1}{m_2}} + \e^{-\frac{c}{\ee_1^2}} \right).
\end{align*}
From the last two bounds it follows that for any $z > \frac{\sqrt{m_1}\sigma}{\ee_1}$,
\begin{equation}
	\label{bourgeon002}
	\abs{J_0 - J_1} \leq J_0 + J_1 \leq \frac{c\left( 1+\max(y,0) \right)}{\sqrt{m_1}} \left( \e^{-\frac{cm_1}{m_2}} + \ee_1 \right).
\end{equation}
Putting together \eqref{bourgeon001}, \eqref{bourgeon002} and \eqref{bourgeon} and taking $\ee_1 = \ee^4$, we obtain the desired inequality for any $z \geq 0$,
\[
\abs{J_0 - J_1} \leq c_{\ee} \frac{\left( 1+\max(y,0) \right)^2}{m_1^{\ee} \sqrt{m_2}} + \frac{c\left( 1+\max(y,0) \right)}{\sqrt{m_1}} \left( \e^{-\frac{cm_1}{m_2}} + \ee^4 \right).
\]
\end{proof}

\begin{lemma}
\label{lezardLLT}
Assume Hypotheses \ref{primitive}-\ref{nondegenere}.
There exists $\ee_0 \in (0,1/2)$ such that for any $\ee \in (0,\ee_0)$, $y \in \bb R$, $n \in \bb N$ such that $\ee^3 n \geq 1$, we have
\begin{align*}
	\sup_{x\in \bb X, \, z \geq 0} &\abs{\frac{n}{\sqrt{n_2}} \bb E_x \left( \varphi \left( \frac{y+S_{n_1}-z}{\sqrt{n_2}\sigma} \right) \,;\, \tau_y > n_1 \right) - \frac{2V(x,y)}{\sqrt{2\pi}\sigma} \varphi_+ \left( \frac{z}{\sqrt{n}\sigma} \right)} \\
	&\hspace{2cm}\leq c_{\ee} \frac{\left( 1+\max(y,0) \right)^2}{n^{\ee}} + c \left( 1+\max(y,0) \right) \ee,
\end{align*}
where $\varphi(t) = \e^{-\frac{t^2}{2}}/\sqrt{2\pi}$, $\varphi_+(t) = t\e^{-\frac{t^2}{2}} \bbm 1_{\{t\geq 0\}}$, $n_2 = \pent{\ee^3 n}$ and $n_1 = n-\pent{\ee^3 n}$.
\end{lemma}

\begin{proof}
Denote
\[
J_0 := \bb E_x \left( \varphi \left( \frac{y+S_{n_1}-z}{\sqrt{n_2}\sigma} \right) \,;\, \tau_y > n_1 \right)
\]
and
\begin{align}
	J_1 &:= \frac{2V(x,y)}{\sqrt{2\pi n_1}\sigma} \int_{0}^{+\infty}  \varphi \left( \sqrt{\frac{n_1}{n_2}} t - \frac{z}{\sqrt{n_2}\sigma} \right) \varphi_+(t) \dd t \nonumber\\
	&=\frac{2V(x,y)}{\sqrt{2\pi n_1}\sigma} \int_{0}^{+\infty} \sqrt{\frac{n_2}{n_1}} \varphi_{\sqrt{\frac{n_2}{n_1}}} \left( t - \frac{z}{\sqrt{n_1}\sigma} \right) \varphi_+(t) \dd t \nonumber\\
	&= \frac{2V(x,y)}{\sqrt{2\pi}\sigma} \frac{\sqrt{n_2}}{n_1}  \varphi_{\sqrt{\frac{n_2}{n_1}}}*\varphi_+ \left( \frac{z}{\sqrt{n_1}\sigma} \right),
	\label{soleil}
\end{align}
where $\varphi_{\{\cdot\}}(\cdot)$ is defined in \eqref{normal}. By Lemma \ref{couleuvre} we have
\[
\frac{n_1}{\sqrt{n_2}} \abs{J_0 - J_1} \leq c_{\ee} n_1 \frac{\left( 1+\max(y,0) \right)^2}{n_1^{\ee} n_2} + c n_1 \frac{ 1+\max(y,0) }{\sqrt{n_1}\sqrt{n_2}} \left( \e^{-c\frac{n_1}{n_2} } + \ee^4 \right).
\]
Since $\frac{n}{2} \leq n_1 \leq n$ and $\ee^3n-1 \leq n_2 \leq \ee^3 n$,
\begin{align}
	\frac{n}{\sqrt{n_2}} \abs{J_0 - J_1} &\leq c_{\ee} \frac{\left( 1+\max(y,0) \right)^2}{n^{\ee}} + c \frac{ 1+\max(y,0) }{\ee^{3/2}} \left( 1+\frac{c_{\ee}}{n} \right) \left( \e^{-\frac{c}{\ee^3}} + \ee^4 \right) \nonumber\\
	&\leq c_{\ee} \frac{\left( 1+\max(y,0) \right)^2}{n^{\ee}} + c \left( 1+\max(y,0) \right) \ee.
	\label{tulipe}
\end{align}
Let $J_2$ be the following term:
\begin{equation}
	\label{lion}
	J_2 := \frac{2V(x,y)}{\sqrt{2\pi}\sigma} \frac{\sqrt{n_2}}{n_1} \varphi_+ \left( \frac{z}{\sqrt{n_1}\sigma} \right).
\end{equation}
Using \eqref{soleil},
\[
\abs{J_1-J_2} \leq \frac{2V(x,y)}{\sqrt{2\pi}\sigma} \frac{\sqrt{n_2}}{n_1} \int_{\bb R} \varphi_{\sqrt{\frac{n_2}{n_1}}} (t) \abs{ \varphi_+ \left( \frac{z}{\sqrt{n_1}\sigma} - t \right) - \varphi_+ \left( \frac{z}{\sqrt{n_1}\sigma} \right) } \dd t.
\]
By the point \ref{RESUI002} of Proposition \ref{RESUI}, we write
\begin{align}
	\frac{n}{\sqrt{n_2}}\abs{J_1-J_2} &\leq c \left( 1+\max(y,0) \right) \norm{ \varphi_+' }_{\infty} \int_{\bb R} \varphi_{\sqrt{\frac{n_2}{n_1}}} (t) \abs{t} \dd t \nonumber\\
	&\leq c \left( 1+\max(y,0) \right) \sqrt{\frac{n_2}{n_1}} \int_{\bb R} \varphi (s) \abs{s} \dd s \nonumber\\
	&\leq c \left( 1+\max(y,0) \right) \ee^{3/2} .
	\label{catapulte}
\end{align}
Putting together \eqref{tulipe} and \eqref{catapulte}, we obtain that
\begin{equation}
	\label{lemurien001}
	\sup_{x\in \bb X,z \geq 0} \frac{n}{\sqrt{n_2}} \abs{J_0 - J_2} \leq c_{\ee} \frac{\left( 1+\max(y,0) \right)^2}{n^{\ee}} + c \left( 1+\max(y,0) \right) \ee.
\end{equation}

It remains to link $J_2$ from \eqref{lion} to the desired equivalent. We distinguish two cases. If $\frac{z}{\sigma} \leq \frac{\sqrt{n}}{\ee}$,
\begin{align*}
	\abs{\frac{n}{\sqrt{n_2}} J_2 - \frac{2V(x,y)}{\sqrt{2\pi}\sigma} \varphi_+ \left( \frac{z}{\sqrt{n}\sigma} \right)} &\leq c V(x,y) \abs{ \frac{n}{n_1}\varphi_+ \left( \frac{z}{\sqrt{n_1}\sigma} \right) - \varphi_+ \left( \frac{z}{\sqrt{n}\sigma} \right)} \\
	&\leq c V(x,y) \left( \norm{\varphi_+}_{\infty} \abs{\frac{n}{n_1} - 1} + \abs{\frac{1}{\sqrt{n_1}} - \frac{1}{\sqrt{n}}} \abs{\frac{z}{\sigma}} \norm{\varphi_+'}_{\infty} \right) \\
	&\leq c V(x,y) \left( \frac{n_2}{n_1} + \frac{1}{\sqrt{n_1}} \abs{ 1 - \sqrt{1-\frac{n_2}{n}} } \frac{\sqrt{n}}{\ee} \right) \\
	&\leq c V(x,y) \left( \ee^3 + \frac{\ee^3}{\ee} \right).
\end{align*}
If $\frac{z}{\sigma} > \frac{\sqrt{n}}{\ee} \geq \frac{\sqrt{n_1}}{\ee}$, we have
\[
\abs{\frac{n}{\sqrt{n_2}} J_2 - \frac{2V(x,y)}{\sqrt{2\pi}\sigma} \varphi_+ \left( \frac{z}{\sqrt{n}\sigma} \right)} \leq c V(x,y) \sup_{u \geq \frac{1}{\ee}} \varphi_+ \left( u \right) \leq c V(x,y) \e^{-\frac{c}{\ee^2}}.
\]
Therefore, using the point \ref{RESUI002} of Proposition \ref{RESUI}, we obtain that in each case
\begin{equation}
	\label{lemurien002}
\abs{\frac{n}{\sqrt{n_2}} J_2 - \frac{2V(x,y)}{\sqrt{2\pi}\sigma} \varphi_+ \left( \frac{z}{\sqrt{n}\sigma} \right)} \leq c \left( 1+\max(y,0) \right) \ee^2.
\end{equation}
Putting together \eqref{lemurien001} and \eqref{lemurien002}, proves the lemma.
\end{proof}

Another consequence of Lemma \ref{couleuvre} is the following lemma which will be used in Section \ref{giraffe}.

\begin{lemma}
\label{cameleonbis}
Assume Hypotheses \ref{primitive}-\ref{nondegenere}.
There exists $\ee_0 \in (0,1/2)$ such that for any $\ee \in (0,\ee_0)$, $y \in \bb R$, $n \in \bb N$ such that $\ee^3 n \geq 2$, we have
\begin{align*}
	\sup_{x\in \bb X} &\abs{\frac{n^{3/2}}{n_2-1} \bb E_x \left( \varphi_+ \left( \frac{y+S_{n_1}}{\sqrt{n_2-1}\sigma} \right) \,;\, \tau_y > n_1 \right) - \frac{V(x,y)}{\sigma}} \\
	&\hspace{3cm} \leq c_{\ee} \frac{\left( 1+\max(y,0) \right)^2}{n^{\ee}} + c \left( 1+\max(y,0)\right) \ee,
\end{align*}
where $\varphi_+(t) = t\e^{-\frac{t^2}{2}} \bbm 1_{\{t\geq 0\}}$ is the Rayleigh density function, $n_1 = n-\pent{\ee^3n}$ and $n_2 = \pent{\ee^3n}$.
\end{lemma}

\begin{proof}
Using Lemma \ref{couleuvre} with $\Psi = \varphi_+$, $m_1=n_1$, $m_2 = n_2-1$ and $z=0,$
\begin{align}
	\frac{n^{3/2}}{n_2-1} \abs{J_0-J_1} &\leq c_{\ee} \frac{\left( 1+\max(y,0) \right)^2 n^{3/2}}{(n_2-1)^{3/2} n_1^{\ee}} + c \frac{\left( 1+\max(y,0)\right) n^{3/2}}{(n_2-1)\sqrt{n_1}} \left( \e^{-c\frac{n_1}{(n_2-1)}} + \ee^4 \right) \nonumber\\
	&\leq c_{\ee} \frac{\left( 1+\max(y,0) \right)^2}{n^{\ee}} + c \frac{\left( 1+\max(y,0)\right)}{\ee^3} \left( 1+\frac{c_{\ee}}{n} \right) \left( \e^{-\frac{c}{\ee^3}} + \ee^4 \right) \nonumber\\
	&\leq c_{\ee} \frac{\left( 1+\max(y,0) \right)^2}{n^{\ee}} + c \left( 1+\max(y,0)\right) \ee,
	\label{loutre001}
\end{align}
where
\[
J_0 := \bb E_x \left( \varphi_+ \left( \frac{y+S_{n_1}}{\sqrt{n_2-1}\sigma} \right) \,;\, \tau_y > n_1 \right)
\]
and
\begin{align*}
	\frac{n^{3/2}}{n_2-1} J_1 &:= \frac{n^{3/2}}{n_2-1} \frac{2V(x,y)}{\sqrt{2\pi n_1}\sigma} \int_{0}^{+\infty} \varphi_+ \left( \sqrt{\frac{n_1}{n_2-1}}t \right) \varphi_+ (t) \dd t \\
	&= \frac{n^{3/2}}{n_2-1} \frac{2V(x,y)}{\sqrt{2\pi n_1}\sigma} \sqrt{\frac{n_1}{n_2-1}} \int_{0}^{+\infty}  t^2 \e^{ -\frac{\left( \frac{n_1}{n_2-1}+1 \right)t^2}{2} } \dd t \\
	&= \frac{n^{3/2}}{(n_2-1)^{3/2}} \frac{2V(x,y)}{\sqrt{2\pi}\sigma} \int_{0}^{+\infty}  t^2  \sqrt{\frac{2\pi (n_2-1)}{n-1}} \varphi_{\sqrt{\frac{n_2-1}{n-1}}} (t) \dd t
\end{align*}
where $\varphi_{\{\cdot\}}(\cdot)$ is defined in \eqref{normal}. So,
\begin{align*}
	\frac{n^{3/2}}{n_2-1} J_1 &= \frac{n^{3/2}}{\sqrt{n-1}(n_2-1)} \frac{2V(x,y)}{\sigma} \frac{n_2-1}{2(n-1)} \\
	&= \frac{n^{3/2}}{(n-1)^{3/2}} \frac{V(x,y)}{\sigma}.
\end{align*}
By the point \ref{RESUI002} of Proposition \ref{RESUI},
\begin{equation}
	\label{loutre002}
	\abs{\frac{n^{3/2}}{n_2-1} J_1 - \frac{V(x,y)}{\sigma}} \leq \frac{c}{n} \left( 1+\max(y,0) \right).
\end{equation}
The lemma follows from \eqref{loutre001} and \eqref{loutre002}.
\end{proof}

Thanks to Lemmata \ref{DOSP} and \ref{lezardLLT} we can bound $E_1$ from \eqref{chatLLT} as follows.

\begin{lemma}
\label{dragon}
Assume Hypotheses \ref{primitive}-\ref{nondegenere}.
For any $a > 0$ there exists $\ee_0 \in (0, 1/4)$ such that for any $\ee \in (0,\ee_0)$, any non-negative function $\psi \in \scr C$, any $y \in \bb R$ and $n \in \bb N$ such that $\ee^3 n \geq 1$, we have
\begin{align*}
	&\sup_{x\in \bb X,\, z \geq 0} n \abs{ E_1 - \frac{2a \bs \nu \left( \psi \right) V(x,y)}{\sqrt{2\pi}\sigma^2} \varphi_+ \left( \frac{z}{\sqrt{n}\sigma} \right)} \\
	&\qquad \leq c \left(1 + \max(y,0) \right) \norm{\psi}_{\infty} \left( \ee + \frac{c_{\ee}\left( 1+\max(y,0) \right)}{n^{\ee}} \right),
\end{align*}
where $E_1 = \bb E_x \left( \psi \left( X_{n} \right) \,;\, y+S_{n} \in [z,z+a] \,,\, \tau_y > n_1 \right)$, $n_1 = n- \pent{\ee^3 n}$ and $\varphi_+$ 
is the Rayleigh density function: $\varphi_+(t) = t\e^{-\frac{t^2}{2}} \bbm 1_{\{t\geq 0\}}$.
\end{lemma}

\begin{proof}
From Lemmas \ref{DOSP} and \ref{lezardLLT}, it follows that
\begin{align*}
	&n \abs{ E_1 - \frac{2a \bs \nu \left( \psi \right) V(x,y)}{\sqrt{2\pi}\sigma^2} \varphi_+ \left( \frac{z}{\sqrt{n}\sigma} \right)} \\
	&\leq c \left(1 + \max(y,0) \right) \norm{\psi}_{\infty} \left( \ee + \frac{c_{\ee}}{\sqrt{n}} \right) + \abs{\frac{a \bs \nu \left( \psi \right)}{\sigma}} \left( c_{\ee} \frac{\left( 1+\max(y,0) \right)^2}{n^{\ee}} + c \left( 1+\max(y,0) \right) \ee \right) \\
	&\leq  c \left(1 + \max(y,0) \right) \norm{\psi}_{\infty} \left( \ee + \frac{c_{\ee}\left( 1+\max(y,0) \right)}{n^{\ee}} \right).
\end{align*}
\end{proof}

\subsection{Control of \texorpdfstring{$E_2$}{}}

In this section we bound the term $E_2$ defined by \eqref{chatLLT}. To this aim let us recall and introduce some notations: for any $\ee \in (0,1)$, we consider $n_2 = \pent{\ee^3 n}$, $n_1 = n-n_2 = n-\pent{\ee^3 n}$, $n_3 = \pent{\frac{n_2}{2}}$ and $n_4 = n_2-n_3$. We define also
\begin{align}
	\label{albatros001}
	E_{21} &:= \bb E_x \left( \psi \left( X_{n} \right) \,;\, y+S_{n} \in [z,z+a] \,,\, y+S_{n_1} \leq \ee \sqrt{n} \,,\, n_1 < \tau_y \leq n \right) \\
	\label{albatros002}
	E_{22} &:= \bb E_x \left( \psi \left( X_{n} \right) \,;\, y+S_{n} \in [z,z+a] \,,\, y+S_{n_1} > \ee \sqrt{n} \,,\, n_1 < \tau_y \leq n_1+n_3 \right) \\
	\label{albatros003}
	E_{23} &:= \bb E_x \left( \psi \left( X_{n} \right) \,;\, y+S_{n} \in [z,z+a] \,,\, y+S_{n_1} > \ee \sqrt{n} \,,\, n_1+n_3 < \tau_y \leq n \right)
\end{align}
and we note that
\begin{equation}
	\label{albatros}
	E_2 = E_{21}+E_{22}+E_{23}.
\end{equation}

\begin{lemma}
\label{chevre001}
Assume Hypotheses \ref{primitive}-\ref{nondegenere}.
For any $a>0$ there exists $\ee_0 \in (0,1/4)$ such that for any $\ee \in (0,\ee_0)$, any non-negative function $\psi \in \scr C$, any $y \in \bb R$ and $n \in \bb N$ such that $\ee^3 n \geq 1$, we have
\[
\sup_{x\in \bb X, z \geq 0} n E_{21} \leq c \norm{\psi}_{\infty}\left( 1+ \max(y,0) \right) \left( \sqrt{\ee} + \frac{c_{\ee} \left( 1+\max(y,0) \right)}{n^{\ee}} \right)
\]
where $E_{21}$ is given as in \eqref{albatros001} by 
\[
E_{21} = \bb E_x \left( \psi \left( X_{n} \right) \,;\, y+S_{n} \in [z,z+a] \,,\, y+S_{n_1} \leq \ee \sqrt{n} \,,\, n_1 < \tau_y \leq n \right)
\]
and $n_1 = n-\pent{\ee^3 n}$.
\end{lemma}

\begin{proof}
Using the Markov property and the uniform bound \eqref{papillon001} of Corollary \ref{papillon}, with $n_2 = \pent{\ee^3 n}$,
\begin{align*}
	E_{21} &= \sum_{x' \in \bb X} \int_{0}^{+\infty} \bb E_{x'} \left( \psi \left( X_{n_2} \right) \,;\, y'+S_{n_2} \in [z,z+a]  \,,\, \tau_{y'} \leq n_2 \right) \\
	&\hspace{3cm} \times \bb P_x \left( X_{n_1} = x' \,,\, y+S_{n_1} \in \dd y' \,,\, y+S_{n_1} \leq \ee \sqrt{n} \,,\, \tau_y > n_1 \right) \\
	&\leq \frac{c \norm{\psi}_{\infty}}{\sqrt{n_2}} \bb P_x \left( y+S_{n_1} \leq \ee \sqrt{n} \,,\, \tau_y > n_1 \right).
\end{align*}
We note that $\frac{\ee \sqrt{n}}{\sigma \sqrt{n_1}} \leq \frac{\ee}{\sigma \sqrt{1-\ee^3}} \leq \frac{2}{\sigma} \ee$ and so by the point \ref{RESUIII002} of Proposition \ref{RESUIII} with $t_0=2\ee/\sigma$:
\begin{align*}
	E_{21} &\leq \frac{c \norm{\psi}_{\infty}}{\sqrt{n_2}} \left( \frac{cV(x,y)}{\sqrt{n_1}} \bf \Phi^+ \left( \frac{\ee \sqrt{n}}{\sigma \sqrt{n_1}} \right) + \frac{c_{\ee} \left( 1+\max(y,0)^2 \right)}{n_1^{1/2+\ee}} \right).
\end{align*}
Using the point \ref{RESUI002} of Proposition \ref{RESUI} and taking into account that $n_2 \geq \ee^3 n \left( 1-\frac{c_{\ee}}{n} \right)$, $n_1 \geq n/2$ and that $\bf \Phi^+(t) \leq \bf \Phi^+(t_0) \leq \frac{t_0^2}{2}$ for any $t\in (0,t_0)$,
\begin{align*}
	n E_{21} &\leq \frac{c \norm{\psi}_{\infty}}{\ee^{3/2}} \left( 1+\frac{c_{\ee}}{n} \right) \left( 1+ \max(y,0) \right) \left( \ee^2 + \frac{c_{\ee} \left( 1+\max(y,0) \right)}{n^{\ee}} \right) \\
	&\leq c \norm{\psi}_{\infty} \left( 1+ \max(y,0) \right) \left( \sqrt{\ee} + \frac{c_{\ee} \left( 1+\max(y,0) \right)}{n^{\ee}} \right),
\end{align*}
which implies the assertion of the lemma.
\end{proof}

\begin{lemma}
\label{chevre002}
Assume Hypotheses \ref{primitive}-\ref{nondegenere}.
For any $a>0$ there exists $\ee_0 \in (0,1/4)$ such that for any $\ee \in (0,\ee_0)$, any non-negative function $\psi \in \scr C$, any $y \in \bb R$, and $n \in \bb N$ satisfying $\ee^3 n \geq 2$, we have
\[
\sup_{x\in \bb X, z \geq 0} n E_{22} \leq c\norm{\psi}_{\infty} \left( 1+\max(y,0) \right) \left( \e^{-\frac{c}{\ee}} + \frac{c_{\ee}}{n^{\ee}} \right),
\]
where $E_{22}$ is given as in \eqref{albatros002} by 
\[
E_{22} = \bb E_x \left( \psi \left( X_{n} \right) \,;\, y+S_{n} \in [z,z+a] \,,\, y+S_{n_1} > \ee \sqrt{n} \,,\, n_1 < \tau_y \leq n_1+n_3 \right)
\]
and $n_1 = n-\pent{\ee^3 n}$, $n_2 = \pent{\ee^3 n}$ and $n_3 = \pent{\frac{n_2}{2}}$.
\end{lemma}

\begin{proof} By the Markov property,
\begin{align}
	\label{coccinelle}
	E_{22} &= \sum_{x' \in \bb X} \int_0^{+\infty} \underbrace{\bb E_{x'} \left( \psi \left( X_{n_2} \right) \,;\, y'+S_{n_2} \in [z,z+a] \,,\, \tau_{y'} \leq n_3 \right)}_{E_{22}'} \\
	&\hspace{2cm} \times \bb P_x \left( X_{n_1} = x' \,,\, y+S_{n_1} \in \dd y' \,,\, y+S_{n_1} > \ee \sqrt{n} \,,\,  \tau_y > n_1 \right). \nonumber
\end{align}

\textit{Bound of $E_{22}'$.} By the Markov property and the uniform bound \eqref{papillon001} in Corollary \ref{papillon}, with $n_4 = n_2 - n_3 = n - n_1-n_3$,
\begin{align*}
	E_{22}' &= \sum_{x'' \in \bb X} \int_{\bb R} \bb E_{x''} \left( \psi \left( X_{n_4} \right) \,;\, y''+S_{n_4} \in [z,z+a] \right) \\
	&\hspace{2cm} \times \bb P_{x'} \left( X_{n_3} = x'' \,,\, y'+S_{n_3} \in \dd y'' \,,\, \tau_{y'} \leq n_3 \right) \\
	&\leq \frac{c\norm{\psi}_{\infty}}{\sqrt{n_4}} \bb P_{x'} \left( \tau_{y'} \leq n_3 \right).
\end{align*}
Let $(B_t)_{t\geq 0}$ be the Brownian motion defined by Proposition \ref{majdeA_kLLT}. Denote by $A_n$ the following event:
\[
A_n = \left\{ \sup_{t\in [0,1]} \abs{S_{\pent{tn}} - \sigma B_{tn}} \leq n^{1/2-\ee} \right\},
\]
and by $\overline{A}_n$ its complement. We have
\begin{equation}
	\label{moineau1}
	E_{22}' \leq \frac{c\norm{\psi}_{\infty}}{\sqrt{n_4}} \left[ \bb P_{x'} \left( \tau_{y'} \leq n_3 \,,\, A_{n_3} \right) + \bb P_{x'} \left( \tau_{y'} \leq n_3 \,,\, \overline{A}_{n_3} \right) \right].
\end{equation}
Note that for any $x' \in \bb X$ and any $y' > \ee \sqrt{n}$,
\[
\bb P_{x'} \left( \tau_{y'} \leq n_3 \,,\, A_{n_3} \right) \leq \bb P \left( \tau_{y'-n_3^{1/2-\ee}}^{bm} \leq n_3 \right), 
\]
where, for any $y'' > 0$, $\tau_{y''}^{bm}$ is the exit time of the Brownian motion starting at $y''$ defined by \eqref{tauybm}. Since $y' > \ee \sqrt{n}$, it implies that
\begin{align*}
	\bb P_{x'} \left( \tau_{y'} \leq n_3 \,,\, A_{n_3} \right) &\leq \bb P \left( \inf_{t\in [0,1]} \sigma B_{tn_3} \leq n_3^{1/2-\ee} - y' \right) \\
	&\leq \bb P \left( \inf_{t\in [0,1]} \sigma B_{tn_3} \leq  \left( \frac{\ee^3 n}{2} \right)^{1/2-\ee} - \ee \sqrt{n} \right) \\
	&\leq \bb P \left( \inf_{t\in [0,1]} \sigma B_{tn_3} \leq  -\ee \sqrt{n} \left( 1- \frac{\ee^{1/2-3\ee}}{n^{\ee}} \right) \right).
\end{align*}
Since $\sqrt{n}/\sqrt{n_3} \geq \sqrt{2}/\ee^{3/2}$,
\begin{align}
	\bb P_{x'} \left( \tau_{y'} \leq n_3 \,,\, A_{n_3} \right) &\leq \bb P \left(  \abs{\frac{B_{n_3}}{\sqrt{n_3}}} \geq  \frac{\ee \sqrt{n}}{\sigma \sqrt{n_3}} \left( 1- \frac{1}{n^{\ee}} \right) \right) \nonumber\\
	&\leq \bb P \left(  \abs{B_1} \geq  \frac{\sqrt{2}}{\sigma \sqrt{\ee}} \left( 1- \frac{1}{n^{\ee}} \right) \right) \nonumber\\
	&\leq c \e^{-\frac{c}{\ee} \left( 1- \frac{c}{n^{\ee}} \right)}.
	\label{moineau2}
\end{align}
Therefore, putting together \eqref{moineau1} and \eqref{moineau2} and using Proposition \ref{majdeA_kLLT},
\[
E_{22}' \leq \frac{c\norm{\psi}_{\infty}}{\sqrt{n_4}} \left( c \e^{-\frac{c}{\ee} \left( 1- \frac{c}{n^{\ee}} \right)} + \bb P_{x'} \left( \overline{A}_{n_3} \right) \right) \leq \frac{c\norm{\psi}_{\infty}}{\sqrt{n_4}} \left( \e^{-\frac{c}{\ee} \left( 1- \frac{c}{n^{\ee}} \right)} + \frac{c_{\ee}}{n_3^{\ee}} \right).
\]
Since $n_4 \geq n_2/2 \geq \frac{\ee^3 n}{2} \left( 1-\frac{c_{\ee}}{n} \right)$ and $n_3 \geq n_2/2-1 \geq \frac{\ee^3 n}{2} \left( 1-\frac{c_{\ee}}{n} \right)$, we have
\begin{equation}
	\label{puceron}
	E_{22}' \leq \frac{c\norm{\psi}_{\infty}}{\ee^{3/2} \sqrt{n}} \left( 1+\frac{c_{\ee}}{n} \right) \left( \e^{-\frac{c}{\ee}} \e^{\frac{c_{\ee}}{n^{\ee}}} + \frac{c_{\ee}}{n^{\ee}} \right) \leq \frac{c\norm{\psi}_{\infty}}{\sqrt{n}} \left( \e^{-\frac{c}{\ee}} + \frac{c_{\ee}}{n^{\ee}} \right).
\end{equation}

Inserting \eqref{puceron} in \eqref{coccinelle} and using the point \ref{RESUII002} of Proposition \ref{RESUII} and the fact that $n_1 \geq n/2$, we conclude that
\[
E_{22} \leq \frac{c\norm{\psi}_{\infty} \left( 1+\max(y,0) \right)}{n} \left( \e^{-\frac{c}{\ee}} + \frac{c_{\ee}}{n^{\ee}} \right).
\]
\end{proof}

\begin{lemma}
\label{chevre003}
Assume Hypotheses \ref{primitive}-\ref{nondegenere}.
For any $a>0$ there exists $\ee_0 \in (0,1/4)$ such that for any $\ee \in (0,\ee_0)$, any non-negative function $\psi \in \scr C$, any $y \in \bb R$, and $n \in \bb N$ such that $\ee^3 n \geq 3$, we have
\[
\sup_{x\in \bb X, z\geq 0} nE_{23} \leq c \norm{\psi}_{\infty} \left( 1+\max(y,0) \right) \left( \ee + \frac{c_{\ee}}{n^{\ee}} \right),
\]
where $E_{23}$ is given as in \eqref{albatros003} by 
\[
E_{23} = \bb E_x \left( \psi \left( X_{n} \right) \,;\, y+S_{n} \in [z,z+a] \,,\, y+S_{n_1} > \ee \sqrt{n} \,,\, n_1+n_3 < \tau_y \leq n \right)
\]
and $n_1 = n-\pent{\ee^3 n}$, $n_2 = \pent{\ee^3 n}$ and $n_3 = \pent{\frac{n_2}{2}}$.
\end{lemma}

\begin{proof}
By the Markov property,
\begin{align}
	E_{23} &\leq \sum_{x'\in \bb X} \int_{0}^{+\infty} \underbrace{\bb E_{x'} \left( \psi \left( X_{n_2} \right) \,;\, y'+S_{n_2} \in [z,z+a] \,,\, n_3 < \tau_{y'} \leq n_2 \right)}_{=:E_{23}'} \nonumber\\
	&\hspace{3cm} \bb P_{x} \left( X_{n_1} = x' \,,\, y+S_{n_1} \in \dd y' \,,\, y+S_{n_1} > \ee \sqrt{n} \,,\, \tau_y > n_1 \right).
	\label{castor}
\end{align}
We consider two cases: when $z \leq \frac{\ee \sqrt{n}}{2}$ and when $z > \frac{\ee \sqrt{n}}{2}$.

Fix first $0 \leq z \leq \frac{\ee \sqrt{n}}{2}$. Using Corollary \ref{papillon}, we have for any $y' > \ee \sqrt{n}$,
\begin{align*}
	E_{23}' &\leq \bb E_{x'} \left( \psi \left( X_{n_2} \right) \,;\, y'+S_{n_2} \in [z,z+a] \right) \\
	&\leq \frac{a \bs \nu(\psi)}{\sqrt{2\pi n_2} \sigma} \e^{-\frac{(z-y')^2}{2n_2\sigma^2}} + \frac{c \norm{\psi}_{\infty}}{\sqrt{n_2}} \left( \frac{1}{\sqrt{n_2}} + \ee^{5/2} + c_{\ee}\e^{-c_{\ee}n_2} \right) \\
	&\leq \frac{c \norm{\psi}_{\infty}}{\ee^{3/2}\sqrt{n}} \left( 1+\frac{c_{\ee}}{n} \right) \left( \e^{-\frac{\ee^2 n}{8n_2\sigma^2}} + \frac{c_{\ee}}{\sqrt{n}} + \ee^{5/2} + c_{\ee}\e^{-c_{\ee}n} \right) \\
	&\leq \frac{c \norm{\psi}_{\infty}}{\ee^{3/2}\sqrt{n}} \left( 1+\frac{c_{\ee}}{n} \right) \left( \e^{-\frac{c}{\ee}} + \frac{c_{\ee}}{\sqrt{n}} + \ee^{5/2} \right).
\end{align*}
So, when $0 \leq z \leq \frac{\ee \sqrt{n}}{2}$, we have
\begin{equation}
	\label{guitare1}
	E_{23}' \leq \frac{c \norm{\psi}_{\infty}}{\sqrt{n}} \left( \frac{c_{\ee}}{\sqrt{n}} + \ee \right).
\end{equation}

Now we consider that $z > \frac{\ee \sqrt{n}}{2}$. Using Lemma \ref{duality} with $\mathfrak{m} = \bs \delta_{x'}$ and
\begin{align*}
F&(x_1,\dots,x_{n_2}) \\
&= \psi(x_{n_2}) \bbm 1_{\left\{ y'+f(x_1)+\dots+f(x_{n_2}) \in [z,z+a]\,,\, \exists k \in \{ n_3+1, \dots, n_2-1\},\, y'+f(x_1)+\dots+f(x_k) \leq 0 \right\}},
\end{align*}
we obtain
\begin{align*}
	E_{23}' &:= \bb E_{x'} \left( \psi\left( X_{n_2} \right) \,;\, y'+S_{n_2} \in [z,z+a] \,,\, n_3 < \tau_{y'} \leq n_2 \right) \\
	&\leq \bb E_{\bs \nu}^* \left( \psi \left( X_1^* \right) \frac{\bbm 1_{\{x'\}}\left( X_{n_2+1}^* \right)}{\bs \nu \left( X_{n_2+1}^* \right)} \,;\, y'+f\left( X_{n_2}^* \right)+ \dots + f\left( X_1^* \right) \in [z,z+a] \,,\, \right. \\
	&\left. \phantom{\frac{\left( X_{n_2}^* \right)}{\bs \nu \left( X_{n_2}^* \right)}} \exists k \in \{ n_3+1, \dots, n_2-1\},\, y'+f\left( X_{n_2}^* \right)+\dots+f \left(X_{n_2-k+1}^* \right) \leq 0 \right).
\end{align*}
By the Markov property,
\begin{align*}
	E_{23}' &\leq \norm{\psi}_{\infty} \bb E_{\bs \nu}^* \left( \psi_{x'}^*\left( X_{n_2}^* \right) \,;\, y'+f\left( X_{n_2}^* \right)+ \dots + f\left( X_1^* \right) \in [z,z+a] \,,\, \right. \\
	&\hspace{3cm} \left. \exists k \in \{ n_3+1, \dots, n_2-1 \},\, y'+f\left( X_{n_2}^* \right)+\dots+f \left(X_{n_2-k+1}^* \right) \leq 0 \right).
\end{align*}
where $\psi_{x'}^*$ is a function defined on $\bb X$ by the equation \eqref{hirondelle}. 
We note that, on the event $\left\{ y'+f\left( X_{n_2}^* \right)+ \dots + f\left( X_1^* \right) \in \left[ z,z+a \right] \right\} = \left\{ z + S_{n_2}^* \in \left[ y'-a,y' \right] \right\}$, we have
\begin{align*}
	&\left\{ \exists k \in \{ n_3+1, \dots, n_2-1\},\, y'+f\left( X_{n_2}^* \right)+\dots+f \left( X_{n_2-k+1}^* \right) \leq 0 \right\} \\
	&\qquad \subset \left\{ \exists k \in \{ n_3+1, \dots, n_2-1 \},\, z-f \left( X_{n_2-k}^* \right)-\dots-f\left( X_1^* \right) \leq 0 \right\} \\
	&\qquad = \left\{ \tau_z^* \leq  n_2-n_3-1 \right\}.
\end{align*}
Consequently,
\[
E_{23}' \leq c\norm{\psi}_{\infty} \bb P_{\bs \nu}^* \left( z+S_{n_2}^* \in [y'-a,y'] \,,\, \tau_z^* \leq n_4-1 \right),
\]
with $n_4 = n_2-n_3 = \pent{\ee^3 n} - \pent{\frac{\ee^3 n}{2}} \geq \frac{\ee^3n}{2} \left( 1-\frac{c_{\ee}}{n} \right)$. Proceeding in the same way as for the term $E_{22}'$ in \eqref{puceron} and using the fact that $z$ is larger than $c\ee\sqrt{n}$, we have
\begin{equation}
	\label{guitare2}
	E_{23}' \leq \frac{c\norm{\psi}_{\infty}}{\sqrt{n}} \left( \e^{-\frac{c}{\ee}} + \frac{c_{\ee}}{n^{\ee}} \right).
\end{equation}
Putting together \eqref{guitare1} and \eqref{guitare2}, for any $z \geq 0$, we obtain
\[
E_{23}' \leq \frac{c\norm{\psi}_{\infty}}{\sqrt{n}} \left( \ee + \frac{c_{\ee}}{n^{\ee}} \right).
\]
Inserting this bound in \eqref{castor} and using the point \ref{RESUII002} of Proposition \ref{RESUII}, we conclude that
\[
E_{23} \leq \frac{c \norm{\psi}_{\infty} \left( 1+\max(y,0) \right)}{n} \left( \ee + \frac{c_{\ee}}{n^{\ee}} \right).
\]
\end{proof}

Putting together Lemmas \ref{chevre001}, \ref{chevre002} and \ref{chevre003}, by \eqref{albatros}, 
we obtain the following bound for $E_2$:  

\begin{lemma}
\label{chevre}
Assume Hypotheses \ref{primitive}-\ref{nondegenere}.
For any $a>0$ there exists $\ee_0 \in (0,1/4)$ such that for any $\ee \in (0,\ee_0)$, any non-negative function $\psi \in \scr C$, any $y \in \bb R$ and $n \in \bb N$ such that $\ee^3 n \geq 3$, we have
\[
\sup_{x\in \bb X, z\geq 0} nE_2 \leq c \norm{\psi}_{\infty} \left( 1+\max(y,0) \right) \left( \sqrt{\ee} + \frac{c_{\ee} \left( 1+\max(y,0) \right)}{n^{\ee}} \right),
\]
where $E_2$ is given as in \eqref{chatLLT} by 
\[
E_2 = \bb E_x \left( \psi \left( X_n \right) \,;\, y+S_n \in [z,z+a] \,,\,  n_1 < \tau_y \leq n \right)
\]
and $n_1 = n-\pent{\ee^3 n}$.
\end{lemma}

\subsection{Proof of Theorem \ref{GNLLT}}

By \eqref{otarie} and \eqref{chatLLT},
\[
\bb E_x \left( \psi \left( X_{n} \right) \,;\, y+S_{n} \in [z,z+a] \,,\, \tau_y > n \right) = E_1+E_2.
\]
Lemma \ref{dragon} estimates $E_1$ and Lemma \ref{chevre} bounds $E_2$. Taking into account these two lemmas, Theorem \ref{GNLLT} follows.

\section{Proof of Theorem \ref{LLTC}}
\label{giraffe}

\subsection{Preliminary results}

\begin{lemma}
\label{caribou}
Assume Hypotheses \ref{primitive}-\ref{nondegenere}.
For any $a>0$ and $p \in \bb N^*$, there exists $\ee_0 \in (0,1/4)$ such that for any $\ee \in (0,\ee_0)$ there exists $n_0(\ee) \geq 1$ such that any non-negative function $\psi \in \scr C$, any $y' >0$, $z \geq 0$, $k\in \{0, \dots, p-1\}$ and $n \geq n_0(\ee)$, we have
\begin{align*}
	\sup_{x' \in \bb X} E_k' &\leq \frac{2a}{\sqrt{2\pi}p(n_2-1)\sigma^2} \varphi_+ \left( \frac{y'}{\sigma\sqrt{n_2-1}} \right) \bb E_{\bs \nu}^* \left( \psi \left( X_1^* \right) V^*\left( X_1^*, z_k+\frac{a}{p}+S_1^* \right)  \,;\, \tau_{z_k+\frac{a}{p}}^* > 1 \right) \\
	&\qquad + \frac{c \norm{\psi}_{\infty}}{n} (1+z)\left( \ee + \frac{c_{\ee}\left( 1+z \right)}{n^{\ee^8}} \right)
\end{align*}
and
\begin{align*}
	\inf_{x' \in \bb X} E_k' &\geq \frac{2a}{\sqrt{2\pi}p(n_2-1)\sigma^2} \varphi_+ \left( \frac{y'}{\sigma\sqrt{n_2-1}} \right) \bb E_{\bs \nu}^* \left( \psi \left( X_1^* \right) V^*\left( X_1^*, z_k+S_1^* \right)  \,;\, \tau_{z_k}^* > 1 \right) \\
	&\qquad - \frac{c \norm{\psi}_{\infty}}{n} (1+z)\left( \ee + \frac{c_{\ee}\left( 1+z \right)}{n^{\ee^8}} \right)
\end{align*}
where $E_k' = \bb E_{x'} \left( \psi \left( X_{n_2} \right) \,;\, y'+S_{n_2} \in \left( z_k,z_k+\frac{a}{p} \right] \,,\, \tau_{y'} > n_2 \right)$, $z_k =z+\frac{ka}{p}$ and $n_2 = \pent{\ee^3 n}$.
\end{lemma}

\begin{proof}
Using Lemma \ref{duality} with $\mathfrak{m} = \bs \delta_{x'}$ and
\[
F(x_1,\dots,x_{n_2}) = \psi(x_{n_2}) \bbm 1_{\left\{ y'+f(x_1) \dots+f(x_{n_2}) \in \left( z_k,z_k+\frac{a}{p} \right]\,,\, \forall i \in \{ 1, \dots, n_2 \},\, y'+f(x_1)+\dots+f(x_i) > 0 \right\}},
\]
we have
\begin{align*}
	E_k' &= \bb E_{\bs \nu}^* \left( \psi \left( X_1^* \right) \psi_{x'}^*\left( X_{n_2}^* \right) \,;\, y'+f\left( X_{n_2}^* \right) + \dots + f\left( X_1^* \right) \in \left( z_k,z_k+\frac{a}{p} \right] \,,\, \right. \\
	&\qquad \qquad \left. \phantom{\frac{\left( X_{n_2}^* \right)}{\bs \nu \left( X_{n_2}^* \right)}} \forall i \in \{ 1, \dots, n_2 \},\, y'+f\left( X_{n_2}^* \right)+\dots+f \left(X_{n_2-i+1}^* \right) > 0 \right).
\end{align*}
where $\psi_{x'}^*$ is the function defined on $\bb X$ by \eqref{hirondelle}.

\textit{The upper bound.} Note that, on the event $\left\{ y'+f\left( X_{n_2}^* \right)+ \dots + f\left( X_1^* \right) \in \left( z_k,z_k+\frac{a}{p} \right] \right\} = \left\{ z_k +\frac{a}{p} + S_{n_2}^* \in \left[ y',y'+\frac{a}{p} \right) \right\}$, we have
\begin{align}
	&\left\{ \forall i \in \{ 1, \dots, n_2 \},\, y'+f\left( X_{n_2}^* \right)+\dots+f \left( X_{n_2-i+1}^* \right) > 0, y'>0 \right\} \nonumber\\
	&\qquad \subset \left\{ \forall i \in \{ 1, \dots, n_2-1 \},\, z_k +\frac{a}{p}-f \left( X_{n_2-i}^* \right)-\dots-f\left( X_1^* \right) > 0, \right. \nonumber\\
	&\hspace{11cm} \left. z_k +\frac{a}{p} + S_{n_2}^* >0 \right\} \nonumber\\
	&\qquad = \left\{ \tau_{z_k+\frac{a}{p}}^* > n_2 \right\}.
	\label{campagnol}
\end{align}
So, for any $y'>0$,
\begin{align*}
	E_k' &\leq \bb E_{\bs \nu}^* \left( \psi \left( X_1^* \right) \psi_{x'}^*\left( X_{n_2}^* \right) \,;\, z_k +\frac{a}{p} + S_{n_2}^* \in \left[ y',y'+\frac{a}{p} \right) \,,\, \tau_{z_k+\frac{a}{p}}^* > n_2 \right) \\
	&\leq \sum_{x'' \in \bb X} \int_{0}^{+\infty} \psi \left( x'' \right) \bb E_{x''}^* \left( \psi_{x'}^*\left( X_{n_2-1}^* \right) \,;\, z'' + S_{n_2-1}^* \in \left[ y',y'+\frac{a}{p} \right] \,,\, \tau_{z''}^* > n_2-1 \right) \\
	&\hspace{3cm} \times \bb P_{\bs \nu}^* \left( X_1^* = \dd x'' \,,\, z_k +\frac{a}{p} +S_1^* \in \dd z'' \,,\, \tau_{z_k +\frac{a}{p}}^* > 1 \right).
\end{align*}
Using Theorem \ref{GNLLT} for the reverse chain with $\ee'=\ee^{8}$, we obtain that
\begin{align*}
	E_k' &\leq \frac{2a \bs \nu \left( \psi_{x'}^* \right)}{\sqrt{2\pi}(n_2-1)p\sigma^2} \varphi_+ \left( \frac{y'}{\sqrt{n_2-1}\sigma} \right) \sum_{x'' \in \bb X} \int_{0}^{+\infty} \psi \left( x'' \right) V^* \left(x'',z'' \right)\\
	&\hspace{3cm} \times \bb P_{\bs \nu}^* \left( X_1^* = \dd x'' \,,\, z_k +\frac{a}{p} +S_1^* \in \dd z'' \,,\,\tau_{z_k +\frac{a}{p}} > 1 \right)\\
	&\qquad + \frac{c \norm{\psi_{x'}^*}_{\infty} \norm{\psi}_{\infty}}{n_2-1} \bb E_{\bs \nu}^* \left( \left(1 + \max\left( z_k +\frac{a}{p} +S_1^*,0 \right) \right) \right. \\
	&\hspace{5cm} \times \left. \left( \sqrt{\ee^{8}} + \frac{c_{\ee}\left( 1+\max\left( z_k +\frac{a}{p} +S_1^*,0 \right) \right)}{(n_2-1)^{\ee^{8}}} \right) \,,\, \tau_{z_k +\frac{a}{p}}^* > 1 \right).
\end{align*}
Note that by \eqref{hirondelle}, $\bs \nu \left( \psi_{x'}^* \right) = 1$ and $\norm{\psi_{x'}^*}_{\infty} \leq c$. So,
\begin{align*}
	E_k' &\leq \frac{2a}{\sqrt{2\pi}(n_2-1)p\sigma^2} \varphi_+ \left( \frac{y'}{\sqrt{n_2-1}\sigma} \right) \bb E_{\bs \nu}^* \left( \psi \left( X_1^* \right) V^* \left( X_1^* ,z_k +\frac{a}{p} +S_1^* \right) \,,\, \tau_{z_k +\frac{a}{p}}^* > 1 \right)\\
	&\qquad + \frac{c \norm{\psi}_{\infty}}{\ee^3 n} \left( 1+\frac{c_{\ee}}{n} \right) (1+z)\left( \ee^4 + \frac{c_{\ee}\left( 1+z \right)}{n^{\ee^8}} \right)
\end{align*}
and the upper bound of the lemma is proved.

\textit{The lower bound.} Similarly as in the proof of the upper bound we note that, on the event $\left\{ y'+f\left( X_{n_2}^* \right)+ \dots + f\left( X_1^* \right) \in \left( z_k,z_k+\frac{a}{p} \right] \right\} = \left\{ z_k + S_{n_2}^* \in \left[ y'-\frac{a}{p},y' \right) \right\}$, we have
\begin{align}
	&\left\{ \forall i \in \{ 1, \dots, n_2 \},\, y'+f\left( X_{n_2}^* \right)+\dots+f \left( X_{n_2-i+1}^* \right) > 0 \right\} \nonumber\\
	&\qquad \supset \left\{ \forall i \in \{ 1, \dots, n_2-1 \},\, z_k -f \left( X_{n_2-i}^* \right)-\dots-f\left( X_1^* \right) > 0 \right\}  \nonumber\\
	&\qquad = \left\{ \tau_{z_k}^* > n_2-1 \right\} \supset \left\{ \tau_{z_k}^* > n_2 \right\}.
	\label{mulot}
\end{align}
Let $y_+' := \max(y'-a/p,0)$ and $a' := \min(y',a/p) \in (0,a]$. For any $\eta \in (0,a')$,
\begin{align*}
	E_k' &\geq \bb E_{\bs \nu}^* \left( \psi \left( X_1^* \right) \psi_{x'}^*\left( X_{n_2}^* \right) \,;\, z_k + S_{n_2}^* \in \left[ y'-\frac{a}{p},y' \right) \,,\, \tau_{z_k}^* > n_2 \right) \\
	&\geq \sum_{x'' \in \bb X} \int_{0}^{+\infty} \psi \left( x'' \right) \bb E_{x''}^* \left( \psi_{x'}^*\left( X_{n_2-1}^* \right) \,;\, z'' + S_{n_2-1}^* \in \left[ y_+',y_+'+a'-\eta \right] \,,\, \tau_{z''}^* > n_2-1 \right) \\
	&\hspace{3cm} \times \bb P_{\bs \nu}^* \left( X_1^* = \dd x'' \,,\, z_k +S_1^* \in \dd z'' \,,\, \tau_{z_k}^* > 1 \right).
\end{align*}
Using Theorem \ref{GNLLT},
\begin{align*}
	E_k' &\geq \frac{2(a'-\eta) \bs \nu \left( \psi_{x'}^* \right)}{\sqrt{2\pi}(n_2-1)\sigma^2} \varphi_+ \left( \frac{y_+'}{\sqrt{n_2-1}\sigma} \right) \sum_{x'' \in \bb X} \int_{0}^{+\infty} \psi \left( x'' \right) V^* \left(x'',z'' \right)\\
	&\hspace{3cm} \times \bb P_{\bs \nu}^* \left( X_1^* = \dd x'' \,,\, z_k +S_1^* \in \dd z'' \,,\, \tau_{z_k}^* > 1 \right)\\
	&\qquad - \frac{c \norm{\psi_{x'}^*}_{\infty} \norm{\psi}_{\infty}}{n_2-1} \bb E_{\bs \nu}^* \left( \left(1 + \max\left( z_k +S_1^*,0 \right) \right) \right. \\
	&\hspace{5cm} \times \left. \left( \sqrt{\ee^{8}} + \frac{c_{\ee}\left( 1+\max\left( z_k +S_1^*,0 \right) \right)}{(n_2-1)^{\ee^{8}}} \right) \,,\, \tau_{z_k}^* > 1 \right) \\
	&\geq \frac{2(a'-\eta)}{\sqrt{2\pi}(n_2-1)\sigma^2} \varphi_+ \left( \frac{y_+'}{\sqrt{n_2-1}\sigma} \right) \bb E_{\bs \nu}^* \left( \psi \left( X_1^* \right) V^* \left( X_1^* ,z_k +S_1^* \right) \,,\, \tau_{z_k}^* > 1 \right)\\
	&\qquad - \frac{c \norm{\psi}_{\infty}}{\ee^3 n} \left( 1+\frac{c_{\ee}}{n} \right) (1+z)\left( \ee^4 + \frac{c_{\ee}\left( 1+z \right)}{n^{\ee^8}} \right).
\end{align*}
Note that, if $y' \geq a/p$ we have
\begin{align*}
(a'-\eta) \varphi_+ \left( \frac{y_+'}{\sqrt{n_2-1}\sigma} \right) &= \left( \frac{a}{p}-\eta \right) \varphi_+ \left( \frac{y'-\frac{a}{p}}{\sqrt{n_2-1}\sigma} \right) \\
&\geq  \left( \frac{a}{p}-\eta \right) \varphi_+ \left( \frac{y'}{\sqrt{n_2-1}\sigma} \right) - \norm{\varphi_+'}_{\infty}\frac{a^2}{p^2 \sqrt{n_2-1}\sigma}
\end{align*}
and if $0 < y' \leq a/p$ we have
\begin{align*}
(a'-\eta) \varphi_+ \left( \frac{y_+'}{\sqrt{n_2-1}\sigma} \right) = 0 &\geq \left( \frac{a}{p}-\eta \right) \varphi_+ \left( \frac{y'}{\sqrt{n_2-1}\sigma} \right) - \norm{\varphi_+'}_{\infty}\frac{a y'}{p \sqrt{n_2-1}\sigma} \\
&\geq  \left( \frac{a}{p}-\eta \right) \varphi_+ \left( \frac{y'}{\sqrt{n_2-1}\sigma} \right) - \norm{\varphi_+'}_{\infty}\frac{a^2}{p^2 \sqrt{n_2-1}\sigma}.
\end{align*}
Moreover, using the points \ref{RESUI001} and \ref{RESUI002} of Proposition \ref{RESUI}, we observe that
\[
\bb E_{\bs \nu}^* \left( \psi\left( X_1^* \right) V^* \left( X_1^* ,z_k +S_1^* \right) \,,\, \tau_{z_k}^* > 1 \right) \leq c \norm{\psi}_{\infty} \left( 1+z \right).
\]
Consequently, for any $y' > 0$, 
\begin{align*}
	E_k' &\geq \frac{2\left( \frac{a}{p}-\eta \right)}{\sqrt{2\pi}(n_2-1)\sigma^2} \varphi_+ \left( \frac{y'}{\sqrt{n_2-1}\sigma} \right) \bb E_{\bs \nu}^* \left( \psi \left( X_1^* \right) V^* \left( X_1^* ,z_k +S_1^* \right) \,,\, \tau_{z_k}^* > 1 \right)\\
	&\qquad - \frac{c_{\ee} \norm{\psi}_{\infty}}{n^{3/2}} \left( 1+z \right) - \frac{c \norm{\psi}_{\infty}}{n} (1+z)\left( \ee + \frac{c_{\ee}\left( 1+z \right)}{n^{\ee^8}} \right).
\end{align*}
Taking the limit as $\eta \to 0$, the lower bound of the lemma follows.
\end{proof}

\begin{lemma}
\label{paon}
Assume Hypotheses \ref{primitive}-\ref{nondegenere}.
For any $a>0$ and $p \in \bb N^*$, there exists $\ee_0 \in (0,1/4)$ such that for any $\ee \in (0,\ee_0)$ there exists $n_0(\ee) \geq 1$ such that any non-negative function $\psi \in \scr C$, any $y \in \bb R$, $z \geq 0$ and $n \geq n_0(\ee)$, we have
\begin{align*}
	\sup_{x \in \bb X} n^{3/2} E_0 &\leq \frac{2aV(x,y)}{p\sqrt{2\pi}\sigma^3} \sum_{k=0}^{p-1} \bb E_{\bs \nu}^* \left( \psi \left( X_1^* \right) V^*\left( X_1^*, z_k+\frac{a}{p}+S_1^* \right)  \,;\, \tau_{z_k+\frac{a}{p}}^* > 1 \right) \\
	&\qquad + p c \norm{\psi}_{\infty} (1+z)\left( 1+\max(y,0) \right) \left( \ee + \frac{c_{\ee}\left( 1+z+\max(y,0) \right)}{n^{\ee^8}} \right)
\end{align*}
and
\begin{align*}
	\inf_{x \in \bb X} n^{3/2} E_0 &\geq \frac{2aV(x,y)}{p\sqrt{2\pi}\sigma^3} \sum_{k=0}^{p-1} \bb E_{\bs \nu}^* \left( \psi \left( X_1^* \right) V^*\left( X_1^*, z_k+S_1^* \right)  \,;\, \tau_{z_k}^* > 1 \right) \\
	&\qquad - p c \norm{\psi}_{\infty} (1+z)\left( 1+\max(y,0) \right) \left( \ee + \frac{c_{\ee}\left( 1+z+\max(y,0) \right)}{n^{\ee^8}} \right)
\end{align*}
where $E_0 = \bb E_{x} \left( \psi \left( X_n \right) \,;\, y+S_n \in \left( z,z+a \right] \,,\, \tau_y > n \right)$ and for any $k\in \{0,\dots,p-1\}$, $z_k = z+\frac{ka}{p}$.
\end{lemma}

\begin{proof}
Set $n_1 = n-\pent{\ee^3 n}$ and $n_2 = \pent{\ee^3n}$. By the Markov property, for any $p \geq 1$,
\begin{align*}
	E_0 &= \sum_{x'\in \bb X} \int_{0}^{+\infty} \bb E_{x'} \left( \psi \left( X_{n_2} \right) \,;\, y'+S_{n_2} \in \left( z,z+a \right] \,,\, \tau_{y'} > n_2 \right) \\
	&\hspace{5cm} \times \bb P_x \left( X_{n_1} = \dd x' \,,\, y+S_{n_1} \in \dd y' \,,\, \tau_y > n_1 \right) \\
	&= \sum_{x'\in \bb X} \int_{0}^{+\infty} \sum_{k=0}^{p-1} \; E_k' \; \times \bb P_x \left( X_{n_1} = \dd x' \,,\, y+S_{n_1} \in \dd y' \,,\, \tau_y > n_1 \right),
\end{align*}
where for any $k \in \{ 0, \dots, p-1 \}$,
\[
E_k' = \bb E_{x'} \left( \psi \left( X_{n_2} \right) \,;\, y'+S_{n_2} \in \left( z_k,z_k+\frac{a}{p} \right] \,,\, \tau_{y'} > n_2 \right)
\]
and $z_k = z+\frac{ka}{p}$.

\textit{The upper bound.} By Lemma \ref{caribou},
\begin{align*}
	E_0 &\leq \frac{2a}{p(n_2-1)\sqrt{2\pi}\sigma^2} \sum_{k=0}^{p-1} \bb E_x \left( \varphi_+ \left( \frac{y+S_{n_1}}{\sigma\sqrt{n_2-1}} \right) \,;\, \tau_y > n_1 \right) J_1(k) \\
	&\qquad + \sum_{k=0}^{p-1} \frac{c \norm{\psi}_{\infty}}{n} (1+z)\left( \ee + \frac{c_{\ee}\left( 1+z \right)}{n^{\ee^8}} \right) \bb P_x \left( \tau_y > n_1 \right),
\end{align*}
where $J_1(k)= \bb E_{\bs \nu}^* \left( \psi \left( X_1^* \right) V^*\left( X_1^*, z_k+\frac{a}{p}+S_1^* \right)  \,;\, \tau_{z_k+\frac{a}{p}}^* > 1 \right)$, for any $k\in \{0,\dots,p-1\}$.  By Lemma \ref{cameleonbis} and the point \ref{RESUII002} of Proposition \ref{RESUII},
\begin{align*}
	n^{3/2} E_0 &\leq \frac{2a}{p\sqrt{2\pi}\sigma^2} \sum_{k=0}^{p-1} J_1(k) \frac{V(x,y)}{\sigma} + \frac{1}{p} \sum_{k=0}^{p-1}  J_1(k) \left( \frac{c_{\ee}\left( 1+\max(y,0) \right)^2}{n^{\ee}} + c \left( 1+\max(y,0)\right) \ee \right) \\
	&\qquad + p c \norm{\psi}_{\infty}(1+z)\left( \ee + \frac{c_{\ee}\left( 1+z \right)}{n^{\ee^8}} \right) \left( 1+\max(y,0) \right).
\end{align*}
Note that, using the points \ref{RESUI001} and \ref{RESUI002} of Proposition \ref{RESUI}, we have
\[
\frac{1}{p} \sum_{k=0}^{p-1} J_1(k) \leq c \norm{\psi}_{\infty} (1+z).
\]
Therefore
\begin{align*}
	n^{3/2} E_0 &\leq \frac{2aV(x,y)}{p\sqrt{2\pi}\sigma^3} \sum_{k=0}^{p-1} J_1(k) \\
	&\qquad + p c \norm{\psi}_{\infty} (1+z)\left( 1+\max(y,0) \right) \left( \ee + \frac{c_{\ee}\left( 1+z+\max(y,0) \right)}{n^{\ee^8}} \right)
\end{align*}
and the upper bound of the lemma is proved.

\textit{The lower bound.} The proof of the lower bound is similar to the proof of the upper bound and therefore will not be detailed.
\end{proof}

\subsection{Proof of Theorem \ref{LLTC}.}

The second point of Theorem \ref{LLTC} was proved by Lemma \ref{gorille002}. It remains to prove the first point.
Let $\psi \in \scr C$, $a>0$, $x\in \bb X$, $y \in \bb R$ and $z \geq 0$. Suppose first that $z>0$. For any $n \geq 1$ and $\eta \in(0,\min(z,1))$,
\begin{equation}
\label{orage}
\bb E_x \left( \psi \left( X_{n} \right) \,;\, y+S_{n} \in [z,z+a] \,,\, \tau_y > n \right) \leq E_0(\eta),
\end{equation}
where $E_0(\eta) = \bb E_{x} \left( \psi \left( X_n \right) \,;\, y+S_n \in \left( z-\eta,z+a \right] \,,\, \tau_y > n \right)$. Taking the limit as $n\to +\infty$ in Lemma \ref{paon}, we have, for any $p \in \bb N^*$ and $\ee \in (0, \ee_0(p))$,
\begin{align*}
	&\limsup_{n\to+\infty} n^{3/2} E_0(\eta) \\
	&\qquad\leq \frac{2(a+\eta)V(x,y)}{\sqrt{2\pi}p\sigma^3} \sum_{k=0}^{p-1} \bb E_{\bs \nu}^* \left( \psi \left( X_1^* \right) V^*\left( X_1^*, z_{k,\eta}+\frac{a+\eta}{p}+S_1^* \right)  \,;\, \tau_{z_{k,\eta}+\frac{a+\eta}{p}}^* > 1 \right) \\
	&\qquad \qquad + p c \norm{\psi}_{\infty} (1+z-\eta)\left( 1+\max(y,0) \right) \ee,
\end{align*}
with $z_{k,\eta} = z-\eta+\frac{k(a+\eta)}{p}$ for $k\in \{0,\dots,p-1\}$. Taking the limit as $\ee \to 0$,
\begin{align*}
	&\limsup_{n\to+\infty} n^{3/2} E_0(\eta) \\
	&\qquad \leq \frac{2(a+\eta)V(x,y)}{\sqrt{2\pi}p\sigma^3} \sum_{k=0}^{p-1} \bb E_{\bs \nu}^* \left( \psi \left( X_1^* \right) V^*\left( X_1^*, z_{k,\eta}+\frac{a+\eta}{p}+S_1^* \right)  \,;\, \tau_{z_{k,\eta}+\frac{a+\eta}{p}}^* > 1 \right).
\end{align*}
By the point \ref{RESUI002} of Proposition \ref{RESUI}, the function $u \mapsto V^*\left( x^*, u-f(x^*) \right) \bbm 1_{\left\{ u-f(x^*) >0 \right\}}$ is monotonic and so is Riemann integrable. Since $\bb  X$ is finite, we have
\begin{align*}
	\lim_{p\to+\infty} \frac{a+\eta}{p} \sum_{k=0}^{p-1} \bb E_{\bs \nu}^* &\left( \psi \left( X_1^* \right) V^*\left( X_1^*, z_{k,\eta} +\frac{a+\eta}{p}+S_1^* \right)  \,;\, \tau_{z_{k,\eta} +\frac{a+\eta}{p}}^* > 1 \right) \\
	&= \bb E_{\bs \nu}^* \left( \psi \left( X_1^* \right) \int_{z-\eta}^{z+a} V^*\left( X_1^*, z'+S_1^* \right)  \bbm 1_{\left\{ z'+S_1^* >0 \right\}}  \dd z' \right) \\
	&= \int_{z-\eta}^{z+a} \bb E_{\bs \nu}^* \left( \psi \left( X_1^* \right) V^*\left( X_1^*, z'+S_1^* \right) \,;\, \tau_{z'}^* > 1 \right)  \dd z'.
\end{align*}
Therefore,
\[
\limsup_{n\to+\infty} n^{3/2} E_0(\eta) \leq \frac{2V(x,y)}{\sqrt{2\pi}\sigma^3} \int_{z-\eta}^{z+a} \bb E_{\bs \nu}^* \left( \psi \left( X_1^* \right) V^*\left( X_1^*, z'+S_1^* \right) \,;\, \tau_{z'}^* > 1 \right)  \dd z'.
\]
Taking the limit as $\eta \to 0$ and using \eqref{orage}, we obtain that, for any $z >0$,
\begin{align}
\limsup_{n\to +\infty} n^{3/2} &\bb E_x \left( \psi \left( X_{n} \right) \,;\, y+S_{n} \in [z,z+a] \,,\, \tau_y > n \right) \nonumber\\
	&\qquad = \frac{2V(x,y)}{\sqrt{2\pi}\sigma^3} \int_z^{z+a} \bb E_{\bs \nu}^* \left( \psi \left( X_1^* \right) V^*\left( X_1^*, z'+S_1^* \right) \,;\, \tau_{z'}^* > 1 \right) \dd z'.
	\label{fleur}
\end{align}
If $z=0$, we have
\[
\bb E_x \left( \psi \left( X_n \right) \,;\, y+S_n \in [0,a] \,,\, \tau_y > n \right) = \bb E_x \left( \psi \left( X_n \right) \,;\, y+S_n \in (0,a] \,,\, \tau_y > n \right).
\]
Using Lemma \ref{paon} and the same arguments as before, it is easy to see that \eqref{fleur} holds for $z=0$.

Since $[z,z+a] \supset  (z,z+a]$ we have obviously
\[
\bb E_x \left( \psi \left( X_{n} \right) \,;\, y+S_{n} \in [z,z+a] \,,\, \tau_y > n \right) \geq \bb E_x \left( \psi \left( X_{n} \right) \,;\, y+S_{n} \in (z,z+a] \,,\, \tau_y > n \right).
\]
Using this and Lemma \ref{paon} we obtain \eqref{fleur} with $\liminf$ instead of $\limsup,$ which concludes the proof of the theorem.

\section{Proof of Theorems \ref{CAPE} and \ref{CAPEBIS}} \label{2theor}

\subsection{Preliminaries results.}

\begin{lemma} Assume Hypotheses \ref{primitive}-\ref{nondegenere}.
\label{sombrero}
For any $x \in \bb X$, $y \in \bb R$, $z \geq 0$, $a >0$, any non-negative function $\psi$: $\bb X \to \bb R_+$ and any non-negative and continuous function $g$: $[z,z+a] \to \bb R_+$, we have
\begin{align*}
&\lim_{n\to +\infty} n^{3/2} \bb E_x \left( g \left( y+S_n \right) \psi\left( X_n \right) \,;\, y+S_n \in [z,z+a) \,,\, \tau_y > n \right) \\
&\hspace{3cm} = \frac{2V(x,y)}{\sqrt{2\pi}\sigma^3} \int_z^{z+a} g(z') \bb E_{\bs \nu}^* \left( \psi\left( X_1^* \right) V^*\left( X_1^*, z'+S_1^* \right) \,;\, \tau_{z'}^* > 1 \right) \dd z'.
\end{align*}
\end{lemma}

\begin{proof}
Fix $x \in \bb X$, $y \in \bb R$, $z \geq 0$, $a >0$, and let $\psi$: $\bb X \to \bb R_+$ be a non-negative function and $g$: $[z,z+a] \to \bb R_+$ be a non-negative and continuous function. For any measurable non-negative and bounded function $\varphi$: $\bb R \to \bb R_+$, we define
\[
I_0(\varphi) := n^{3/2} \bb E_x \left( \psi\left( X_n \right) \varphi\left( y+S_n \right) \,;\, \tau_y > n \right).
\]

We first prove that for any $0 \leq \alpha < \beta$ we have
\begin{equation}
I_0 \left( \bbm 1_{[\alpha,\beta)} \right) \underset{n \to +\infty}{\longrightarrow} \frac{2V(x,y)}{\sqrt{2\pi}\sigma^3} \int_\alpha^\beta  \bb E_{\bs \nu}^* \left( \psi\left( X_1^* \right) V^*\left( X_1^*, z'+S_1^* \right) \,;\, \tau_{z'}^* > 1 \right) \dd z'.
\label{perle}
\end{equation}
Since $[\alpha,\beta) \subset [\alpha,\beta]$, the upper limit is a straightforward consequence of Theorem \ref{LLTC}:
\begin{align*}
\limsup_{n\to+\infty} I_0 \left( \bbm 1_{[\alpha,\beta)} \right) &\leq \limsup_{n\to+\infty} n^{3/2} \bb E_x \left( \psi\left( X_n \right) \,;\, y+S_n \in [\alpha,\beta] \,,\, \tau_y > n \right) \\
&= \frac{2V(x,y)}{\sqrt{2\pi}\sigma^3} \int_\alpha^\beta \bb E_{\bs \nu}^* \left( \psi\left( X_1^* \right) V^*\left( X_1^*, z'+S_1^* \right) \,;\, \tau_{z'}^* > 1 \right) \dd z'.
\end{align*}
and for the lower limit, we write for any $\eta \in (0,\beta-\alpha)$,
\begin{align*}
\liminf_{n\to+\infty} I_0 \left( \bbm 1_{[\alpha,\beta)} \right) &\geq \liminf_{n\to+\infty} n^{3/2} \bb E_x \left( \psi\left( X_n \right) \,;\, y+S_n \in [\alpha,\beta-\eta] \,,\, \tau_y > n \right) \\
&= \frac{2V(x,y)}{\sqrt{2\pi}\sigma^3} \int_\alpha^{\beta-\eta} \bb E_{\bs \nu}^* \left( \psi\left( X_1^* \right)  V^*\left( X_1^*, z'+S_1^* \right) \,;\, \tau_{z'}^* > 1 \right) \dd z'.
\end{align*}
Taking the limit as $\eta \to 0$, it proves \eqref{perle}.

From \eqref{perle}, it is clear that by linearity, for any non-negative stepwise function $\varphi = \sum_{k=1}^N \gamma_k \bbm 1_{[\alpha_k,\beta_k)}$, where $N \geq 1$, $\gamma_1, \dots, \gamma_N \in \bb R_+$ and $0 < \alpha_1 < \beta_1 = \alpha_2 < \dots < \beta_N$, we have
\[
\lim_{n\to +\infty} I_0 \left( \varphi \right) = \frac{2V(x,y)}{\sqrt{2\pi}\sigma^3} \int_{\alpha_1}^{\beta_N} \varphi(z') \bb E_{\bs \nu}^* \left( \psi\left( X_1^* \right) V^*\left( X_1^*, z'+S_1^* \right) \,;\, \tau_{z'}^* > 1 \right) \dd z'.
\]
Since $g$ is continuous on $[z,z+a]$, for any $\ee \in (0,1)$ there exists $\varphi_{1,\ee}$ and $\varphi_{2,\ee}$ two stepwise functions on $[z,z+a)$ such that $g-\ee \leq \varphi_{1,\ee} \leq g \leq \varphi_{2,\ee} \leq g+\ee$. Consequently,
\begin{align*}
&\abs{\lim_{n\to +\infty} I_0(g) - \frac{2V(x,y)}{\sqrt{2\pi}\sigma^3} \int_{z}^{z+a} g(z') \bb E_{\bs \nu}^* \left( \psi\left( X_1^* \right) V^*\left( X_1^*, z'+S_1^* \right) \,;\, \tau_{z'}^* > 1 \right) \dd z'} \\
&\qquad\leq \frac{2V(x,y)}{\sqrt{2\pi}\sigma^3} \ee \int_{z}^{z+a} \bb E_{\bs \nu}^* \left( \psi\left( X_1^* \right) V^*\left( X_1^*, z'+S_1^* \right) \,;\, \tau_{z'}^* > 1 \right) \dd z'.
\end{align*}
Taking the limit as $\ee \to 0$, concludes the proof of the lemma.
\end{proof}

For any $l \geq 1$ we denote by $\scr C_b^+ \left( \bb X^l \times \bb R \right)$ the set of measurable non-negative functions $g$: $\bb X^l \times \bb R \to \bb R_+$ bounded and such that for any $(x_1,\dots,x_l) \in \bb X^l$, the function $z \mapsto g(x_1,\dots,x_l,z)$ is continuous.

\begin{lemma} Assume Hypotheses \ref{primitive}-\ref{nondegenere}.
\label{diner}
For any $x \in \bb X$, $y \in \bb R$, $z \geq 0$, $a >0$, $l \geq 1$, any non-negative functions $\psi$: $\bb X \to \bb R_+$ and $g\in \scr C_b^+ \left( \bb X^l \times \bb R \right)$, we have
\begin{align*}
&\lim_{n\to +\infty} n^{3/2} \bb E_x \left( g \left( X_1, \dots, X_l, y+S_n \right) \psi\left( X_n \right) \,;\, y+S_n \in [z,z+a) \,,\, \tau_y > n \right) \\
&\hspace{3cm} = \frac{2}{\sqrt{2\pi}\sigma^3} \int_z^{z+a} \bb E_x \left( g \left( X_1, \dots, X_l, z' \right) V\left( X_l, y+S_l \right) \,;\, \tau_y > l \right) \\
&\hspace{6cm} \times \bb E_{\bs \nu}^* \left( \psi\left( X_1^* \right) V^*\left( X_1^*, z'+S_1^* \right) \,;\, \tau_{z'}^* > 1 \right) \dd z'.
\end{align*}
\end{lemma}

\begin{proof}
We reduce the proof to the previous case using the Markov property. Fix $x \in \bb X$, $y \in \bb R$, $z \geq 0$, $a >0$, $l \geq 1$, $\psi$: $\bb X \to \bb R_+$ and $g \in \scr C_b^+ \left( \bb X^l \times \bb R \right)$. For any $n \geq l+1$, by the Markov property,
\begin{align*}
I_0 &:= n^{3/2} \bb E_x \left( g \left( X_1, \dots, X_l,y+S_n \right) \psi\left( X_n \right) \,;\, y+S_n \in [z,z+a) \,,\, \tau_y > n \right) \\
&= \bb E_x \left( n^{3/2} J_{n-l} \left( X_1, \dots, X_l, y+S_l \right) \,,\, \tau_y > l \right),
\end{align*}
where for any $(x_1,\dots,x_l) \in \bb X^l$, $y' \in \bb R$ and $k \geq 1$,
\[
J_k(x_1,\dots,x_l,y') = \bb E_{x_l} \left( g \left( x_1,\dots,x_l,y'+S_k \right) \psi\left( X_k \right) \,;\, y'+S_k \in [z,z+a) \,,\, \tau_{y'} > k \right).
\]
By the point \ref{LLTC002} of Theorem \ref{LLTC},
\[
n^{3/2} J_{n-l} \left( X_1, \dots, X_l, y+S_l \right) \leq c \norm{g}_{\infty} \norm{\psi}_{\infty} (1+z)\left( 1+\max \left( y+S_l,0 \right) \right).
\]
Consequently, by the Lebesgue dominated convergence theorem (in fact the expectation $\bb E_x$ is a finite sum) and Lemma \ref{sombrero},
\begin{align*}
\lim_{n\to+\infty} I_0 &= \frac{2}{\sqrt{2\pi}\sigma^3} \int_z^{z+a} \bb E_x \left( g \left( X_1, \dots, X_l,z' \right) V \left( X_l, y+S_l \right) \,;\, \tau_y > l \right) \\
&\hspace{3cm}\times \bb E_{\bs \nu}^* \left( \psi\left( X_1^* \right) V^*\left( X_1^*, z'+S_1^* \right) \,;\, \tau_{z'}^* > 1 \right) \dd z'.
\end{align*}
\end{proof}

Lemma \ref{diner} can be reformulated for the dual Markov walk as follows:

\begin{lemma} Assume Hypotheses \ref{primitive}-\ref{nondegenere}.
\label{animal}
For any $x' \in \bb X$, $z \geq 0$, $y' \geq 0$, $a >0$, $m \geq 1$ and any function $g \in \scr C_b^+ \left( \bb X^m \times \bb R \right)$, we have
\begin{align*}
&\lim_{n\to +\infty} n^{3/2} \bb E_{\bs \nu}^* \left( g \left( X_m^*, \dots, X_1^*,y'-S_n^* \right) \frac{\bbm 1_{\left\{ X_{n+1}^* = x' \right\}}}{\bs \nu \left( X_{n+1}^* \right) } \,;\, z+S_n^* \in [y',y'+a) \,,\, \tau_z^* > n \right) \\
&\quad = \frac{2}{\sqrt{2\pi}\sigma^3} \int_{y'}^{y'+a} \bb E_{\bs \nu}^* \left( g \left( X_m^*, \dots, X_1^*, y'-y''+z \right) V^* \left( X_m^*, z+S_m^* \right) \,;\, \tau_z^* > m \right) V\left( x', y'' \right) \dd y''.
\end{align*}
\end{lemma}

\begin{proof}
Fix $x' \in \bb X$, $z \geq 0$, $y' \geq 0$, $a >0$, $m \geq 1$ and $g \in \scr C_b^+ \left( \bb X^m \times \bb R \right)$. Let $\psi_{x'}^*$ be the function defined on $\bb X$ by \eqref{hirondelle} and consider for any $n \geq m+1$,
\[
I_0 := n^{3/2} \bb E_{\bs \nu}^* \left( g \left( X_m^*, \dots, X_1^*,y'-S_n^* \right) \psi_{x'}^* \left( X_n^* \right) \,;\, z+S_n^* \in [y',y'+a) \,,\, \tau_z^* > n \right).
\]
By Lemma \ref{diner} applied to the dual Markov walk, we have
\begin{align*}
I_0 &\underset{n\to+\infty}{\longrightarrow} \frac{2}{\sqrt{2\pi}\sigma^3} \sum_{x^* \in \bb X} \int_{y'}^{y'+a} \bb E_{x^*}^* \left( g \left( X_m^*, \dots, X_1^*, y'+z-y'' \right) V^* \left( X_m^*, z+S_m^* \right) \,;\, \tau_z^* > m \right) \bs \nu(x^*) \\
&\hspace{4cm} \times \bb E_{\bs \nu} \left( \psi_{x'}^* \left( X_1 \right) V\left( X_1, y''+S_1 \right) \,;\, \tau_{y''} > 1 \right) \dd y''.
\end{align*}
Moreover, using \eqref{hirondelle} and the fact that $\bs \nu$ is $\bf P$-invariant, for any $x' \in \bb X$, $y'' \geq 0$,
\begin{align*}
&\bb E_{\bs \nu} \left( \psi_{x'}^* \left( X_1 \right) V\left( X_1, y''+S_1 \right) \,;\, \tau_{y''} > 1 \right)  \\
&\qquad = \sum_{x_1 \in \bb X} \frac{\bf P(x',x_1)}{\bs \nu(x_1)} V\left( x_1, y''+f(x_1) \right) \bbm 1_{\{ y''+f(x_1) > 0 \}} \bs \nu(x_1) \\
&\qquad = \bb E_{x'} \left( V\left( X_1, y''+S_1 \right) \,;\, \tau_{y''} > 1 \right).
\end{align*}
By the point \ref{RESUI001} of Proposition \ref{RESUI}, the function $V$ is harmonic and so
\begin{align*}
\lim_{n\to+\infty} I_0 &= \frac{2}{\sqrt{2\pi}\sigma^3} \int_{y'}^{y'+a} \bb E_{\bs \nu}^* \left( g \left( X_m^*, \dots, X_1^*, y'-y''+z \right) V^* \left( X_m^*, z+S_m^* \right) \,;\, \tau_z^* > m \right) \\
&\hspace{10cm} \times V\left( x', y'' \right) \dd y''.
\end{align*}
\end{proof}

\begin{lemma} Assume Hypotheses \ref{primitive}-\ref{nondegenere}.
\label{pluie}
For any $x \in \bb X$, $y \in \bb R$, $z \geq 0$, $a >0$, $m \geq 1$ and any function $g \in \scr C_b^+ \left( \bb X^m \times \bb R \right)$, we have
\begin{align*}
&\lim_{n\to +\infty} n^{3/2} \bb E_x \left( g \left( X_{n-m+1}, \dots, X_n, y+S_n \right) \,;\, y+S_n \in (z,z+a] \,,\, \tau_y > n \right) \\
&\qquad = \frac{2 V(x,y)}{\sqrt{2\pi}\sigma^3} \int_z^{z+a} \bb E_{\bs \nu}^* \left( g \left( X_m^*, \dots, X_1^*, z' \right) V^* \left( X_m^*, z'+S_m^* \right) \,;\, \tau_{z'}^* > m \right) \dd z'.
\end{align*}
\end{lemma}

\begin{proof}
Fix $x \in \bb X$, $y \in \bb R$, $z \geq 0$, $a >0$, $m \geq 1$ and $g \in \scr C_b^+ \left( \bb X^m \times \bb R \right)$. For any $n \geq m$, consider
\begin{equation}
	\label{orchestre}
	I_n(x,y) := \bb E_x \left( g \left( X_{n-m+1}, \dots, X_n, y+S_n \right) \,;\, y+S_n \in (z,z+a] \,,\, \tau_y > n \right).
\end{equation}
For any $l \geq 1$ and $n \geq l+m$, by the Markov property, we have
\begin{equation}
\label{cobra}
n^{3/2} I_n(x,y) = \bb E_x \left( n^{3/2} I_{n-l} \left( X_l, y+S_l \right) \,;\, \tau_y > l \right).
\end{equation}
For any $p \geq 1$ and $0 \leq k \leq p$ we define $z_k := z+\frac{a k}{p}$. For any $x' \in \bb X$, $y' > 0$, $n \geq l+m$ and $p \geq 1$, we write
\begin{align*}
n^{3/2} I_{n-l}(x',y') &= \sum_{k=0}^{p-1} n^{3/2} \bb E_{x'} \left( g \left( X_{n-l-m+1}, \dots, X_{n-l}, y'+S_{n-l} \right) \,;\, \right. \\
&\hspace{6cm} \left. y'+S_{n-l} \in ( z_k,z_{k+1} ] \,,\, \tau_{y'} > n-l \right).
\end{align*}
Using Lemma \ref{duality}, we get
\begin{align*}
n^{3/2} I_{n-l}(x',y') &= \sum_{k=0}^{p-1} n^{3/2} \bb E_{\bs \nu}^* \left( g \left( X_m^*, \dots, X_1^*, y'-S_{n-l}^* \right) \psi_{x'}^* \left( X_{n-l}^* \right) \,;\, y'-S_{n-l}^* \in ( z_k,z_{k+1} ] \,,\, \right. \\
&\hspace{2cm} \left. \forall i \in \{1, \dots, n-l \}, \; y'+f\left( X_{n-l}^* \right) + \cdots + f\left( X_{n-l-i+1}^* \right) > 0 \right),
\end{align*}
where $\psi_{x'}^*$ is defined by \eqref{hirondelle}.

\textit{The upper bound.} Using \eqref{campagnol}, we have
\begin{align*}
n^{3/2} I_{n-l}(x',y') &\leq \sum_{k=0}^{p-1} n^{3/2} \bb E_{\bs \nu}^* \left( g \left( X_m^*, \dots, X_1^*, y'-S_{n-l}^* \right) \psi_{x'}^* \left( X_{n-l}^* \right) \,;\, \right. \\
&\hspace{4cm} \left. z_{k+1} + S_{n-l}^* \in \left[ y',y'+a/p \right) \,,\, \tau_{z_{k+1}}^* > n-l \right).
\end{align*}
By Lemma \ref{animal},
\[
\limsup_{n\to +\infty} n^{3/2} I_{n-l}(x',y') \leq \frac{2}{\sqrt{2\pi} \sigma^3} \sum_{k=0}^{p-1} \int_{y'}^{y'+a/p} J_k(y'-y'') V\left( x', y'' \right) \dd y'',
\]
where for any $k \geq 0$ and $t \in \bb R$,
\[
J_k(t) := \bb E_{\bs \nu}^* \left( g \left( X_m^*, \dots, X_1^*, t+z_{k+1} \right) V^* \left( X_m^*, z_{k+1}+S_m^* \right) \,;\, \tau_{z_{k+1}}^* > m \right).
\]
Note that for any $t \in [-a/p,0]$
\begin{equation}
\label{caribou001}
J_k(t) \leq \underbrace{\bb E_{\bs \nu}^* \left( \sup_{t \in [-a/p,0]} g \left( X_m^*, \dots, X_1^*, t+z_{k+1} \right) V^* \left( X_m^*, z_{k+1}+S_m^* \right) \,;\, \tau_{z_{k+1}}^* > m \right)}_{=:J_k^p}.
\end{equation}
Since $y'' \mapsto V\left( x', y'' \right)$ is non-decreasing (see the point \ref{RESUI002} of Proposition \ref{RESUI}), we have
\[
\limsup_{n\to +\infty} n^{3/2} I_{n-l}(x',y') \leq \frac{a}{p} \sum_{k=0}^{p-1}  \frac{2 J_k^p}{\sqrt{2\pi} \sigma^3} V\left( x', y'+\frac{a}{p} \right).
\]
Moreover, by \eqref{orchestre} and the point \ref{LLTC002} of Theorem \ref{LLTC},
\[
n^{3/2} I_{n-l}(X_l,y+S_l) \leq \norm{g}_{\infty} c \left( 1+z \right)\left( 1+\max(y+S_l,0) \right).
\]
Consequently, by \eqref{cobra} and the Lebesgue dominated convergence theorem (or using just the fact that $\bb X$ is finite),
\[
\limsup_{n\to +\infty} n^{3/2} I_n(x,y) \leq \frac{a}{p} \sum_{k=0}^{p-1} \frac{2 J_k^p}{\sqrt{2\pi}\sigma^3} \bb E_x \left( V\left( X_l, y+S_l+\frac{a}{p} \right) \,;\, \tau_y > l \right).
\]
Using the point \ref{RESUI003} of Proposition \ref{RESUI}, for any $\delta \in (0,1)$,
\[
\limsup_{n\to +\infty} n^{3/2} I_n(x,y) \leq \frac{a}{p} \sum_{k=0}^{p-1} \frac{2 J_k^p}{\sqrt{2\pi}\sigma^3} \bb E_x \left( \left( 1+\delta \right)\left( y+S_l+\frac{a}{p} \right) + c_{\delta} \,;\, \tau_y > l \right)
\]
and again using the point \ref{RESUI003} of Proposition \ref{RESUI}, for any $\delta \in (0,1)$,
\[
\limsup_{n\to +\infty} n^{3/2} I_n(x,y) \leq \frac{a}{p} \sum_{k=0}^{p-1} \frac{2 J_k^p}{\sqrt{2\pi}\sigma^3} \bb E_x \left( \frac{1+\delta}{1-\delta} V \left( X_l, y+S_l \right) +2\frac{a}{p} + c_{\delta} \,;\, \tau_y > l \right).
\]
Using the point \ref{RESUI001} of Proposition \ref{RESUI} and the point \ref{RESUII002} of Proposition \ref{RESUII} and taking the limit as $l \to +\infty$,
\[
\limsup_{n\to +\infty} n^{3/2} I_n(x,y) \leq \frac{a}{p} \sum_{k=0}^{p-1} \frac{2 J_k^p}{\sqrt{2\pi}\sigma^3} \frac{1+\delta}{1-\delta} V(x,y).
\]
Taking the limit as $\delta \to 0$,
\begin{equation}
\label{mimosa}
\limsup_{n\to +\infty} n^{3/2} I_n(x,y) \leq \frac{a}{p} \sum_{k=0}^{p-1} \frac{2 J_k^p}{\sqrt{2\pi}\sigma^3} V(x,y).
\end{equation}
For any $(x_1^*,\dots,x_m^*) \in \bb X^m$ and $u \in \bb R$, let
\begin{align}
\label{brouette}
g_m(u) &:= g \left( x_m^*, \dots, x_1^*, u \right), \nonumber\\
V_m^*(u) &:= V^* (x_m^*, u-f(x_1^*) - \dots- f(x_m^*)) \bbm 1_{\{u-f(x_1^*) > 0, \dots, u-f(x_1^*)-\dots-f(x_m^*) > 0\}}.
\end{align}
The function $u \mapsto g_m(u)$ is uniformly continuous on $[z,z+a]$. Consequently, for any $\ee > 0$, there exists $p_0 \geq 1$ such that for any $p \geq p_0$,
\[
\frac{a}{p} \sum_{k=0}^{p-1} \sup_{t \in [-a/p,0]} g_m \left( t+z_{k+1} \right) V_m^* (z_{k+1}) \leq  \frac{a}{p} \sum_{k=0}^{p-1} \left( g_m \left( z_{k+1} \right) + \ee \right) V_m^* (z_{k+1}).
\]
Moreover, using the point \ref{RESUI002} of Proposition \ref{RESUI}, it is easy to see that the function $u \mapsto V_m^*(u)$ is non-decreasing and so is Riemann-integrable. Therefore, as $p \to +\infty$, we have
\[
\limsup_{p\to +\infty} \frac{a}{p} \sum_{k=0}^{p-1} \sup_{t \in [-a/p,0]} g_m \left( t+z_{k+1} \right) V_m^* (z_{k+1}) \leq  \int_z^{z+a} \left( g_m \left( z' \right) + \ee \right) V_m^* (z') \dd z'.
\]
Thus, when $\ee \to 0$,
\begin{equation}
\label{caribou002}
\limsup_{p\to +\infty} \frac{a}{p} \sum_{k=0}^{p-1} \sup_{t \in [-a/p,0]} g_m \left( t+z_{k+1} \right) V_m^* (z_{k+1}) \leq  \int_z^{z+a} g_m \left( z' \right) V_m^* (z') \dd z'.
\end{equation}
Moreover, since $u \mapsto V_m^*(u)$ is non-decreasing,
\[
\frac{a}{p} \sum_{k=0}^{p-1} \sup_{t \in [-a/p,0]} g_m \left( t+z_{k+1} \right) V_m^* (z_{k+1}) \leq \norm{g}_{\infty} V_m^* (z+a) a.
\]
Consequently, by the Lebesgue dominated convergence theorem, \eqref{caribou001}, \eqref{caribou002} and the Fubini theorem,
\begin{align*}
&\limsup_{p\to +\infty} \frac{a}{p} \sum_{k=0}^{p-1} \frac{2 J_k^p}{\sqrt{2\pi}\sigma^3} V(x,y) \\
&\hspace{2cm}= \frac{2 V(x,y)}{\sqrt{2\pi}\sigma^3} \bb E_{\bs \nu}^* \left( \limsup_{p\to +\infty} \frac{a}{p} \sum_{k=0}^{p-1} \sup_{t \in [-a/p,0]} g \left( X_m^*, \dots, X_1^*, t+z_{k+1} \right)  \right. \\
&\hspace{10cm} \left. \times V^* \left( X_m^*, z_{k+1}+S_m^* \right) \,;\, \tau_{z_{k+1}}^* > m \right) \\
&\hspace{2cm}\leq \frac{2 V(x,y)}{\sqrt{2\pi}\sigma^3} \int_z^{z+a} \bb E_{\bs \nu}^* \left( g \left( X_m^*, \dots, X_1^*, z' \right) V^* \left( X_m^*, z'+S_m^* \right) \,;\, \tau_{z'}^* > m \right) \dd z'.
\end{align*}
By \eqref{mimosa}, we obtain that,
\begin{align*}
&\limsup_{n\to +\infty} n^{3/2} I_n(x,y) \\
&\qquad \leq \frac{2 V(x,y)}{\sqrt{2\pi}\sigma^3} \int_z^{z+a} \bb E_{\bs \nu}^* \left( g \left( X_m^*, \dots, X_1^*, z' \right) V^* \left( X_m^*, z'+S_m^* \right) \,;\, \tau_{z'}^* > m \right) \dd z'.
\end{align*}

\textit{The lower bound.} Repeating similar arguments as in the upper bound, by \eqref{mulot}, we have for any $x' \in \bb X$, $y' > 0$, $l \geq 1$, $n \geq l+m+1$, $p \geq 1$,
\begin{align*}
n^{3/2} I_{n-l}(x',y') &\geq \sum_{k=0}^{p-1} n^{3/2} \bb E_{\bs \nu}^* \left( g \left( X_m^*, \dots, X_1^*,y'-S_{n-l}^* \right) \psi_{x'}^* \left( X_{n-l}^* \right) \,;\, \right. \\
&\hspace{4cm} \left. z_k+S_{n-l}^* \in [ y'-a/p,y' ) \,,\, \tau_{z_k}^* > n-l \right) \\
&= \sum_{k=0}^{p-1} n^{3/2} \bb E_{\bs \nu}^* \left( g \left( X_m^*, \dots, X_1^*,y_+'+a'-S_{n-l}^* \right) \psi_{x'}^* \left( X_{n-l}^* \right) \,;\, \right. \\
&\hspace{4cm} \left. z_k+S_{n-l}^* \in [ y_+',y_+'+a' ) \,,\, \tau_{z_k}^* > n-l \right),
\end{align*}
where $y_+' = \max(y'-a/p,0)$ and $a' = \min(y',a/p) \in (0,a/p)$. Using Lemma \ref{animal},
\[
\liminf_{n\to +\infty} n^{3/2} I_{n-l}(x',y') \geq \sum_{k=0}^{p-1} \frac{2}{\sqrt{2\pi}\sigma^3} \int_{y_+'}^{y_+'+a'} L_k(y_+'+a'-y'') V\left( x', y'' \right) \dd y'',
\]
where, for any $t \in \bb R$,
\[
L_k(t) := \bb E_{\bs \nu}^* \left( g \left( X_m^*, \dots, X_1^*, t+z_k \right) V^* \left( X_m^*, z_k+S_m^* \right) \,;\, \tau_{z_k}^* > m \right).
\]
Since $y'' \mapsto V\left( x', y'' \right)$ is non-decreasing (see the point \ref{RESUI002} of Proposition \ref{RESUI}), we have
\[
\liminf_{n\to +\infty} n^{3/2} I_{n-l}(x',y') \geq a' \sum_{k=0}^{p-1}  \frac{2 L_k^p}{\sqrt{2\pi}\sigma^3} V\left( x', y_+' \right),
\]
where
\begin{equation}
\label{scarabee}
L_k^p := \bb E_{\bs \nu}^* \left( \inf_{t\in [0,a/p]} g \left( X_m^*, \dots, X_1^*, t+z_k \right) V^* \left( X_m^*, z_k+S_m^* \right) \,;\, \tau_{z_k}^* > m \right).
\end{equation}
Moreover, by the point \ref{RESUI003} of Proposition \ref{RESUI}, for any $\delta \in (0,1)$,
\[
a' V(x', y_+') \geq (1-\delta) a' y_+' - c_{\delta} \geq (1-\delta) \left( y' - \frac{a}{p} \right) \frac{a}{p} - c_{\delta} \geq \frac{a}{p} \frac{1-\delta}{1+\delta} V(x',y') - \frac{a}{p} c_{\delta}  - \left( \frac{a}{p} \right)^2 -c_{\delta}.
\]
Consequently, using \eqref{cobra} and the Fatou Lemma,
\begin{align*}
\liminf_{n\to +\infty} n^{3/2} I_n(x,y) &\geq \sum_{k=0}^{p-1} \frac{2 L_k^p}{\sqrt{2\pi}\sigma^3} \bb E_x \left( \frac{a}{p}\frac{1-\delta}{1+\delta} V \left( X_l, y+S_l \right) -c_{\delta}  \left( 1+a^2\right) \,;\, \tau_y > l \right).
\end{align*}
Using the point \ref{RESUI001} of Proposition \ref{RESUI} and the point \ref{RESUII002} of Proposition \ref{RESUII} and taking the limit as $l \to +\infty$ and then as $\delta \to 0$,
\begin{equation}
\label{mille-pattes}
\liminf_{n\to +\infty} n^{3/2} I_n(x,y) \geq \frac{a}{p} \sum_{k=0}^{p-1} \frac{2 L_k^p}{\sqrt{2\pi}\sigma^3} V(x,y).
\end{equation}
Using the notation from \eqref{brouette} and the fact that $u \mapsto g_m(u)$ is uniformly continuous on $[z,z+a]$, for any $\ee > 0$,
\[
\liminf_{p\to +\infty} \frac{a}{p} \sum_{k=0}^{p-1} \inf_{t \in [0,a/p]} g_m \left( t+z_k \right) V_m^* (z_k) \geq  \int_z^{z+a} \left( g_m \left( z' \right) - \ee \right) V_m^* (z') \dd z'.
\]
Taking the limit as $\ee \to 0$,
\[
\liminf_{p\to +\infty} \frac{a}{p} \sum_{k=0}^{p-1} \inf_{t \in [0,a/p]} g_m \left( t+z_k \right) V_m^* (z_k) \geq  \int_z^{z+a} g_m \left( z' \right) V_m^* (z') \dd z'.
\]
By the Fatou lemma, \eqref{scarabee} and \eqref{mille-pattes}, we conclude that
\begin{align*}
&\liminf_{n\to +\infty} n^{3/2} I_n(x,y) \geq \frac{2V(x,y)}{\sqrt{2\pi}\sigma^3} \bb E_{\bs \nu}^* \left( \liminf_{p\to +\infty} \frac{a}{p} \sum_{k=0}^{p-1} \inf_{t\in [0,a/p]} g \left( X_m^*, \dots, X_1^*, t+z_k \right)  \right. \\
&\hspace{8cm} \left. \phantom{\sum_{k=0}^{p-1}} \times V^* \left( X_m^*, z_k+S_m^* \right) \,;\, \tau_{z_k}^* > m \right) \\
&\hspace{1cm} \geq \frac{2V(x,y)}{\sqrt{2\pi}\sigma^3} \int_z^{z+a} \bb E_{\bs \nu}^* \left( g \left( X_m^*, \dots, X_1^*, z' \right) V^* \left( X_m^*, z'+S_m^* \right) \,;\, \tau_{z'}^* > m \right) \dd z'.
\end{align*}
\end{proof}

From now on, we consider that the dual Markov chain $\left( X_n^* \right)_{n\geq 0}$ is \textit{independent} of $\left( X_n \right)_{n\geq 0}$. Recall that its transition probability $\bf P^*$ is defined by \eqref{statue}
and that, for any $z \geq 0$, the associated Markov walk $( z+S_n^* )_{n\geq 0}$ and the associated exit time $\tau_z^*$ are defined by \eqref{bataille} and \eqref{bataillebis} respectively. Recall also that for any $(x,x^*) \in \bb X^2$, we denote by $\bb P_{x,x^*}$ and $\bb E_{x,x^*}$ the probability and the expectation generated by the finite dimensional distributions of the Markov chains $( X_n )_{n\geq 0}$ and $( X_n^* )_{n\geq 0}$ starting at $X_0 = x$ and $X_0^* = x^*$ respectively.

\begin{lemma} Assume Hypotheses \ref{primitive}-\ref{nondegenere}.
\label{capella}
For any $x \in \bb X$, $y \in \bb R$, $z \geq 0$, $a >0$, $l \geq 1$, $m \geq 1$ and any function $g \in \scr C_b^+ \left( \bb X^{l+m} \times \bb R \right)$, we have
\begin{align*}
&\lim_{n\to +\infty} n^{3/2} \bb E_x \left( g \left(X_1, \dots, X_l, X_{n-m+1}, \dots, X_n, y+S_n \right) \,;\, y+S_n \in (z,z+a] \,,\, \tau_y > n \right) \\
&\qquad= \frac{2}{\sqrt{2\pi}\sigma^3} \int_{z}^{z+a} \sum_{x^* \in \bb X} \bb E_{x,x^*} \left( g \left( X_1, \dots, X_l, X_m^*, \dots, X_1^*, z' \right) \right. \\
&\hspace{5cm} \left. \times V\left( X_l, y+S_l \right) V^* \left( X_m^*, z'+S_m^* \right)  \,;\, \tau_y > l \,,\, \tau_{z'}^* > m \right) \dd z' \bs \nu(x^*).
\end{align*}
\end{lemma}

\begin{proof}
Fix $x \in \bb X$, $y \in \bb R$, $z \geq 0$, $a >0$, $l \geq 1$, $m \geq 1$ and $g \in \scr C_b^+ \left( \bb X^{l+m} \times \bb R \right)$. For any $n \geq l+m$, by the Markov property,
\begin{align*}
I_0 &:= n^{3/2} \bb E_x \left( g \left(X_1, \dots, X_l, X_{n-m+1}, \dots, X_n, y+S_n \right) \,;\, y+S_n \in (z,z+a] \,,\, \tau_y > n \right) \\
&= \sum_{x_1, \dots, x_l \in \bb X^l} n^{3/2} \bb E_{x_l} \left( g \left(x_1, \dots, x_l, X_{n-l-m+1}, \dots, X_{n-l}, y_l+S_{n-l} \right) \,;\,  \right. \\
&\hspace{3cm} \left. y_l+S_{n-l} \in (z,z+a] \,,\, \tau_{y_l} > n-l \right) \times  \bb P_x \left( X_1 = x_1, \dots, X_l = x_l, \tau_y > l \right),
\end{align*}
where $y_l = x_1+\dots+x_l$. Using the Lebesgue dominated convergence theorem (or simply the fact that $\bb X^l$ is finite) and Lemma \ref{pluie}, we conclude that
\begin{align*}
\lim_{n\to +\infty} I_0 &= \frac{2}{\sqrt{2\pi}\sigma^3} \sum_{x_1, \dots, x_l \in \bb X^l} V\left( x_l, y_l \right) \bb P_x \left( X_1 = x_1, \dots, X_l = x_l, \tau_y > l \right) \\
& \qquad  \times \int_{z}^{z+a} \bb E_{\bs \nu}^* \left( g \left(x_1,\dots,x_l, X_m^*, \dots, X_1^*,z' \right) V^* \left( X_m^*, z'+S_m^* \right) \,;\, \tau_{z'}^* > m \right) \dd z'.
\end{align*}
\end{proof}

\subsection{Proof of Theorem \ref{CAPE}.}

For any $l \geq 1$, denote by $\scr C^+ ( \bb X^l \times \bb R_+ )$ the set of non-negative functions $g$: $\bb X^l \times \bb R_+ \to \bb R_+$ 
satisfying the following properties:
\begin{itemize}
\item for any $(x_1,\dots,x_l) \in \bb X^l$, the function $z \mapsto g(x_1,\dots,x_l,z)$ is continuous,
\item there exists $\ee > 0$ such that $\max_{x_1,\dots x_l \in \bb X} \sup_{z \geq 0} g(x_1,\dots,x_l,z) (1+z)^{2+\ee} < +\infty$.
\end{itemize}

Fix $x \in \bb X$, $y \in \bb R$, $l \geq 1$, $m \geq 1$ and $g \in \scr C^+ \left( \bb X^{l+m} \times \bb R \right)$. For brevity, denote
\[
g_{l,m}(y+S_n) = g \left(X_1, \dots, X_l, X_{n-m+1}, \dots, X_n, y+S_n \right).
\]
Set
\begin{align*}
I_0 &:= n^{3/2} \bb E_x \left( g_{l,m}(y+S_n) \,;\, \tau_y > n \right) \\
&= \sum_{k=0}^{+\infty} \underbrace{n^{3/2} \bb E_x \left( g_{l,m}(y+S_n) \,;\, y+S_n \in (k,k+1] \,,\, \tau_y > n \right)}_{=:I_k(n)}.
\end{align*}
Since $g \in \scr C^+ \left( \bb X^{l+m} \times \bb R \right)$, we have
\[
I_k(n) \leq \frac{N(g)}{(1+k)^{2+\ee}} n^{3/2} \bb P_x \left( y+S_n \in (k,k+1] \,,\, \tau_y > n \right),
\]
where $N(g) = \max_{x_1,\dots,x_{l+m} \in \bb X} \sup_{z \geq 0}  g(x_1,\dots,x_{l+m},z)(1+z)^{2+\ee} <+\infty$. By the point \ref{LLTC002} of Theorem \ref{LLTC}, we have
\[
I_k(n) \leq \frac{c N(g) (1+\max(y,0))}{(k+1)^{1+\ee}}.
\]
Consequently, by the Lebesgue dominated convergence theorem,
\[
\lim_{n\to+\infty} I_0 = \sum_{k=0}^{+\infty} \lim_{n\to+\infty} n^{3/2} \bb E_x \left( g_{l,m}(y+S_n) \,;\, y+S_n \in (k,k+1] \,,\, \tau_y > n \right).
\]
By Lemma \ref{capella},
\begin{align*}
\lim_{n\to+\infty} I_0 &= \frac{2}{\sqrt{2\pi}\sigma^3} \sum_{k=0}^{+\infty} \int_k^{k+1} \sum_{x^* \in \bb X} \bb E_{x,x^*} \left( g \left( X_1, \dots, X_l, X_m^*, \dots, X_1^*,z' \right) V\left( X_l, y+S_l \right)  \right. \\
&\hspace{6cm} \left. \times V^* \left( X_m^*, z'+S_m^* \right)  \,;\, \tau_y > l \,,\, \tau_{z'}^* > m \right) \dd z' \bs \nu(x^*),
\end{align*}
which establishes Theorem \ref{CAPE}.

\subsection{Proof of Theorem \ref{CAPEBIS}.}

Theorem \ref{CAPEBIS} will be deduced from Theorem \ref{CAPE}.

Let $x \in \bb X$, $y \in \bb R$ and $n \geq 1$. Since $\bb X$ is finite we note that $\norm{f}_{\infty} = \sup_{x \in \bb X} \abs{f(x)}$ exists. This implies
\[
\bb P_x \left( \tau_y = n+1 \right) = \bb P_x \left( y+S_n+f(X_{n+1}) \leq 0 \,,\, y+S_n \in \left[ 0,\norm{f}_{\infty} \right] \,,\, \tau_y > n \right).
\]
By the Markov property,
\[
\bb P_x \left( \tau_y = n+1 \right) = \bb E_x \left( g(X_n,y+S_n) \,;\, \tau_y > n \right),
\]
where, for any $(x',y') \in \bb X \times \bb R$,
\[
g(x',y') = \bb P_{x'} \left( y'+f(X_1) \leq 0 \right) \bbm 1_{\left\{y' \in \left[ 0,\norm{f}_{\infty} \right] \right\}} = \bbm 1_{\left\{y' \in \left[ 0,\norm{f}_{\infty} \right] \right\}} \sum_{x_1\in \bb X} \bf P(x',x_1) \bbm 1_{\{y'+f(x_1) \leq 0 \}}.
\]
Since $g(x',\cdot)$ is a staircase function, for any $\ee > 0$ there exist two functions $\varphi_{\ee}$ and $\psi_{\ee}$ on $\bb X \times \bb R$ and $N \subset \bb X \times \bb R$ such that
\begin{itemize}
\item for any $x' \in \bb X$, the functions $\varphi_{\ee}(x',\cdot)$ and $\psi_{\ee}(x',\cdot)$ are continuous and have a compact support included in $\left[ -1,\norm{f}_{\infty}+1 \right]$,
\item for any $(x',y') \in \left( \bb X \times \bb R \right) \setminus N$, it holds $\varphi_{\ee}(x',y') = g(x',y') = \psi_{\ee}(x',y')$,
\item for any $(x',y') \in \bb X \times \bb R$, it holds $0 \leq \varphi_{\ee}(x',y') \leq g(x',y') \leq \psi_{\ee}(x',y') \leq 1$,
\item the set $N$ is sufficiently small:
\begin{equation}
	\label{cachalot}
	\int_{-1}^{\norm{f}_{\infty}+1} \bb E_{\bs \nu}^* \left( V^*\left( X_1, z+S_1^* \right) \,;\, \tau_z^* > 1 \,,\, \left( X_1, z \right) \in N \right) \dd z \leq \ee.
\end{equation}
\end{itemize}

\textit{The upper bound.} For any $\ee > 0$, using Theorem \ref{CAPE}, we have
\begin{align*}
	I^+ &:= \limsup_{n\to+\infty} n^{3/2}\bb P_x \left( \tau_y = n+1 \right) \\
	&\leq \limsup_{n\to+\infty} n^{3/2}\bb E_x \left( \psi_{\ee}(X_n,y+S_n) \,;\, \tau_y > n \right) \\
	&= \frac{2}{\sqrt{2\pi} \sigma^3} \int_0^{+\infty} \sum_{x^* \in \bb X} \bb E_{x,x^*} \left( \psi_{\ee} \left( X_1^*,z \right) V(X_1,y+S_1) \right. \\
	&\hspace{7cm} \left.V^*(X_1^*,z+S_1^*) \,;\, \tau_y > 1 \,,\, \tau_z^* > 1 \right) \bs \nu(x^*) \dd z.
\end{align*}
Using the point \ref{RESUI001} of Proposition \ref{RESUI},
\begin{align}
	I^+ &\leq \frac{2V(x,y)}{\sqrt{2\pi} \sigma^3} \int_0^{\norm{f}_{\infty}+1} \bb E_{\bs \nu}^* \left( \psi_{\ee} \left( X_1^*,z \right) V^*(X_1^*,z+S_1^*) \,;\, \tau_z^* > 1 \right) \dd z \nonumber\\
	&\leq \underbrace{\frac{2V(x,y)}{\sqrt{2\pi} \sigma^3} \int_0^{\norm{f}_{\infty}} \bb E_{\bs \nu}^* \left( g \left( X_1^*,z \right) V^*(X_1^*,z+S_1^*) \,;\, \tau_z^* > 1 \right) \dd z}_{=:I_1} \nonumber\\
	&\qquad + \underbrace{\frac{2V(x,y)}{\sqrt{2\pi} \sigma^3} \int_0^{\norm{f}_{\infty}+1} \bb E_{\bs \nu}^* \left( V^*(X_1^*,z+S_1^*) \,;\, \tau_z^* > 1 \,,\, \left( X_1^*,z \right) \in N \right) \dd z}_{=:I_2}.
	\label{cachalot00}
\end{align}
Since $\bs \nu$ is $\bf P^*$-invariant, we have
\begin{align*}
I_1 &= \frac{2V(x,y)}{\sqrt{2\pi} \sigma^3} \int_0^{\norm{f}_{\infty}} \sum_{x^* \in \bb X} g \left( x^*,z \right) V^*(x^*,z-f(x^*)) \bbm 1_{\left\{ z-f(x^*) > 0 \right\}} \bs \nu(x^*) \dd z \\
	&= \frac{2V(x,y)}{\sqrt{2\pi} \sigma^3} \int_0^{\norm{f}_{\infty}} \sum_{x^*,x_1 \in \bb X} \bbm 1_{\left\{ z+f(x_1) \leq 0 \right\}} \bf P(x^*,x_1)\bs \nu(x^*) V^*(x^*,z-f(x^*)) \bbm 1_{\left\{ z-f(x^*) > 0 \right\}} \dd z \\
	&= \frac{2V(x,y)}{\sqrt{2\pi} \sigma^3} \int_0^{\norm{f}_{\infty}} \sum_{x^*,x_1 \in \bb X} \bbm 1_{\left\{ z+f(x_1) \leq 0 \right\}} \bf P^*(x_1,x^*)\bs \nu(x_1) V^*(x^*,z-f(x^*)) \bbm 1_{\left\{ z-f(x^*) > 0 \right\}} \dd z \\
	&= \frac{2V(x,y)}{\sqrt{2\pi} \sigma^3} \int_0^{\norm{f}_{\infty}} \sum_{x_1 \in \bb X} \bbm 1_{\left\{ z+f(x_1) \leq 0 \right\}} \bs \nu(x_1) \bb E_{x_1}^* \left( V^*(X_1^*,z+S_1^*) \,;\, \tau_z^* > 1 \right) \dd z.
\end{align*}
Using the point \ref{RESUI001} of Proposition \ref{RESUI},
\begin{equation}
\label{cachalot01}
I_1 = \frac{2V(x,y)}{\sqrt{2\pi} \sigma^3} \int_0^{\norm{f}_{\infty}} \bb E_{\bs \nu}^* \left( V^*(X_1^*,z) \,;\, S_1^* \geq z \right) \dd z.
\end{equation}
Moreover, by \eqref{cachalot}, we get
\begin{equation}
\label{cachalot02}
I_2 \leq \frac{2V(x,y)}{\sqrt{2\pi} \sigma^3} \ee.
\end{equation}
Putting together \eqref{cachalot00}, \eqref{cachalot01} and \eqref{cachalot02} and taking the limit as $\ee \to 0$, we obtain that
\begin{equation}
	\label{vallon}
	I^+ \leq \frac{2V(x,y)}{\sqrt{2\pi} \sigma^3} \int_0^{\norm{f}_{\infty}} \bb E_{\bs \nu}^* \left( V^*(X_1^*,z) \,;\, S_1^* \geq z \right) \dd z.
\end{equation}

\textit{Lower bound.} In a similar way, using Theorem \ref{CAPE}, we write
\begin{align*}
	I^- &:= \liminf_{n\to+\infty} n^{3/2}\bb P_x \left( \tau_y = n+1 \right) \\
	&\geq \liminf_{n\to+\infty} n^{3/2}\bb E_x \left( \varphi_{\ee}(X_n,y+S_n) \,;\, \tau_y > n \right) \\
	&= \frac{2V(x,y)}{\sqrt{2\pi} \sigma^3} \int_0^{\norm{f}_{\infty}+1} \bb E_{\bs \nu}^* \left( \varphi_{\ee} \left( X_1^*,z \right) V^*(X_1^*,z+S_1^*) \,;\, \tau_z^* > 1 \right) \dd z \\
	&\geq I_1 - I_2.
\end{align*}
Using \eqref{cachalot01} and \eqref{cachalot02} and taking the limit as $\ee \to 0$, we obtain that
\[
I^- \geq \frac{2V(x,y)}{\sqrt{2\pi} \sigma^3} \int_0^{\norm{f}_{\infty}} \bb E_{\bs \nu}^* \left( V^*(X_1^*,z) \,;\, S_1^* \geq z \right) \dd z,
\]
which together with \eqref{vallon} concludes the proof.

\section{Appendix}
\label{appendix}

\subsection{The non degeneracy of the Markov walk}
\label{AppA}

In \cite{grama_limit_2016-1}, it is proved that the statements of Propositions \ref{RESUI}-\ref{RESUIII} hold under more general assumptions (see Hypotheses M1-M5 of \cite{grama_limit_2016-1}). We will link these assumptions to our Hypotheses \ref{primitive}-\ref{nondegenere}. The assumptions M1-M3 in \cite{grama_limit_2016-1}, with the Banach space $\scr C$, are well known consequences of Hypothesis \ref{primitive} of this paper. Hypothesis M4 in \cite{grama_limit_2016-1} is also obvious with $N=N_1 = \dots = 0$. By Hypothesis \ref{nucentre}, 
to obtain Hypothesis M5 of \cite{grama_limit_2016-1}, it remains only to prove that $\sigma$ defined by \eqref{mu-sigma001LLT} is strictly positive. First we give a necessary and sufficient condition. Recall that the words \textit{path} and \textit{orbit} are defined in Section \ref{OPPET}.

\begin{lemma}
\label{sabot}
Assume Hypothesis \ref{primitive}. The following statements are equivalent:
\begin{enumerate}[ref=\arabic*, leftmargin=*, label=\arabic*.]
\item \label{sabot001} The Ces\'aro mean of $f$ on the orbits is constant: there exists $m \in \bb R$ such that for any orbit $x_0,\dots,x_n$ we have
\[
f(x_0) + \cdots + f(x_n) = (n+1)m.
\]
\item \label{sabot002} There exist a constant $m \in \bb R$ and a function $h \in \scr C$ such that for any $(x,x') \in \bb X^2$,
\[
\bf P(x,x') f(x') = \bf P(x,x') \left( h(x)-h(x')+m \right).
\]
\item \label{sabot003} The following real $\tilde \sigma^2$ is equal to $0$
\[
\tilde \sigma^2 = \bs \nu \left( f^2 \right) -  \bs \nu \left( f \right)^2 + 2 \sum_{n=1}^{+\infty} \left[ \bs \nu \left( f \bf P^n f \right) -  \bs \nu \left( f \right)^2 \right] = 0.
\]
\end{enumerate}
\end{lemma}

\begin{proof}
\textit{The point \ref{sabot001} implies the point \ref{sabot002}.} Suppose that the point \ref{sabot001} holds. Fix $x_0 \in \bb X$ and set $h(x_0)= 0$. For any $x \in \bb X$, we define $h(x)$ in the following way: for any path $x_0,x_1,\dots,x_n,x$ in $\bb X$, we set
\[
h(x) = -f(x)-f(x_n)-\dots-f(x_1)+(n+1)m.
\]
We shall verify that $h$ is well defined. By Hypothesis \ref{primitive}, we can find at least a path to define $h(x)$. Now we have to check that this definition does not depend on the choice of the path. Let $x_0,x_1,\dots,x_p,x$ and $x_0,y_1,\dots,y_q,x$ be two paths. By Hypothesis \ref{primitive}, there exists a path $x,z_1, \dots, z_n,x_0$ in $\bb X$ between $x$ and $x_0$. Since $x_0,x_1,\dots,x_p,x,z_1,\dots,z_n$ and $x_0,y_1,\dots,y_p,x,z_1,\dots,z_n$ are two orbits, by the point \ref{sabot001}, we have
\begin{align*}
-f(x)-f(x_p)-\dots-f(x_1)+(p+1)m &= f(x_0)+f(z_1)+\dots+f(z_n)-(n+1)m \\
&= -f(x)-f(y_q)-\dots-f(y_1)+(q+1)m
\end{align*}
and so the function $h$ is well defined on $\bb X$. Now let $(x,x') \in \bb X^2$ such that $\bf P(x,x') > 0$. By Hypothesis \ref{primitive}, there exists $x_0,x_1, \dots, x_n,x$ a path between $x_0$ and $x$. Since $\bf P(x_0,x_1) \cdots \bf P(x_n,x)\bf P(x,x') > 0$, by the definition of $h$, we have
\begin{align*}
h(x) &= -f(x)-f(x_n)-\dots-f(x_1)+(n+1)m \\
h(x') &= -f(x')-f(x)-f(x_n)-\dots-f(x_1)+(n+2)m.
\end{align*}
In particular
\[
h(x') = -f(x') + h(x) + m.
\]

\textit{The point \ref{sabot002} implies the point \ref{sabot001}.} Suppose that the point \ref{sabot002} holds and let $x_0,\dots,x_n$ be an orbit. Using the point \ref{sabot002},
\[
h(x_0) = h(x_n)-f(x_0)+m = \dots = h(x_0)-f(x_0)-f(x_n)-\dots-f(x_1)+(n+1)m,
\]
and the point \ref{sabot001} follows.

\textit{The point \ref{sabot002} implies the point \ref{sabot003}.} Suppose that the point \ref{sabot002} holds. Denote by $\tilde f$ the $\bs \nu$-centred function:
\begin{equation}
	\label{amadou}
	\tilde f(x) = f(x) - \bs \nu(f), \qquad \forall x \in \bb X.
\end{equation}
By the point \ref{sabot002}, for any $x\in \bb X$,
\begin{equation}
	\label{fleuve}
	\bf P \tilde f(x) = h(x) - \bf P h(x) + m -\bs \nu(f).
\end{equation}
Using the fact that $\bs \nu$ is $\bf P$-invariant, we obtain that $\bs \nu \left( \tilde f \right) = 0 = m-\bs \nu(f)$ and so,
\begin{equation}
	\label{sabre}	
	m=\bs \nu(f).
\end{equation}
Consequently, by \eqref{fleuve}, $\bf P^n \tilde f = \bf P^{n-1} h - \bf P^n h$ for any $n \geq 1$ and therefore,
\begin{equation}
	\label{ruche}
	\sum_{k=1}^{n} \bf P^{k} \tilde f = h - \bf P^n h.
\end{equation}
Let
\[
\tilde \Theta := \sum_{k=0}^{+\infty} \bf P^{k} \tilde f
\]
be the solution of the Poisson equation $\tilde \Theta - \bf P \tilde \Theta = \tilde f$, which by \eqref{decexpNLLT}, is well defined. Taking the limit as $n \to +\infty$ in \eqref{ruche} and using \eqref{decexpNLLT},
\[
	\bf P \tilde \Theta = \tilde \Theta - \tilde f = h - \bs \nu (h).
\]
Therefore, for any $(x,x') \in \bb X^2$,
\[
	\tilde \Theta(x') - \bf P \tilde \Theta(x) = \tilde f(x') + \bf P \tilde \Theta(x') - \bf P \tilde \Theta(x) = \tilde f(x') + h(x') - h(x).
\]
Using the point \ref{sabot002} and \eqref{sabre}, it follows that 
\begin{equation}
	\label{biniou}
	\tilde \Theta(x') - \bf P \tilde \Theta(x) = 0,
\end{equation}
for any $(x,x') \in \bb X^2$ such that $\bf P(x,x') > 0$. Moreover,
\[
	\tilde \sigma^2 = \bs \nu \left( \tilde f^2 \right) + 2 \sum_{n=1}^{+\infty} \bs \nu \left( \tilde f \bf P^n \tilde f \right) = \bs \nu \left( \tilde f \left( \tilde f + 2 \bf P \tilde \Theta \right) \right) = \bs \nu \left( \left( \tilde \Theta - \bf P \tilde \Theta \right) \left( \tilde \Theta + \bf P \tilde \Theta \right) \right).
\]
Since $\bs \nu$ is $\bf P$-invariant,
\begin{align}
	\tilde \sigma^2 &= \bs \nu \left( \bf P \left( \tilde \Theta^2 \right) \right) - 2\bs \nu \left( \left( \bf P \tilde \Theta \right)^2 \right) + \bs \nu \left( \left( \bf P \tilde \Theta \right)^2 \right) \nonumber\\
	&=  \sum_{(x,x') \in \bb X} \left[ \tilde \Theta(x')^2 -2 \tilde \Theta(x') \bf P \tilde \Theta(x) + \left( \bf P \tilde \Theta(x) \right)^2  \right] \bf P(x,x') \bs \nu(x) \nonumber\\
	&= \sum_{(x,x') \in \bb X} \left( \tilde \Theta(x') - \bf P \tilde \Theta(x) \right)^2  \bf P(x,x') \bs \nu(x).
	\label{kermesse}
\end{align}
By \eqref{biniou}, we conclude that $\tilde \sigma^2=0$.

\textit{The point \ref{sabot003} implies the point \ref{sabot002}.} Suppose that the point \ref{sabot003} holds. By \eqref{kermesse}, for any $(x,x') \in \bb X$ such that $\bf P(x,x')>0$ we have
\[
\tilde \Theta(x') - \bf P \tilde \Theta(x) = 0.
\]
Let $h = \bf P \tilde \Theta$. Since $\tilde \Theta$ is the solution of the Poisson equation,
\[
\tilde f(x') + h(x') - h(x) = 0.
\]
By the definition of $\tilde f$ in \eqref{amadou}, for any $(x,x') \in \bb X$ such that $\bf P(x,x')>0$,
\[
f(x') = h(x) - h(x') + m,
\]
with $m = \bs \nu(f)$.
\end{proof}

Note that under Hypothesis \ref{nucentre}, Lemma \ref{sabot} can be rewritten as follows.
\begin{lemma}
\label{sabotbis}
Assume Hypotheses \ref{primitive} and \ref{nucentre}. The following statements are equivalent:
\begin{enumerate}[ref=\arabic*, leftmargin=*, label=\arabic*.]
\item The mean of $f$ on the orbits is equal to zero: for any orbit $x_0,\dots,x_n$, we have
\[
f(x_0) + \cdots + f(x_n) = 0.
\]
\item There exists a function $h \in \scr C$ such that for any $(x,x') \in \bb X^2$,
\[
\bf P(x,x') f(x') = \bf P(x,x') \left( h(x)-h(x') \right).
\]
\item The real $\sigma^2$ is equal to $0$:
\[
\sigma^2 = \bs \nu \left( f^2 \right) + 2 \sum_{n=1}^{+\infty} \bs \nu \left( f \bf P^n f \right) = 0.
\]
\end{enumerate}
\end{lemma}

Now we prove that the Hypothesis \ref{nondegenere} (the "non-lattice" condition), implies that the Markov walk has non-zero asymptotic variance.
\begin{lemma}
\label{couette}
Under Hypotheses \ref{primitive}-\ref{nondegenere}, we have
\[
\sigma^2 = \bs \nu \left( f^2 \right) + 2 \sum_{n=1}^{+\infty} \bs \nu \left( f \bf P^n f \right) > 0
\]
\end{lemma}

\begin{proof}
We proceed by \textit{reductio ad absurdum}. Suppose that $\sigma^2 = 0$. By Lemma \ref{sabotbis}, for any orbit $x_0,\dots,x_n$, we have
\[
f(x_0) + \cdots + f(x_n) = 0,
\]
which implies the negation of Hypothesis \ref{nondegenere} with $\theta = a = 0$.
\end{proof}

\subsection{Strong approximation}

Let $(B_t)_{t\geq 0}$ be the standard Brownian motion on $\bb R$  
defined on the probability space $(\Omega, \scr F, \bb P)$. Consider the exit time
\begin{equation}
	\label{tauybm}
	\tau_y^{bm} := \inf \{ t \geq 0, \; y+\sigma B_t \leq 0 \},
\end{equation}
where $\sigma$ is defined by \eqref{mu-sigma001LLT}. It is proved in Grama, Le Page and Peign\'e \cite{ion_grama_rate_2014} that there is a version of the Markov walk $(S_n)_{n\geq 0}$ and of the standard Brownian motion $(B_t)_{t\geq 0}$ living on the same 
probability space which are close enough in the following sense: 

\begin{proposition}
\label{majdeA_kLLT}
There exists $\ee_0 >0$ such that, for any $\ee \in (0,\ee_0]$, $x\in \bb X$ and $n\geq 1$, without loss of generality (on an extension of the initial probability space) one can reconstruct the sequence $(S_n)_{n\geq 0}$ with a continuous time Brownian motion $(B_t)_{t\in \bb R_{+} }$, such that  
\[
\bb P_x \left( \underset{0 \leq t \leq 1}{\sup} \abs{S_{\pent{tn}}-\sigma B_{tn}} > n^{1/2-\ee} \right) \leq \frac{c_{\ee}}{n^{\ee}}.
\]
\end{proposition}

\bibliographystyle{plain}
\bibliography{biblioT}
\vskip0.5cm

\end{document}